\documentclass[11pt,a4paper]{article}
\ifdefined\pdfminorversion\pdfminorversion=7\fi
\usepackage[T1]{fontenc}
\usepackage[utf8]{inputenc}
\usepackage{lmodern,amsmath,amssymb,amsthm,mathtools}
\usepackage[margin=27mm]{geometry}
\usepackage{microtype}
\usepackage{graphicx,booktabs,array}
\usepackage[font=small,labelfont=bf]{caption}
\usepackage[hidelinks]{hyperref}
\hypersetup{pdftitle={The P\'olya conjecture for Neumann eigenvalues of spherical shells},pdfauthor={Yutian Li}}
\numberwithin{equation}{section}
\newtheorem{theorem}{Theorem}[section]
\newtheorem{lemma}[theorem]{Lemma}
\newtheorem{proposition}[theorem]{Proposition}
\newtheorem{corollary}[theorem]{Corollary}
\theoremstyle{remark}
\newtheorem{remark}[theorem]{Remark}

\newcommand{\W}{\mathcal W}

\allowdisplaybreaks[1]
\title{P\'olya's conjecture for Neumann eigenvalues of spherical shells}
\author{Yutian Li\\[3pt]
  \small School of Mathematics and Statistics\\
  \small Nanfang College, Guangzhou 510970, Guangdong, China\\
  \small\href{mailto:yutianlee@gmail.com}{\texttt{yutianlee@gmail.com}}}
\date{22 October 2026}
\begin{document}
\maketitle
\begin{abstract}
We prove the strict P\'olya inequality for Neumann eigenvalues of spherical
shells in every dimension $d\ge2$, for every radius ratio and every positive
energy. The proof also covers balls. For shells, the excess over the leading
Weyl term has a positive lower bound depending only on dimension. Above
explicit frequency thresholds independent of the radius ratio, we obtain
a surplus of boundary order. These estimates remain uniform as the hole
shrinks or the shell becomes thin. We compare weighted sampling errors
for inverse radial actions while retaining the removed volume exactly.
The proof is entirely analytic and requires no computer-assisted
verification.
\end{abstract}


\begingroup\small
\makeatletter
\renewcommand*\l@section[2]{%
  \@dottedtocline{1}{0em}{1.5em}{\bfseries #1}{\bfseries #2}}
\makeatother
\tableofcontents
\endgroup

\section{Introduction and main results}\label{sec:introduction}

P\'olya's conjecture asks whether the leading term in Weyl's law is an
upper bound for the Dirichlet counting function and a lower bound for
the Neumann counting function at every positive energy
\cite{Polya1954,Polya1961,Weyl1912}. We prove the strict Neumann inequality
for spherical shells in every dimension $d\ge2$ and for every radius
ratio. The estimates also include zero inner radius, giving an analytic
proof of the strict ball inequality without computer-assisted estimates.
The planar annulus inequality is due to Dai and Sun~\cite{DSannuli};
we give a proof using the same inverse-action comparison as in higher
dimensions. The earlier ball results are discussed after
Corollary~\ref{cor:balls}.

For shells, the quantitative estimates are uniform as the inner radius tends to zero
and as the shell becomes thin. They give a surplus of boundary order
above explicit frequency thresholds and an enhancement at every positive
energy with a constant depending only on dimension. For a fixed smooth
hole shape, the enhancement also permits removal of a sufficiently small
additional hole, uniformly over a compact family of radius ratios.

Let $d\ge2$, $0<r<R$, and
\[
 \Omega_{r,R}=\{y\in\mathbb R^d:r<|y|<R\}.
\]
For the ball endpoint, set $\Omega_{0,R}=B_R$.
Write $0=\lambda_0^N<\lambda_1^N\le\lambda_2^N\le\cdots$ for
its Neumann eigenvalues, counted with multiplicity, and define the
strict counting function by
\[
 N^<_\Omega(\Lambda)=\#\{j\ge0:\lambda_j^N(\Omega)<\Lambda\}.
\]
For $n\ge1$, let $\omega_n$ be the volume of the unit ball in
$\mathbb R^n$ and put $L_{0,n}=\omega_n/(2\pi)^n$. Thus
\[
 L_{0,d}|\Omega_{r,R}|\Lambda^{d/2}
 =\frac{(R^d-r^d)\Lambda^{d/2}}{2^d\Gamma(d/2+1)^2}.
\]

\subsection{The strict inequality for shells and balls}

\begin{theorem}[Strict P\'olya inequality]\label{thm:all-dimensions}
For every $d\ge2$, $0<r<R$, and $\Lambda>0$,
\begin{equation}\label{eq:polya}
 N^<_{\Omega_{r,R}}(\Lambda)>
 \frac{(R^d-r^d)\Lambda^{d/2}}{2^d\Gamma(d/2+1)^2}.
\end{equation}
Equivalently, for every integer $k\ge1$,
\[
 \lambda_k^N(\Omega_{r,R})<
 \left(\frac{2^d\Gamma(d/2+1)^2k}{R^d-r^d}\right)^{2/d}.
\]
At $\Lambda=0$, both sides of the counting inequality are zero.
The planar annulus result is due to Dai and Sun~\cite[Proposition~5.13]{DSannuli};
Section~\ref{sec:planar} gives a proof using the comparison developed here.
\end{theorem}

The phase, variational, and counting estimates include $r=0$ in every
dimension. At this endpoint the inner branch is empty and the phase
defect vanishes, giving the ball inequality directly.

\begin{corollary}[Neumann balls]\label{cor:balls}
For every $d\ge2$, $R>0$ and $\Lambda>0$,
\[
 N^<_{B_R}(\Lambda)>\frac{R^d\Lambda^{d/2}}{2^d\Gamma(d/2+1)^2},
 \qquad B_R=\{y\in\mathbb R^d:|y|<R\}.
\]
\end{corollary}

The strict disk inequality was proved in
\cite[Corollary~1.6]{FLPSdisk}, with certified computation on a bounded
frequency interval. The disk proof in \cite[Section~4]{Li} replaces that
step by explicit Rayleigh--Ritz estimates and finite rational calculations
printed in the paper. In dimensions $d\ge3$, the Neumann ball inequality
was proved in \cite{FLPSballs}, where certified computation completes
$3\le d\le12$, and independently in \cite{Li}, where an ancillary
exact-arithmetic program verifies one compact estimate for $d\ge7$.
Their main statements use a non-strict inequality, with inclusive counting
in \cite{FLPSballs} and strict counting in \cite{Li}.
The proof given here is strict at every positive energy and requires no
computer-assisted step.

\subsection{Uniform spectral surplus}

Following He--Jiang--Lin--Wang~\cite{HJLW}, a domain satisfies the
\emph{enhanced Neumann P\'olya inequality} if there are $c_1,c_2>0$
such that
\[
 N^<_\Omega(\Lambda)
 \ge L_{0,d}|\Omega|\Lambda^{d/2}
       +\max\{c_1\Lambda^{(d-1)/2},c_2\}
 \qquad(\Lambda>0).
\]
For shells, the enhancement constant can be chosen independently of
the radius ratio. The dimensional proofs below provide the required
quantitative estimates. In dimension two,
Proposition~\ref{prop:planar-surplus} gives a surplus greater than $K/4000$
for $K=R\sqrt\Lambda\ge8$, uniformly in the radius ratio.

\begin{theorem}[Enhancement uniform over all radius ratios]
\label{thm:uniform-enhancement}\label{thm:uniform-enp}
For each $d\ge2$ there is $\varepsilon_d>0$, depending only on $d$,
such that, for every $0<r<R$ and $\Lambda>0$,
\begin{equation}\label{eq:uniform-enhancement}
 N^<_{\Omega_{r,R}}(\Lambda)
 \ge\frac{(R^d-r^d)\Lambda^{d/2}}
           {2^d\Gamma(d/2+1)^2}
 +\varepsilon_d\bigl[1+R^{d-1}\Lambda^{(d-1)/2}\bigr].
\end{equation}
\end{theorem}

The proof combines the explicit high-frequency surplus of
Theorem~\ref{thm:effective-surplus} with lower semicontinuity of the
counting function on compact sets of radius ratios that contain the ball
endpoint $r=0$. Strictness at this endpoint follows from
Corollary~\ref{cor:balls} in every dimension.
The zero mode covers the remaining limit near $r=R$.
The constant $\varepsilon_d$ is not explicit. The next theorem gives
explicit coefficients and two frequency thresholds.

\paragraph{Explicit high-frequency bounds.}

\begin{theorem}[Explicit frequency thresholds]\label{thm:effective-surplus}
For each $d\ge2$, let $k_d,w_d,c_d$ be the following constants:
\begin{equation}\label{eq:effective-surplus-constants}
\begin{array}{c|c|c|c}
 d & k_d & w_d & c_d\\ \hline
 2 & 8 & 0 & 1/4000\\[2pt]
 3 & 44/5 & 4\pi/3 & 1/100000\\[2pt]
 4 & 12 & 16/3 & 1/20000\\[2pt]
 d\ge5 & \frac74d^{3/2} & \frac{3\pi}{4}\sqrt{d-1}
        & \frac{1}{1000(d-1)!}
\end{array}
\end{equation}
The estimate
\begin{equation}\label{eq:effective-surplus}
 N^<_{\Omega_{r,R}}(\Lambda)
 -\frac{(R^d-r^d)\Lambda^{d/2}}{2^d\Gamma(d/2+1)^2}
 >c_dR^{d-1}\Lambda^{(d-1)/2}
\end{equation}
holds under either of the following conditions:
\begin{equation}\label{eq:effective-surplus-onset}
 R\sqrt\Lambda\ge k_d,
 \qquad (R-r)\sqrt\Lambda\ge w_d,
\end{equation}
or
\begin{equation}\label{eq:all-radii-explicit-onset}
 R\sqrt\Lambda\ge K_d^*,\qquad
 K_d^*=\begin{cases}
 8,&d=2,\\
 24(d-1)^2,&d\ge3.
 \end{cases}
\end{equation}
\end{theorem}

The planar row follows from Proposition~\ref{prop:planar-surplus}.
Dai--Sun's larger coefficient $1/20$ for $K>80$ is recorded in
Remark~\ref{rem:planar-dai-sun}.
Condition~\eqref{eq:effective-surplus-onset} separates the outer-radius
scale from the width scale and can give a lower threshold for a
fixed shell. Condition~\eqref{eq:all-radii-explicit-onset} is independent
of the width. It follows by combining the thick-shell comparison with a
Neumann cosine-trial estimate and an elementary Riesz-mean inequality
for the Neumann interval (Lemma~\ref{lem:interval-riesz}); a related
interval inequality is \cite[Lemma~2.6]{HJLW}.
The power $\Lambda^{(d-1)/2}$ is the boundary order in the shell Weyl
law~\cite{GJWY}; the constants above are convenient explicit choices
and are not optimized.

\subsection{Removal of a small hole}

He--Jiang--Lin--Wang's removal theorem~\cite[Theorem~1.1]{HJLW}
preserves an enhanced inequality when a sufficiently small smooth hole
is removed from a fixed host. For a compact family of shells, the
enhancement above and a uniform estimate on bounded-energy spectral
subspaces give a common admissible hole size.

\begin{corollary}[Uniform removal from a compact family of shells]
\label{cor:perforated-shells}\label{cor:uniform-perforated-shells}
Fix $d\ge2$, a compact interval
$I=[\rho_-,\rho_+]\Subset(0,1)$, and a bounded connected domain
$U\subset\mathbb R^d$ with smooth boundary. Put
$\Omega_\rho=\Omega_{\rho,1}$. There exist constants
$a_0,C_1',C_2'>0$, depending only on $d,I,U$, such that, for every
$\rho\in I$, every $0<a<a_0$, and every $y\in\mathbb R^d$ satisfying
$\overline{y+aU}\subset\Omega_\rho$, the set
\[
 \Omega'_{\rho,a,y}=\Omega_\rho\setminus\overline{y+aU}
\]
satisfies
\[
 N_{\Omega'_{\rho,a,y}}^<(\Lambda)
 \ge L_{0,d}|\Omega'_{\rho,a,y}|\Lambda^{d/2}
       +\max\{C_1'\Lambda^{(d-1)/2},C_2'\}
 \qquad(\Lambda>0).
\]
In particular, these perforated shells satisfy the strict Neumann
P\'olya inequality.
\end{corollary}

The proof applies the same removal argument in every dimension;
its planar case uses the high-frequency surplus of
Proposition~\ref{prop:planar-surplus} through
Theorem~\ref{thm:uniform-enhancement}.
The additional hole need not be spherical or concentric with the
host shell. For outer radius $R$, the same conclusion applies to holes
$y+aR\,U$ with $a<a_0$, with enhancement constants
$R^{d-1}C_1'$ and $C_2'$. The compactness restriction on $I$
controls the spectral-subspace estimates used in the removal step.

\subsection{Previous work}

P\'olya proved the conjecture for tiling domains \cite{Polya1961};
Kellner removed the additional assumptions
in the Neumann case \cite{Kellner}. For general domains the
Berezin--Li--Yau and Kr\"oger inequalities give bounds with dimensional
factors \cite{Berezin1972,LiYau,Kroger,Laptev}; Filonov improved the
Neumann constant for planar convex domains \cite{FilonovConvex}.
General background is given in \cite{LMP,FLW}.

Filonov, Levitin, Polterovich and Sher proved the conjecture for disks and planar sectors and its
Dirichlet form for balls of every dimension \cite{FLPSdisk}, followed
by the Dirichlet annulus theorem \cite{FLPSDirichletAnnuli}.
The Neumann disk and Dirichlet annulus proofs include certified
computation on bounded parameter ranges. Dai and Sun proved the
Neumann inequality for planar annuli, with quantitative refinements
\cite{DSannuli}, and studied annular sectors \cite{DSsectors}.
The companion paper \cite{LiDirichletShells} proves the Dirichlet
inequality for spherical shells and Aharonov--Bohm annuli. Together,
the two papers give the
Dirichlet and Neumann inequalities for every concentric spherical shell.

Product constructions give further examples. Laptev's Dirichlet theorem
applies to $\Omega_1\times\Omega_2$ when
$\Omega_1\subset\mathbb R^{d_1}$ satisfies the conjecture, $d_1\ge2$,
and $\Omega_2$ is bounded \cite[Theorem~2.8]{Laptev}; thin products
are studied in \cite{HeWang}. Other extensions concern unions and
non-tiling families \cite{FreitasLagacePayette,FreitasSalavessa} and
magnetic disks \cite{FLPSAB}. Quantitative bounds are developed in
\cite{JiangLin,GMWZ}; the enhancement and removal results above use
the formulation of He, Jiang, Lin and Wang \cite{HJLW}.

\subsection{Proof strategy and organization}\label{sec:proof-overview}

Write
\[
 K=R\sqrt\Lambda,\qquad \rho=r/R,\qquad \mu=\rho K,
 \qquad T=(K-\mu)/\pi.
\]
Separation of variables reduces the spectrum to radial counts with
explicit angular multiplicities. Phase estimates and lattice-counting
arguments for balls and annuli were developed in
\cite{FLPSdisk,FLPSballs,FLPSDirichletAnnuli,DSannuli}; we use the uniform
Bessel phase bounds of~\cite{FLPSphase}. The Dirichlet annulus argument
already uses a difference of ball actions and exact volume cancellation
\cite[Section~7.1]{FLPSDirichletAnnuli}.

The main difficulty is that the shell action $A=G_K-G_\mu$, defined in
\eqref{eq:action}, is concave on $(0,\mu)$ and convex on $(\mu,K)$.
It is $1/2$-Lipschitz, but the convex-profile lattice estimate of
\cite[Theorem~6.1]{FLPSdisk} does not apply on its whole domain.
Lemma~\ref{lem:variation-comparison} bounds the difference of sampling
errors by the variation of the difference of derivatives. We apply it to
the inverse actions of the shell and the ball, whose common outer branch
cancels. For the weights used in dimensions at least three, monotonicity
of the ball slope and unimodality of the shell slope bound the remaining
error by one eighth of the maximal weighted inner slope
(Lemma~\ref{lem:general-transfer}). The quadrature
argument uses a bounded periodic primitive, also appearing in
\cite[Lemma~2.8]{DSannuli}.

An exact moment identity preserves the volume of the removed ball.
Together with the localized phase estimate of
Proposition~\ref{prop:phase}, this gives the counting comparison
\[
 \text{count}-\text{Weyl term}
 \ \ge\ \text{ball surplus}+\text{inner-volume gain}
       -\text{sampling loss}-\text{phase loss}.
\]
One clipped angular polynomial treats every $d\ge5$; quadratic and cubic
weights treat $d=3,4$. The phase estimate identifies the angular degrees
where the Neumann boundary correction can cause a loss near the inner
turning point. Radial trial spaces cover the low-frequency and small-width
ranges where the comparison alone is insufficient.

Dimension two requires a linear angular weight, whose inverse-action
slope is singular at the terminal point. Subtracting a matching
semicircle and the ball's terminal linear term gives a finite comparison
bound (Lemma~\ref{d2:lem:semicircle}). Concavity in the squared inner
frequency (Lemma~\ref{d2:lem:concavity}) reduces the count to width
endpoints, and one estimate covers every endpoint with at least three
radial rows. The two-row and variational estimates complete the annulus
proof. The disk is treated first by the linear-weight ball estimate.

Section~\ref{sec:common-quadrature} develops the sampling comparison;
Sections~\ref{sec:common-phase} and~\ref{sec:variational} supply the radial
phase and variational estimates. Section~\ref{sec:dimensions} proves the
strict inequality in every dimension and concludes with the shell and
ball results. Its parameter ranges are summarized in
Table~\ref{tab:roadmap}. Section~\ref{sec:enhanced-removal} proves the
uniform surplus and hole-removal results. The appendices contain the
phase enclosures, scalar estimates, and detailed variational calculations.

The source archive includes optional exact-arithmetic checks of the
finite inequalities; the formulas and their proofs are given in the text.

\section{The weighted shell--ball comparison}
\label{sec:common-quadrature}

We compare weighted action sums for the shell and the ball, then relate
them to the eigenvalue count. The comparison uses only the action
curves. Its spectral application requires a radial count estimate,
which is proved in Section~\ref{sec:common-phase}.

\subsection{Normalization and the shell action}
We use the dimensionless parameters
\begin{equation}\label{eq:notation}
 K=R\sqrt\Lambda,\qquad \mu=r\sqrt\Lambda=\rho K,
 \qquad \rho=r/R,\qquad h=1-\rho,
 \qquad p=d-1,\qquad \alpha=(d-2)/2.
\end{equation}
We allow $0\le\rho<1$. At $\rho=0$, $\Omega_{0,1}$ denotes the
unit ball. After scaling, $N^<(K^2)$ is the strict Neumann count of
$\Omega_{\rho,1}$, including its constant eigenfunction when $K>0$. Write
\[
 \W_d(K,\mu)=\frac{K^d-\mu^d}{2^d\Gamma(d/2+1)^2},
 \qquad W_d(K)=\W_d(K,0).
\]
The angular eigenvalue and Bessel order in degree $\ell$ are
$q_\ell=\ell(\ell+d-2)$ and $\nu_\ell=\ell+\alpha$.
The multiplicities are $\kappa_{2,0}=1$, $\kappa_{2,\ell}=2$ for
$\ell\ge1$, and, for $d\ge3$,
\[
 \kappa_{d,\ell}=\frac{(2\ell+d-2)(\ell+d-3)!}{\ell!(d-2)!},
\]
with $\kappa_{3,0}=1$.
Define
\begin{equation}\label{eq:action}
 G_L(z)=\frac{\sqrt{L^2-z^2}-z\arccos(z/L)}{\pi}
 \quad(0\le z\le L,\ L>0),\qquad G_L(z)=0\quad(z>L),
\end{equation}
set $G_0\equiv0$, and put $A(z)=G_K(z)-G_\mu(z)$ and
$T=(K-\mu)/\pi$.
Table~\ref{tab:notation} collects the notation used in several sections.

\begin{table}[htbp]
\centering\small
\renewcommand{\arraystretch}{1.15}
\begin{tabular}{@{}p{0.25\textwidth}p{0.68\textwidth}@{}}
\toprule
Notation & Meaning and definition\\
\midrule
$K,\mu,\rho,h,T$
& Outer and inner frequencies, radius ratio, normalized width, and
action height; \eqref{eq:notation} and $T=(K-\mu)/\pi$.\\
$G_L,A;\ \xi_L,x$
& Ball and shell actions, and their inverses;
\eqref{eq:action} and Section~\ref{sec:inverse-actions}.\\
$E_\delta;\ E_S,E_B;\ M_F$
& Shifted sampling error, its shell and ball versions, and the
maximum weighted inner slope; Section~\ref{sec:weighted-comparison}.\\
$\Psi_p;\ \Phi_2,\Phi_3$
& Angular comparison weights for $d\ge5$ and $d=3,4$;
\eqref{eq:clipped-weight} and \eqref{eq:jp-poly}.\\
$D_{\mathrm{ph}},H_d^*$
& Localized phase defect and its angular cutoff;
Proposition~\ref{prop:phase} and \eqref{eq:phase-new-H}.\\
$P_2,\W_2,\Sigma_n$
& Planar lattice count, Weyl term, and fixed-row surplus;
\eqref{d2:eq:phase-count} and \eqref{d2:eq:fixed-row-surplus}.\\
\bottomrule
\end{tabular}
\caption{Notation used across the comparison and dimensional proofs.
Auxiliary scalar functions are defined locally.}
\label{tab:notation}
\end{table}

\subsection{Inverse actions and their moments}\label{sec:inverse-actions}
The shell action decreases strictly from $T$ to zero on $[0,K]$.
Let $x=A^{-1}$ and, for $L>0$, $\xi_L=G_L^{-1}$, all extended by
zero beyond their natural domains; put $\xi_0\equiv0$. The shell and outer-ball inverses agree for $0\le t\le c$, where
$c=A(\mu)=G_K(\mu)$, and $x\le\xi_K$ everywhere. At $\mu=0$
one has $x=\xi_K$ and $c=T$.
For every $m\ge0$, changing variables $t=A(z)$ and integrating by
parts gives
\begin{equation}\label{eq:moment}
 \int_0^T x(t)^m\,dt
 =\frac{K^{m+1}-\mu^{m+1}}{\pi(m+1)}I_{m+1},
 \qquad I_j=\int_0^{\pi/2}\cos^j u\,du.
\end{equation}
Indeed, $-G_K'(z)=\pi^{-1}\arccos(z/K)$; substitution $z=K\cos u$
evaluates its weighted integral. The Wallis formula therefore gives,
for $p=d-1$,
\begin{equation}\label{eq:weylmoment}
 \W_d(K,\mu)=\frac2{p!}\int_0^T x(t)^p\,dt.
\end{equation}

\begin{figure}[htbp]
  \centering
  \includegraphics[width=\linewidth]{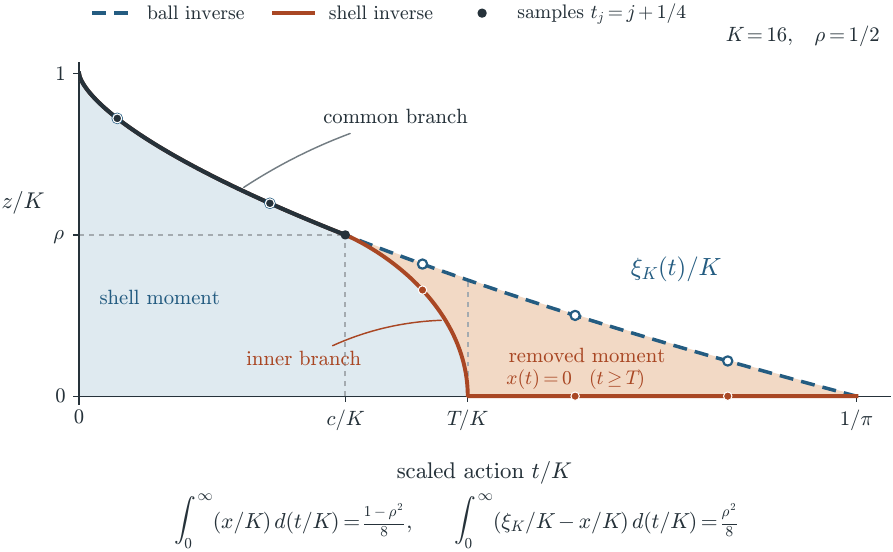}
  \caption{Inverse actions for $K=16$ and $\mu=8$, plotted in the
  scaled coordinates $(t/K,z/K)$. The ball inverse
  $\xi_K=G_K^{-1}$ and the shell inverse
  $x=(G_K-G_\mu)^{-1}$ agree on $0\le t\le c=G_K(\mu)$;
  the shell's inner branch occupies $c<t<T=(K-\mu)/\pi$.
  Both inverses are extended by zero, and the dots mark the quarter
  samples $t_j=j+1/4$. The shaded regions give the normalized first
  moment of the shell and the moment removed from the ball.
  The moment identity \eqref{eq:moment} gives the corresponding
  formula for the power $p=d-1$ used in dimension $d$.}
  \label{fig:action-comparison}
\end{figure}

\subsection{A weighted quadrature comparison}\label{sec:weighted-comparison}
The elementary comparison below separates the quadrature argument
from the geometry of the inverse actions.

\begin{lemma}[Sampling errors and variation]
\label{lem:variation-comparison}
Let $f_1,f_2$ be absolutely continuous, compactly supported functions
on $[0,\infty)$, and put
\[
 E_\delta(f)=\sum_{j\ge0}f(j+\delta)-\int_0^\infty f(t)\,dt.
\]
Suppose $v=f_2'-f_1'$ has a right-continuous representative of
bounded variation on $[0,\infty)$. Let $V_+,V_-$ be its positive
and negative variation masses on $(0,\infty)$, including all jumps
to zero at the ends of the supports. Then, for every $0<\delta<1$,
\[
 \left|E_\delta(f_1)-E_\delta(f_2)
 -(1/2-\delta)\bigl(f_1(0)-f_2(0)\bigr)\right|
 \le\frac18\max\{V_+,V_-\}.
\]
\end{lemma}
\begin{proof}
Set
\[
 \psi_\delta(t)=\{t+1-\delta\}-\tfrac12,\qquad
 Q_\delta(t)=\tfrac12\psi_\delta(t)^2,\qquad
 P_\delta(t)=Q_\delta(t)-Q_\delta(0).
\]
The function $Q_\delta$ is continuous, $Q_\delta'=\psi_\delta$
almost everywhere, and $0\le Q_\delta\le1/8$.
Fubini's theorem and Stieltjes integration by parts give
\[
 \begin{aligned}
 &E_\delta(f_1)-E_\delta(f_2)
 -(1/2-\delta)\bigl(f_1(0)-f_2(0)\bigr)\\
 &\qquad=-\int_0^\infty v(t)\psi_\delta(t)\,dt
 =\int_0^\infty P_\delta(t)\,dv(t),
 \end{aligned}
\]
because $P_\delta(0)=0$ and $v$ is eventually zero.
Writing $q_0=Q_\delta(0)$, the last integral lies between
$-q_0V_+-(1/8-q_0)V_-$ and
$(1/8-q_0)V_++q_0V_-$.
Since $0\le q_0\le1/8$, both bounds imply the assertion.
\end{proof}

For a nonnegative function $F$ with $F(0)=0$ and $0<\delta<1$, define
\[
 E_S(F;\delta)=\sum_{j\ge0}F(x(j+\delta))
                  -\int_0^\infty F(x(t))\,dt,
\]
and define $E_B(F;\delta)$ by replacing $x$ with $\xi_K$.
Write $E_{S,p}(\delta),E_{B,p}(\delta)$ when $F(z)=z^p$.
When the shift is omitted, it is $\delta=1/4$.
The next lemma bounds the absolute difference of the sampling errors.
Under its hypotheses, the inner slope is unimodal, and the bound is
one eighth of its maximum.

\begin{lemma}[Transfer for a convex angular weight]
\label{lem:general-transfer}
Suppose $F$ is continuous and convex, vanishes on $[0,a]$ for
some $a\ge0$, is $C^2$ with $F'>0$ on $(a,\infty)$, and
\[
 z\longmapsto\frac{zF''(z)}{F'(z)}-1
 \quad\hbox{is nonincreasing on }(a,\infty).
\]
For $0<z<\mu$, put
\[
 H(z^2)=\frac1\pi\int_\mu^K
       \frac{ds}{s^2\sqrt{1-z^2/s^2}}.
\]
If $\mu>a$, assume the one-sided supremum
\[
 M_F=\sup_{a<z<\mu}\frac{F'(z)}{zH(z^2)}
\]
is finite; otherwise set $M_F=0$. Then, for every $0<\delta<1$,
\begin{equation}\label{eq:general-transfer}
 \bigl|E_S(F;\delta)-E_B(F;\delta)\bigr|\le\frac18M_F.
\end{equation}
If $\mu\le a$, the two errors are equal. If $\mu>a$, then
\begin{equation}\label{eq:general-Jensen}
 M_F\le\frac{\mu K}{T}
       \sup_{a<z<\mu}\frac{F'(z)}z\sqrt{1-\beta_*z^2},
 \qquad \beta_* =\frac{K^2+K\mu+\mu^2}{3K^2\mu^2}.
\end{equation}
\end{lemma}
\begin{proof}
Equality when $\mu\le a$ follows because the two inverse curves
agree wherever $F$ is nonzero. Suppose $\mu>a$. On the inner branch,
$-A'(z)=zH(z^2)$, so the magnitude of the weighted slope is
\[
 g_S(t)=-\frac d{dt}F(x(t))
       =\frac{F'(z)}{zH(z^2)},\qquad z=x(t)>a.
\]
The positive integrands of $H(y)$ are strictly log-convex in $y$.
H\"older's inequality makes $H$ log-convex, and $H'>0$.
Consequently $2y(\log H)'(y)$ increases strictly. Thus
\[
 \frac{d\log g_S}{d\log z}
 =\frac{zF''(z)}{F'(z)}-1-2z^2\frac{H'(z^2)}{H(z^2)}
\]
decreases strictly. The slope is unimodal in $z$, and hence also
in the reverse variable $t$; either monotone part may be empty.
Its terminal jump to zero at $t=A(a)$ preserves unimodality.
This includes $a=0$, using its finite one-sided limit.
The corresponding ball slope is
\[
 g_B(t)=\frac{\pi F'(z)}{\arccos(z/K)},\qquad z=\xi_K(t)>a.
\]
It decreases with $t$, by convexity of $F$, and its terminal jump
is again downwards. Extend both slopes by zero after their
respective vanishing times. They agree on $(0,c)$ and at $c$;
write $g_0=g_S(c)=g_B(c)$.

On $[c,\infty)$, the positive and negative variations of $g_S$
have masses $M_F-g_0$ and $M_F$, including its terminal jump;
those of $g_B$ have masses zero and $g_0$.
Thus each variation mass of $g_S-g_B$ is at most $M_F$.
The inverse curves are absolutely continuous and $F$ is Lipschitz on
$[0,K]$, so both weighted compositions are absolutely continuous.
Apply Lemma~\ref{lem:variation-comparison} to
$f_1=F\circ x$ and $f_2=F\circ\xi_K$, which agree on $[0,c]$.

For \eqref{eq:general-Jensen}, use probability measure proportional
to $s^{-2}ds$ on $[\mu,K]$, under which the mean of $s^{-2}$ is
$\beta_*$. Jensen's inequality for $v\mapsto(1-z^2v)^{-1/2}$ gives
\[
 H(z^2)\ge\frac{T}{\mu K}(1-\beta_*z^2)^{-1/2},
\]
which bounds the defining supremum for $M_F$.
\end{proof}

\begin{corollary}[Power weights]\label{lem:radiusM}
For $p>2$ and $\mu>0$, the inner slope $-(x^p)'$ is unimodal.
With $M_p=0$ at $\mu=0$, one has
\begin{equation}\label{eq:transfer}
 \left|E_{S,p}(\delta)-E_{B,p}(\delta)\right|\le\frac18M_p.
\end{equation}
Moreover,
\begin{equation}\label{eq:Mbound}
 \begin{split}
 M_p&\le h_p\rho
       \left(\frac{3\rho^2}{1+\rho+\rho^2}\right)^{(p-2)/2}
                  \frac{K^p}{T}
       \le h_p\rho^{p/2}\frac{K^p}{T},\\
 h_p&=\frac p{\sqrt{p-1}}
       \left(\frac{p-2}{p-1}\right)^{(p-2)/2}.
 \end{split}
\end{equation}
The ratio $h_p/\sqrt p$ decreases strictly for $p>2$.
\end{corollary}
\begin{proof}
At $\mu=0$, both the transfer difference and $M_p$ vanish.
Suppose $\mu>0$ and take $F(z)=z^p$ in Lemma~\ref{lem:general-transfer}. The expression in the
monotonicity hypothesis is the constant $p-2$. Maximizing $z^{p-2}\sqrt{1-\beta_*z^2}$ over
$0\le z\le\beta_*^{-1/2}$ gives the first bound. The second uses
$3\rho^2/(1+\rho+\rho^2)\le\rho$.
Put $n=p-1>1$. Differentiating the logarithm of $h_p^2/p$ gives
\[
 \frac d{dn}\{\log(1+1/n)+(n-1)\log(1-1/n)\}
 =\log(1-1/n)+\frac1n-\frac1{n(n+1)}<0.
\]

\end{proof}

\subsection{Angular rows and the counting reduction}
\label{sec:common-angular}
Let $p=d-1\ge2$, $\alpha=(p-1)/2$, and
$C_d(z)=\sum_{\ell+\alpha\le z}\kappa_{d,\ell}$.
\begin{lemma}[Angular multiplicity bounds]
\label{lem:angular}
For every $z\ge0$,
\begin{equation}\label{eq:angular}
 C_d(z)\le\frac2{p!}(z+1/2)^p.
\end{equation}
For $p\ge5$, put $b_p=p(p-1)(p-5)/24$. Then
\begin{equation}\label{eq:angular-threepower}
 C_d(z)\ge\frac2{p!}
       \left(z^p-\frac p2z^{p-1}-b_pz^{p-2}\right).
\end{equation}
When $z\ge\alpha$, put $L=\lfloor z-\alpha\rfloor$ and $u=L+p/2$.
The exact formula is
\begin{equation}\label{eq:angular-exact}
 C_d(z)=\frac{2u}{p!}\prod_{i=1}^{p-1}(u+i-p/2).
\end{equation}
\end{lemma}
\begin{proof}
Summing the spherical harmonic multiplicities gives the product.
Its offsets are symmetric, with the squares of the positive offsets
summing to $\gamma_p=p(p-1)(p-2)/24$. Pairing factors gives \eqref{eq:angular}, since
$u\le z+1/2$.

For the lower bound assume $p\ge5$ and $z\ge\alpha$.
The product increases with $u\ge(p-2)/2$, so substitute
$u\ge z-1/2$ and use $\prod_i(1-v_i)\ge1-\sum_i v_i$ for
$0\le v_i\le1$ after pairing. Thus
\[
 \frac{p!}{2}C_d(z)\ge(z-1/2)^p-\gamma_p(z-1/2)^{p-2}.
\]
Put $v=1/(2z)\le1/(p-1)$. Taylor's lower bound through degree
three for $(1-v)^p$ and upper bound through degree two for
$(1-v)^{p-2}$ yield
\begin{align*}
 \frac{p!C_d(z)}{2z^p}
 &\ge(1-v)^p-4\gamma_pv^2(1-v)^{p-2}\\
 &\ge1-pv-4b_pv^2
 +\frac{p(p-1)(p-2)(p-3)}6v^3
                 \left(1-\frac{p-2}{2}v\right)\\
 &\ge1-pv-4b_pv^2.
\end{align*}
This proves \eqref{eq:angular-threepower} above $\alpha$.
Below $\alpha<p/2$, its right side is nonpositive because $b_p\ge0$.
\end{proof}

Define the three-quarter action count
\[
 P_d(K,\mu)=\sum_{\ell\ge0}\kappa_{d,\ell}
        \lfloor A(\ell+\alpha)+3/4\rfloor.
\]
Interchanging angular sectors and radial rows gives exactly
\begin{equation}\label{eq:rows}
 P_d(K,\mu)=\sum_{j\ge0}C_d(x(j+1/4)).
\end{equation}
Here $\alpha>0$, so a zero-extended terminal inverse contributes
zero. Equality at a node is included on both sides.

For a continuous weight $F$ with $F(0)=0$, write
\begin{align*}
 Q_F(K)&=\sum_{j\ge0}F(\xi_K(j+1/4)),\\
 J_F(L)&=\frac1\pi\int_0^L F(z)\arccos(z/L)\,dz\quad(L>0),
 \qquad J_F(0)=0.
\end{align*}
Changing variables in the inverse action gives
$\int F(x(t))\,dt=J_F(K)-J_F(\mu)$.

\begin{proposition}[A counting reduction for angular weights]
\label{prop:physical-transfer}
Let $d\ge3$ and suppose that $0\le F\le C_d$ satisfies the
hypotheses of Lemma~\ref{lem:general-transfer}. Let $D_{\mathrm{ph}}\ge0$
satisfy $N^<(K^2)\ge P_d(K,\mu)-D_{\mathrm{ph}}$. Then
\begin{equation}\label{eq:physical-transfer}
\begin{split}
 N^<(K^2)-\W_d(K,\mu)\ge{}&Q_F(K)-\W_d(K,0)\\
 &+\bigl[\W_d(\mu,0)-J_F(\mu)\bigr]
 -\frac18M_F-D_{\mathrm{ph}}.
\end{split}
\end{equation}
If $F(z)\le2z^p/p!$, the bracketed inner-volume term is nonnegative.
At $\mu=0$, this term and $M_F$ both vanish.
\end{proposition}
\begin{proof}
The angular row identity gives
$N^<(K^2)\ge\sum_jF(x(j+1/4))-D_{\mathrm{ph}}$.
Add $E_S(F)\ge E_B(F)-M_F/8$ to the integral identity above and use
$\W_d(K,\mu)=\W_d(K,0)-\W_d(\mu,0)$.
The final sign follows from \eqref{eq:weylmoment} applied to $2z^p/p!$.
\end{proof}

The proposition is conditional only on the stated comparison between
the radial count and the action count. Section~\ref{sec:common-phase}
will supply a valid defect $D_{\mathrm{ph}}$ for use in the dimensional
proofs.

\subsection{The clipped angular polynomial}
For $p\ge4$, set $b_4=0$ and retain
$b_p=p(p-1)(p-5)/24$ for $p\ge5$. Define
\begin{equation}\label{eq:clipped-weight}
 R_p(z)=z^p-\frac p2z^{p-1}-b_pz^{p-2},\qquad
 \Psi_p(z)=(R_p(z))_+.
\end{equation}
Let $a$ be the positive root of $z^2-(p/2)z-b_p$.

\begin{lemma}[The clipped angular weight]\label{lem:clipped-angular}
For every integer $p\ge4$, one has $C_{p+1}\ge2\Psi_p/p!$.
The weight $\Psi_p$ satisfies the hypotheses of
Lemma~\ref{lem:general-transfer}, and $M_{\Psi_p}\le M_p$.
If $C=\Psi_p-R_p$ and $K\ge K_0=(7/4)(p+1)^{3/2}$, then
\begin{equation}\label{eq:clipping-correction}
 0\le\frac{J_C(\mu)}{K^p}
 <\begin{cases}
 1/100000,&p=4,\\
 1/200000,&p\ge5.
 \end{cases}
\end{equation}
The correction vanishes at $\mu=0$.
\end{lemma}
\begin{proof}
For $p\ge5$, use \eqref{eq:angular-threepower} and $C_{p+1}\ge0$.
For $p=4$ and $z\ge2$, the exact multiplicity formula, with
$u=L+2\ge z-1/2$, gives
\[
 12C_5(z)\ge(z-1/2)^2\bigl((z-1/2)^2-1\bigr)
 =z^4-2z^3+\frac12z^2+\frac12z-\frac3{16}\ge R_4(z).
\]
Below two, $\Psi_4=0$.

Factor $R_p=z^{p-2}(z-a)(z+b)$, where $b\ge0$.
On $z\ge a$, its factors and their derivatives are nonnegative,
so $R_p$ is increasing and convex. At $a$, its derivative jumps
upwards from zero after clipping. Moreover,
$0<R_p'(z)\le pz^{p-1}$ for $z>a$.
Write
\[
 R_p'(z)=pz^{p-3}(z-r)(z+s),\qquad
 r-s=(p-1)/2>0,\qquad rs=(p-2)b_p/p\ge0.
\]
Thus $r>s\ge0$. Since
$R_p'(a)=a^{p-3}((p/2)a+2b_p)>0$, one has $a>r$, and
\[
 \frac{d}{dz}\left(\frac{zR_p''(z)}{R_p'(z)}-1\right)
 =-\frac r{(z-r)^2}+\frac s{(z+s)^2}<0\qquad(z>a).
\]
The one-sided inner slope is finite because $a>0$.
These facts verify the transfer hypotheses, and the derivative bound
gives $M_{\Psi_p}\le M_p$.

The correction $C=(-R_p)\mathbf1_{[0,a]}$ is nonnegative. Since the
action density is at most $1/2$ and $a^2=(p/2)a+b_p$,
\begin{equation}\label{eq:clipping-integral}
 J_C(\mu)\le\frac12\int_0^a[-R_p(z)]\,dz
 =\frac{a^p}{4(p+1)}
   +\frac{b_pa^{p-1}}{(p-1)(p+1)}.
\end{equation}
For $p=4$, $a=2$ and the last expression is $4/5$. Hence
$J_C(\mu)/K^4\le1024/187578125<1/100000$.
For $p\ge5$, use $a\ge p/2$ in \eqref{eq:clipping-integral} to get
\[
 \frac{J_C(\mu)}{K^p}
 \le\frac{p-2}{12(p+1)}\left(\frac aK\right)^p.
\]
For $Q(z)=z^2-(p/2)z-b_p$,
\[
 Q(p^{3/2}/4)=\frac p{48}
       \bigl[p(\sqrt p-3)^2+3p-10\bigr]>0.
\]
Since $p^{3/2}/4>p/2$, it follows that $a<p^{3/2}/4$ and
$a/K_0<1/7$. Thus $J_C(\mu)/K^p<7^{-p}/12<1/200000$.
\end{proof}

When Proposition~\ref{prop:physical-transfer} is applied to
$F=2\Psi_p/p!$, expanding $\Psi_p=R_p+C$ gives
\begin{equation}\label{eq:clipped-ball-reduction}
 Q_{\Psi_p}(K)-J_{\Psi_p}(\mu)
 \ge Q_{R_p}(K)-J_{R_p}(\mu)-J_C(\mu).
\end{equation}
Here $Q_{R_p}$ and $J_{R_p}$ are defined by linearity. The discarded
term is $Q_C(K)\ge0$, so only the inner-ball clipping integral is lost.

\subsection{Quarter sampling for power weights}
\label{sec:common-ball}
For the ball inverse, the parameterization
$\xi_K=K\cos u$, $t=K(\sin u-u\cos u)/\pi$ gives
\begin{equation}\label{eq:ballslope}
 -(\xi_K^p)'=\frac{p\pi K^{p-1}\cos^{p-1}u}{u}.
\end{equation}
This slope decreases for $p\ge1$.

\begin{lemma}[An upper bound for quarter sampling]\label{lem:exactquarter}
Let $f$ be compactly supported on $[0,\infty)$, continuous at zero,
with $f,f'$ locally absolutely continuous on $(0,\infty)$.
Suppose that $\eta=-D^3f$ is a nonnegative Radon measure on
$(0,\infty)$, where the distributional derivative includes the
extension of $f$ by zero through the right endpoint of its support. Then
\begin{equation}\label{eq:shifted-quarter-upper}
 E_{1/4}(f):=\sum_{j\ge0}f(j+1/4)-\int_0^\infty f(t)\,dt
 \le\frac14f(1/24)\le\frac14f(0).
\end{equation}
The first inequality is strict if $f$ is nonzero. In particular, for
every integer $m\ge2$,
\begin{equation}\label{eq:exactquarter-ball}
 \sum_{j\ge0}\xi_K(j+1/4)^m
 <K^m\left(\frac{KI_{m+1}}{\pi(m+1)}+\frac14\right).
\end{equation}
The sum with the shell inverse $x$ in place of $\xi_K$ satisfies
the same upper bound.
\end{lemma}
\begin{proof}
Integration from the right endpoint gives the ramp representation
\[
 f(t)=\int_0^\infty k_s(t)\,\eta(ds),\qquad
 k_s(t)=\frac12(s-t)_+^2,\qquad
 \int_0^\infty s^2\eta(ds)=2f(0)<\infty.
\]
The last identity follows by continuity and monotone convergence;
the representation also makes $f$ nonnegative and nonincreasing.
For $B_3(v)=v^3-3v^2/2+v/2$, direct summation gives
\begin{equation}\label{eq:quarter-ramp-deficit}
 D(s):=\frac14k_s(0)-E_{1/4}(k_s)
 =\frac{s}{96}+\frac1{128}+\frac16B_3(\{s+3/4\}).
\end{equation}
Here $\{\cdot\}$ denotes the fractional part, and
$|B_3(v)|\le\sqrt3/36$ for $0\le v\le1$. We shall also use,
for every $\tau>0$, the identities
\begin{equation}\label{eq:quarter-ramp-moments}
 \frac12\int\min\{s,\tau\}^2\eta(ds)
 =f(0)-f(\tau)+\tau f'(\tau),\qquad
 \int(s-\tau)_+\eta(ds)=-f'(\tau).
\end{equation}

Take $\tau=1/24$. For every $s>0$,
\[
 D(s)>\frac18\min\{s,\tau\}^2+\frac1{96}(s-\tau)_+.
\]
For $s\le1/4$, use $D(s)=s^2/8+s^3/6$: the difference is $s^3/6$
below $\tau$, and $(s-\tau)^2/8+s^3/6$ above it. For $s\ge1/4$,
the Bernoulli bound and $\sqrt3<111/64$ give
\[
 D(s)\ge\frac{s}{96}+\frac1{128}-\frac{\sqrt3}{216}
 >\frac{s}{96}-\frac1{4608}.
\]
By Tonelli's theorem and \eqref{eq:quarter-ramp-moments}, integration
cancels the derivative terms because $\tau/4=1/96$. This proves
\eqref{eq:shifted-quarter-upper}, strictly unless $\eta=0$, when $f=0$.

For the application put $\xi_K=K\cos u$ and
$t=K(\sin u-u\cos u)/\pi$. If $f=(\xi_K/K)^m$, then
\[
 f''=\frac{m\pi^2}{K^2}\left[
 \frac{(m-1)\cos^{m-2}u}{u^2}
 +\frac{\cos^{m-1}u}{u^3\sin u}\right].
\]
Both summands decrease with $u$ for $m\ge2$, so $-f'''\ge0$
in the interior. At $t=K/\pi$, $f$ vanishes to order $m$; its
second derivative extends continuously by zero if $m>2$, and jumps
down from $8/K^2$ to zero if $m=2$, adding a positive atom to $-D^3f$.
The moment identity \eqref{eq:moment} gives \eqref{eq:exactquarter-ball};
the shell assertion follows from $x\le\xi_K$.
\end{proof}

\begin{lemma}[Lower bounds for ball sampling]
\label{lem:onepoint-ball}
For every integer $p\ge2$ and $K\ge(7/4)(p+1)^{3/2}$,
\begin{equation}\label{eq:onepoint-ball}
 \frac{E_{B,p}(1/4)}{K^p}>\frac{18}{125}.
\end{equation}
For $p=4$, the right side may be replaced by $161/1000$;
for $p=5,6$, it may be replaced by $31/200$.
For the linear weight, every $K\ge8$ satisfies
\begin{equation}\label{eq:onepoint-ball-linear}
 E_{1/4}(\xi_K)>\frac K4-\frac4{25}K^{1/3}-\frac16.
\end{equation}
\end{lemma}
\begin{proof}
Put $\tau=17/50$ and $c=27/80$. We first prove that, under the
hypotheses of Lemma~\ref{lem:exactquarter},
\begin{equation}\label{eq:onepoint-quarter}
 E_{1/4}(f)\ge\mathcal L(f):=
 -\frac7{80}f(0)+cf(\tau)-\frac{313}{3000}f'(\tau).
\end{equation}
The ramp deficit \eqref{eq:quarter-ramp-deficit} satisfies
\[
 D(s)\le\frac c2\min\{s,\tau\}^2+\frac1{96}(s-\tau)_+.
\]
For $s\le1/4$, use $D(s)=s^2/8+s^3/6\le s^2/6$.
For $1/4\le s\le\tau$, set $w=s-1/4\in[0,9/100]$ and use
$80w^3\le9w^2$. The required inequality follows from
\[
 80s^3-261s^2+120s-15
 \le-192(w-3/256)^2-37/1024<0.
\]
For $s\ge\tau$, the Bernoulli bound gives
\[
 D(s)-\frac{s}{96}\le\frac1{128}+\frac{\sqrt3}{216}
 <\frac1{128}+\frac7{864}
 <\frac{c\tau^2}{2}-\frac\tau{96}
 =\frac{19159}{1200000}.
\]
Integrate against $\eta=-D^3f$ and apply
\eqref{eq:quarter-ramp-moments}. The coefficients are
$1/4-c=-7/80$ and $c\tau-1/96=313/3000$, proving
\eqref{eq:onepoint-quarter}.

Now put $f=(\xi_K/K)^p$, $\beta=313/1530$,
$a=(3\pi\tau)^{2/3}/2$, and $\varepsilon=aK^{-2/3}$.
Both frequency hypotheses imply $K\ge8$; since $\pi<22/7$,
$a<11/10$ and $\varepsilon<11/40<3/10$.
For $p\ge2$, the proof of Lemma~\ref{lem:exactquarter} verifies
the hypotheses of \eqref{eq:onepoint-quarter} for $f$.
For $p=1$, subtract
$r(t)=(2/K)(K/\pi-t)_+$ and put $h=f-r$.
The ball parameterization $\xi_K=K\cos u$ gives
$f'=-\pi/(Ku)$ and $f''=\pi^2/(K^2u^3\sin u)$.
Thus $f''$ decreases, $h(K/\pi)=h'(K/\pi)=0$, and the zero
extension of $h''$ has a downward terminal jump. Hence $-D^3h$
is a nonnegative measure, and \eqref{eq:onepoint-quarter} applies to $h$.
If $M$ counts the samples $j+1/4<K/\pi$ and $z=M-K/\pi$, then
\[
 E_{1/4}(r)=\frac1{2\pi}+\frac1K\left(\frac z2-z^2\right)
 \ge\frac1{2\pi}-\frac3{16K},\qquad
 -\frac14\le z\le\frac34,
\]
while $\mathcal L(r)=1/(2\pi)-1/(48K)$.
Consequently
\[
 E_{1/4}(f)\ge\mathcal L(f)-\frac{\mathbf1_{\{p=1\}}}{6K}.
\]

Writing $\xi_K(\tau)=K\cos u$, integration of
$\arccos(1-s)\ge\sqrt{2s}$ and inversion give $\cos u\ge1-\varepsilon$.
For $0\le s\le1/2$, squaring shows
$(1-s/2)^{-1/2}\le1+s/3$; integrating yields
\[
 u\le\arccos(1-\varepsilon)\le\sqrt{2\varepsilon}(1+\varepsilon/9),
 \qquad
 -f'(\tau)\ge\frac{2p\varepsilon}{3\tau}
       \frac{(1-\varepsilon)^{p-1}}{1+\varepsilon/9}.
\]
It follows that
\[
 \frac{E_{B,p}}{K^p}\ge B_p(\varepsilon)
                  -\frac{\mathbf1_{\{p=1\}}}{6K},\qquad
 B_p(\varepsilon):=-\frac7{80}
 +(1-\varepsilon)^p
 \left[c+\frac{\beta p\varepsilon}{(1-\varepsilon)(1+\varepsilon/9)}\right].
\]
For $p=1$, the same calculation gives
\[
 \mathcal L(\xi_K)\ge\frac K4
 -a\left(c-\frac\beta{1+\varepsilon/9}\right)K^{1/3}.
\]
Since $\beta>1/5$ and $\varepsilon<3/10$,
\[
 a\left(c-\frac\beta{1+\varepsilon/9}\right)
 <\frac{11}{10}\left(\frac{27}{80}-\frac6{31}\right)
 =\frac{3927}{24800}<\frac4{25},
\]
so
\begin{equation}\label{eq:linear-onepoint-functional}
 \mathcal L(\xi_K)>\frac K4-\frac4{25}K^{1/3}.
\end{equation}
Together with $E_{1/4}(\xi_K)\ge\mathcal L(\xi_K)-1/6$, this proves
\eqref{eq:onepoint-ball-linear}.

For $p\ge2$, put $n=p+1$. The frequency assumption gives
$0<\varepsilon<\varepsilon_0:=3/(4n)\le1/4$.
Since $\beta<c$, differentiation gives
$B_p'(\varepsilon)\le p(1-\varepsilon)^{p-1}(\beta-c)<0$.
The logarithmic bound $\log(1-t)\ge-t-t^2/[2(1-t)]$ gives
$(1-\varepsilon_0)^p\ge e^{-3/4}(1+3/(8n))$.
Also $(1-\varepsilon_0)(1+\varepsilon_0/9)<1$. With
$L=c+3\beta/4$, it follows that
\[
 B_p(\varepsilon_0)+\frac7{80}
 >e^{-3/4}\left(1+\frac3{8n}\right)
              \left(L-\frac{3\beta}{4n}\right)>Le^{-3/4}.
\]
For the last inequality use $n\ge3$ and $2c-3\beta=25/408>0$.
A geometric bound for the exponential tail completes the uniform estimate:
\[
 e^{3/4}<1+\frac34+\frac{(3/4)^2}{2}
 +\frac{(3/4)^3}{6(1-3/16)}=\frac{881}{416}<\frac{125}{59},
 \qquad -\frac7{80}+\frac{59L}{125}-\frac{18}{125}=\frac7{31875}>0.
\]
For $p=4,5,6$, direct substitution and clearing positive denominators give
\[
 B_4(3/20)>\frac{161}{1000},\qquad
 B_5(1/8)>\frac{31}{200},\qquad
 B_6(3/28)>\frac{31}{200}.
\]
\end{proof}

\section{The radial count and its phase defect}
\label{sec:common-phase}

The weighted comparison requires a lower bound for the eigenvalue count
in terms of the action count. We include $\rho=0$, with the regular
condition at the center of the ball. Proposition~\ref{prop:phase}
states the bound; the remainder of the section proves the ball case
and the three sufficient phase criteria for shells.

For $q\ge0$, let $\lambda_{q,j}(\rho)$,
$j=0,1,\ldots$, be the eigenvalues of the radial form
\[
 \frac{\displaystyle\int_\rho^1
        \bigl(s^{d-1}|u'|^2+q s^{d-3}|u|^2\bigr)\,ds}
      {\displaystyle\int_\rho^1s^{d-1}|u|^2\,ds}.
\]
The form domain consists of functions locally absolutely continuous
on $(\rho,1]$ with finite mass and energy in the displayed integrals; it is
$H^1(\rho,1)$ when $\rho>0$. The associated operator has Neumann
conditions at each positive endpoint and the regular ball condition
at zero. All trial functions below are taken in the form domain.
Thus $n_\ell^<(K^2)$ counts the radial eigenvalues
$\lambda_{q_\ell,j}(\rho)<K^2$, and
$N^<=\sum_{\ell\ge0}\kappa_{d,\ell}n_\ell^<$.

Write $S_\nu(t)=\pi G_t(\nu)$.
All radial counts below are strict and include the constant eigenfunction
in degree zero.  Our aim is the three-quarter count
\begin{equation}\label{eq:threequarter}
 n_\ell^<(K^2)\ge\lfloor A(\nu_\ell)+3/4\rfloor.
\end{equation}
For the ball it holds in every angular degree. For shells in dimension
three, it holds in every positive degree. In dimension four and above,
we bound the possible exceptional degrees explicitly. The planar count
is proved separately in Section~\ref{sec:planar}.

For $p=d-1$, define the phase cutoff
\begin{equation}\label{eq:phase-new-H}
 H_d^*=\max\left\{\frac{19}{2},
       \sqrt{\frac14+\frac{62500}{50653}p^3}\right\}.
\end{equation}
Useful envelopes, obtained by squaring, are
\begin{equation}\label{eq:cp-H-envelopes}
 H_d^*\le\max\{19/2,(9/8)p^{3/2}\}\quad(p\ge3),\qquad
 H_d^*\le(10/9)p^{3/2}\quad(p\ge8).
\end{equation}

\begin{proposition}[A localized phase defect]\label{prop:phase}\label{prop:tp-phasewindow}
Let $d\ge3$, $K>0$ and $0\le\rho<1$, with $\mu=\rho K$.
Use the action count $P_d$ from \eqref{eq:rows}, and set
\[
 \mathfrak b_d(L)=L(L+d)+\frac{2L^2}{2L+d+2}\quad(L\ge0).
\]
Define
\[
 \mathcal E_d=\left\{\ell\ge0:
 \nu_\ell<H_d^*,\quad
 \mu^2<\nu_\ell^2+\nu_\ell^{4/3}
                  +\frac53(\alpha+3/8)^2,\quad
 \mu^2>\mathfrak b_d(\ell)\right\},
 \qquad \mathfrak b_d(0)=0.
\]
Then
\begin{equation}\label{eq:phase-double-window}
 N^<(K^2)\ge P_d(K,\mu)-\sum_{\ell\in\mathcal E_d}\kappa_{d,\ell}.
\end{equation}

At $\mu=0$ the exceptional set is empty and the defect is zero.
The defect also vanishes if
\begin{equation}\label{eq:tp-nodefect}
 \mu^2\ge (H_d^*)^2+(H_d^*)^{4/3}
                       +\frac53(\alpha+3/8)^2.
\end{equation}
Every statement remains valid with any larger cutoff $H$ in place
of $H_d^*$.
\end{proposition}

We prove this proposition after establishing the phase criteria below.

\subsection{Exact counting and the two phase estimates}

Write $H_\nu=J_\nu+iY_\nu$ and
$Q_{\alpha,\nu}(t)=tH'_\nu(t)-\alpha H_\nu(t)$, and choose its
continuous phase $\phi_{\alpha,\nu}$ with limit $\pi/2$ at zero.
Set $\phi_{\alpha,\nu}(0)=\pi/2$.
\begin{lemma}\label{lem:exact}
For $K>0$, $0\le\rho<1$, and each angular degree,
\begin{equation}\label{eq:ceiling}
 n_\ell^<(K^2)=\left\lceil
 \frac{\phi_{\alpha,\nu_\ell}(K)-\phi_{\alpha,\nu_\ell}(\mu)}\pi
 \right\rceil,
\end{equation}
and
\begin{equation}\label{eq:phiprime}
 \phi'_{\alpha,\nu}(t)
 =\frac{2(t^2-\nu^2+\alpha^2)}{\pi t|Q_{\alpha,\nu}(t)|^2},\qquad t>0.
\end{equation}
\end{lemma}
\begin{proof}
The Bessel equation and Wronskian give \eqref{eq:phiprime}.
First suppose $\rho>0$. The Neumann determinant vanishes when the
phase increment is in $\pi\mathbb Z$, and these crossings are upwards.
Indeed, normalize the inner Neumann
solution by $v(\rho,k)=1$, $v'(\rho,k)=0$.  At an eigenfrequency,
Green's identity gives
\[
 -v(1,k)\,\partial_kv'(1,k)
 =2k\int_\rho^1 s^{d-1}v(s,k)^2\,ds>0.
\]
Put $\Delta(k)=\phi_{\alpha,\nu}(k)-\phi_{\alpha,\nu}(\rho k)$.
The Bessel Wronskian gives
\[
 v'(1,k)=-\frac\pi2\rho^\alpha
       |Q_{\alpha,\nu}(\rho k)||Q_{\alpha,\nu}(k)|\sin\Delta(k).
\]
At $\Delta(k)=j\pi$, it also gives
$\operatorname{sgn}v(1,k)=(-1)^j$. Substitution into Green's
identity proves $\Delta'(k)>0$ at every crossing. For $\ell>0$ the
increment tends to zero from below as the frequency tends to zero;
its successive crossings are $0,\pi,2\pi,\ldots$.
For $\ell=0$ it is positive near zero, and the zero eigenvalue is already
present. This proves the ceiling formula for $\rho>0$, including
its left value at every positive eigenfrequency.

For $\rho=0$, use the regular solution
$v(s,k)=s^{-\alpha}J_\nu(ks)$. Its outer Neumann condition is
$\operatorname{Re}Q_{\alpha,\nu}(k)=0$, or
$\phi_{\alpha,\nu}(k)\in\pi/2+\pi\mathbb Z$.
The boundary term at zero vanishes, so the same Green identity holds
with the integral over $(0,1)$. At a crossing, the Wronskian gives
\[
 J_\nu(k)\operatorname{Im}Q_{\alpha,\nu}(k)=\frac2\pi.
\]
Since $\partial_kv'(1,k)=-|Q_{\alpha,\nu}(k)|
\sin\phi_{\alpha,\nu}(k)\,\phi'_{\alpha,\nu}(k)$ there,
Green's identity proves that every crossing is strictly upwards.
For $\ell>0$, the small-argument expansions put
$\phi_{\alpha,\nu_\ell}$ below $\pi/2$ near zero.
For $\ell=0$, the identity
$Q_{\alpha,\alpha}(k)=-kH_{\alpha+1}(k)$ puts it above $\pi/2$.
Thus
\[
 n_\ell^<(K^2)=\left\lceil
       \frac{\phi_{\alpha,\nu_\ell}(K)-\pi/2}{\pi}
                  \right\rceil.
\]
The initial values are zero in positive angular degrees and one in
degree zero, accounting for the constant eigenfunction. At a positive
eigenfrequency the ceiling again gives the left value. This is
\eqref{eq:ceiling} with $\mu=0$.
\end{proof}

The FLPS phase bounds \cite[Theorems~1.4 and~1.10]{FLPSphase}, with
$\theta_\nu=\arg H_\nu$ and $\varphi_\nu=\arg H'_\nu$ tending to
$-\pi/2$ and $\pi/2$ respectively, give, for $t>\nu$,
\begin{equation}\label{eq:flps}
 \theta_\nu(t)<S_\nu(t)-\frac\pi4,\qquad
 S_\nu(t)+\frac\pi4<\varphi_\nu(t)
 <S_\nu(t)+\frac\pi4+E_\nu(t),
 \qquad E_\nu(t)=\frac{9t^2-2\nu^2}{24(t^2-\nu^2)^{3/2}}.
\end{equation}
The upper bound for $\varphi_\nu$ follows from the majorant in
\cite[Theorem~1.10]{FLPSphase}. By \cite[Lemma~1.8]{FLPSphase},
$S_\nu+\pi/4+E_\nu$ decreases to its unique minimum and then increases.
The majorant replaces the part before this minimum by its minimum
value and agrees with $S_\nu+\pi/4+E_\nu$ afterwards, so it never
exceeds the expression used in \eqref{eq:flps}.
Nicholson's integral \cite[10.9.30]{DLMF} and the Wronskian give
$H'_\nu/H_\nu=a+ib$ with $a<0<b$.  Moreover, for $t>\nu$,
\[
 |H_\nu(t)|^2
 =\frac8{\pi^2}\int_0^\infty K_0(2t\sinh s)\cosh(2\nu s)\,ds
 <\frac8{\pi^2}\int_0^\infty K_0(2ts)\cosh(2\nu s)\,ds
 =\frac2{\pi\sqrt{t^2-\nu^2}}.
\]
The last integral follows by inserting the standard integral for $K_0$
and integrating in $s$ first.  Hence $b>\sqrt{t^2-\nu^2}/t$ and
\begin{align}
 \phi_{\alpha,\nu}&=\varphi_\nu+\delta_\alpha,\qquad
 \delta_0=0,\label{eq:robinphase}\\
 0<\delta_\alpha(t)<\arctan\frac{\alpha}{\sqrt{t^2-\nu^2}}
 &\le\frac{\alpha}{\sqrt{t^2-\nu^2}}\quad(t>\nu,\ \alpha>0).
 \label{eq:deltabound}
\end{align}
For $\alpha=0$, $\delta_\alpha=0$.  Indeed,
$\tan\delta_\alpha=(\alpha/t)b/(a^2-(\alpha/t)a+b^2)$.
Also $\phi_{\alpha,\nu}<\theta_\nu+\pi$ at every positive argument.

For the ball, the lower bound
$\phi_{\alpha,\nu}(K)>S_\nu(K)+\pi/4$ for $K>\nu$
and the exact ceiling formula give
\[
 n_\ell^<(K^2)\ge\lfloor G_K(\nu_\ell)+3/4\rfloor.
\]
If $K\le\nu_\ell$, the right side is zero. Thus the
three-quarter count holds in every degree when $\rho=0$, including
degree zero, and $N^<_{B_1}(K^2)\ge P_d(K,0)$ with no phase defect.

The same phase bounds give the universal estimate
\begin{equation}\label{eq:halfshift}
 n_\ell^<(K^2)\ge\lfloor A(\nu_\ell)+1/2\rfloor.
\end{equation}
For $t>\nu$, the preceding bounds give
$S_\nu+\pi/4<\phi_{\alpha,\nu}<S_\nu+3\pi/4$.
For $0<t\le\nu$, the Wronskian identity
$\theta_\nu'(t)=2/(\pi t|H_\nu(t)|^2)>0$ and the limiting upper
bound in \eqref{eq:flps} give $\theta_\nu(t)\le-\pi/4$.
Hence $\phi_{\alpha,\nu}(t)<\theta_\nu(t)+\pi\le3\pi/4$.
Subtracting these endpoint bounds in \eqref{eq:ceiling} proves
\eqref{eq:halfshift}. If $K\le\nu$, the desired floor is zero.

The following argument joins the phase bounds for $\nu>\alpha$.
If $t_* >\nu$
lies before the first positive zero of $tJ'_\nu-\alpha J_\nu$, and
\[
 E_\nu(t_*)+\arctan\frac\alpha{\sqrt{t_*^2-\nu^2}}<\frac\pi4,
\]
then
\begin{equation}\label{eq:cp-enclosure}
 \phi_{\alpha,\nu}(t)<\pi/2\ (t\le\nu),\qquad
 S_\nu(t)+\pi/4<\phi_{\alpha,\nu}(t)<S_\nu(t)+\pi/2\ (t>\nu).
\end{equation}
Before $t_*$ the first crossing has not occurred; afterwards both
error bounds decrease.  The same conclusion follows if the upper FLPS
estimate itself gives $\phi_{\alpha,\nu}(t_*)<\pi/2$, because
\eqref{eq:phiprime} has only one change of sign.  Subtracting endpoints
in \eqref{eq:cp-enclosure} proves \eqref{eq:threequarter}.

\subsection{First-crossing bounds and localization at the inner endpoint}

\begin{lemma}\label{lem:phase-upper-window}
For $L>0$ and $d\ge2$, the first ball eigenvalue in angular degree $L$ is greater than
$\mathfrak b_d(L)$. The same radial bound holds for real $d\ge2$ and $L>0$.
Consequently \eqref{eq:threequarter} holds if $\mu^2\le \mathfrak b_d(\ell)$ and $\ell>0$.

\end{lemma}
\begin{proof}
Choose
\[
 b=\frac{L^2}{2L+d+2},\qquad a=L-b,\qquad
 f(r)=r^L e^{-ar^2/2-br^4/4}.
\]
Then $f'(1)=0$, and the radial operator quotient, with $t=r^2$, is
\[
 V(t)=(2L+d)(a+bt)+2bt-t(a+bt)^2.
\]
It has $V''(t)=-4ab-6b^2t<0$ and $V(0)=V(1)=\mathfrak b_d(L)$.
Ground-state factorization proves the strict eigenvalue bound.
If $\mu^2\le \mathfrak b_d(\ell)$, then $\phi(\mu)<\pi/2$; the outer
FLPS lower bound and \eqref{eq:ceiling} prove the counting assertion.
\end{proof}

\begin{lemma}\label{lem:tp-localphase}
The three-quarter count also holds if $\mu>\nu_\ell$ and
\begin{equation}\label{eq:tp-errorcriterion}
 E_{\nu_\ell}(\mu)+\frac\alpha{\sqrt{\mu^2-\nu_\ell^2}}
 \le\frac\pi4.
\end{equation}
A sufficient condition is
\begin{equation}\label{eq:tp-simplephase}
 \mu^2\ge\nu_\ell^2+\nu_\ell^{4/3}
                  +\frac53\left(\alpha+\frac38\right)^2.
\end{equation}
\end{lemma}
\begin{proof}
The two endpoint phase bounds give an increment strictly greater
than $\pi A(\nu_\ell)-E_{\nu_\ell}(\mu)
-\alpha/\sqrt{\mu^2-\nu_\ell^2}$.
For the sufficient condition set $D=\mu^2-\nu^2$, $a=\alpha+3/8$
and $y=\sqrt{5/3}\,a/\sqrt D\le1$. Since
$\nu^{4/3}/D\le1-y^2$ and $(1-y^2)^{3/2}\le2(1-y)$,
\[
 \frac{7\nu^2}{24D^{3/2}}+\frac a{\sqrt D}
 \le\frac7{12}(1-y)+\sqrt{\frac35}\,y
 \le\sqrt{\frac35}<\frac{25}{32}<\frac\pi4.
\]
This proves \eqref{eq:tp-errorcriterion}, including $\nu=0$.
\end{proof}

\subsection{A uniform phase cutoff}

The next lemma is the only Airy estimate needed in the proof.
\begin{lemma}\label{lem:phase-new-airy}
The lowest eigenvalue of $-v''+xv$ on $(0,\infty)$ with
$v'(0)=-cv(0)$ is greater than $29/50$ when $0\le c\le37/100$.
\end{lemma}
\begin{proof}See Appendix~\ref{app:proof-lem-phase-new-airy}.\end{proof}

\begin{proposition}\label{prop:phase-new-cutoff}
If $\nu\ge19/2$ and
\begin{equation}\label{eq:phase-new-condition}
 \alpha+\frac12\le\frac{37}{100}
                     \bigl(2(\nu^2-1/4)\bigr)^{1/3},
\end{equation}
then \eqref{eq:cp-enclosure}, and consequently the three-quarter
count, hold.
\end{proposition}
\begin{proof}See Appendix~\ref{app:proof-prop-phase-new-cutoff}.\end{proof}

\begin{lemma}\label{lem:common-low-robin-phase}
For $0\le\alpha\le1$ and $\nu\ge\max\{1,5\alpha\}$,
the phase enclosure \eqref{eq:cp-enclosure} holds.
\end{lemma}
\begin{proof}
For $\nu\ge19/2$, Proposition~\ref{prop:phase-new-cutoff} applies.
For the remaining orders, the quartic Barta calculation in
Lemma~\ref{lem:phase-upper-window}, with $L=\nu-\alpha$ and
$d=2\alpha+2$, puts the first squared crossing frequency above $\nu^2+D$, where
\[
 D=3\nu-(\alpha^2+4\alpha+2)+\frac{(\alpha+2)^2}{\nu+2}
 >\frac{7\nu}{4}.
\]
Indeed, $\alpha\le\min\{1,\nu/5\}$ gives
\[
 \alpha^2\le\min\{1,\nu^2/25\}\le\frac{6\nu-5}{25},
 \qquad
 D-\frac{7\nu}{4}
 \ge\frac{21(\nu-1)^2+8(\nu-1)+7}{100(\nu+2)}>0.
\]

Thus $t_*=\sqrt{\nu^2+7\nu/4}$ precedes the first crossing. The sum
of the FLPS error bound and the linear Robin bound there is at most
\[
 \frac2{\sqrt7}\,
 \frac{\nu/6+\min\{1,\nu/5\}+3/8}{\sqrt\nu}
 \le\frac{53}{12\sqrt{35}}<\frac34<\frac\pi4.
\]
Indeed, on $[1,5]$ and $[5,19/2]$ the expression is convex as a
function of $\sqrt\nu$; its endpoint maximum is attained at $\nu=5$.
The last strict inequality follows from $53^2<81\cdot35$.
The common joining argument proves the result.
\end{proof}

\begin{proof}[Proof of Proposition~\ref{prop:phase}]
The ball case was proved after \eqref{eq:robinphase}; its exceptional
set is empty because $\mathfrak b_d(\ell)\ge0$.
Now suppose $\mu>0$. For each degree outside the exceptional set,
one of the three sufficient conditions above gives the three-quarter count.
In each remaining degree,
the half-shift and three-quarter floors differ by at most one. Summing
with the angular multiplicities proves the estimate. The strict
inequalities defining the exceptional set exclude the endpoints at which
a sufficient condition holds.
\end{proof}

\section{Variational estimates}
\label{sec:variational}

The phase comparison is most useful after a finite frequency threshold.
Explicit radial trials cover the complementary range and shells whose
dimensionless width $hK$ is small. The endpoint principle below makes
the same trials work for every $0\le\rho<1$, including the ball.

\subsection{The range covered by radial trials}
\begin{proposition}\label{prop:var-ranges}
For $d\ge4$, $K>0$ and $0\le\rho<1$, the strict Neumann
P\'olya inequality holds if
\[
 hK\le\frac{3\pi}{4}\sqrt{d-1},
 \qquad\hbox{or}\qquad
 K\le\begin{cases}12,&d=4,\\[2pt]
                  7d^{3/2}/4,&d\ge5.
       \end{cases}
\]
Moreover, for every $d\ge4$, $0\le\rho<1$ and
$K>\sqrt{\beta_\rho}/2$, it suffices that
\begin{equation}\label{eq:var-thin-exact}
 hK\le B_d\left(1-\frac{\sqrt{\beta_\rho}}{2K}\right)^{d-1},\quad
 B_d=\frac{2^{d+1}\Gamma(d/2+1)^2}{d!},\quad
 \beta_\rho=\frac d{d-2}\frac{1-\rho^{d-2}}{1-\rho^d}.
\end{equation}
In particular, for $d=4$ the sufficient condition is
$hK\le(16/3)(1-\sqrt2/(2K))^3$ whenever $K>\sqrt2/2$.
\end{proposition}
\begin{proof}See Appendix~\ref{app:proof-prop-var-ranges}.\end{proof}

For $d\ge5$, the proof has three steps. The constant trial proves the width
criterion and the initial range $K\le(27/40)d^{3/2}$. A polynomial
ground trial then reaches $K_1=(3/2)d^{3/2}$, using
Lemma~\ref{lem:var-trial-ratio-unimodal} to reduce an interval to
its endpoints. Between $K_1$ and $K_2=(7/4)d^{3/2}$, the ground
and first excited rows contribute more than $451/500$ and $49/500$
of the ball Weyl term. Their sum exceeds one. The endpoint estimates
are proved in Appendix~\ref{app:variational-details}; the first excited
trial reduces to one scalar concavity calculation.

\subsection{Radial endpoint comparison and trial spaces}
\begin{lemma}\label{lem:var-radial}
For $d\ge2$, $q\ge0$, $j\ge0$ and $0<r<1$,
\begin{align}
 \lambda_{q,j}(\rho)&\le
 \max\{\lambda_{q,j}(0),\lambda_{q,j}(r)\},\quad0\le\rho\le r,
 \label{eq:var-endpoint}\\
 \lambda_{q,0}(\rho)&\le\lambda_{q,0}(0),
 \label{eq:var-ground}\\
 \lambda_{q,1}(r)&<\frac{\pi^2}{(1-r)^2}
                 +\frac{q+(d^2-1)/4}{r^2}.
 \label{eq:var-excited}
\end{align}
The inequality in \eqref{eq:var-ground} is strict when $q>0$ and
$\rho>0$.
\end{lemma}
\begin{proof}
Simplicity and differentiation of the radial variational identity give
\[
 \lambda'_{q,j}(\rho)=a(\rho)
 \left(\lambda_{q,j}(\rho)-\frac q{\rho^2}\right),\qquad
 a(\rho)=\frac{\rho^{d-1}u(\rho)^2}
 {\int_\rho^1s^{d-1}u(s)^2\,ds}>0.
\]
For $q>0$, $g=\lambda_{q,j}-q/\rho^2$ satisfies
$g'=ag+2q/\rho^3$. An integrating factor shows that $g$ can change
sign only from negative to positive. Thus $\lambda_{q,j}$ has no
interior maximum. For $q=0$, positive levels increase and the zero
level is constant. Restricting the first $j+1$ ball eigenfunctions
to $(\varepsilon,1)$ gives
$\limsup_{\varepsilon\downarrow0}\lambda_{q,j}(\varepsilon)
\le\lambda_{q,j}(0)$, proving \eqref{eq:var-endpoint}.
The constant trial gives $\lambda_{q,0}(\rho)<q/\rho^2$ for $q>0$,
so the ground level decreases strictly.
Finally, conjugation by $s^{(d-1)/2}$ on $(r,1)$ writes the
zero-angular radial operator as $\mathcal D^*\mathcal D$, where
$\mathcal D=d/ds-(d-1)/(2s)$ and $\mathcal Dv=0$ at both endpoints.
Its positive spectrum is the Dirichlet spectrum of
$\mathcal D\mathcal D^*=-d^2/ds^2+(d^2-1)/(4s^2)$.
The first Dirichlet sine trial bounds $\lambda_{0,1}(r)$ strictly by
the right-hand side of \eqref{eq:var-excited} with $q=0$.
Since $q/s^2\le q/r^2$ on $(r,1)$, the min--max principle gives
$\lambda_{q,1}(r)\le\lambda_{0,1}(r)+q/r^2$, proving
\eqref{eq:var-excited}.
\end{proof}

\begin{lemma}[A uniform trial in dimension four]\label{lem:d4-ground-direct}
In dimension four the lowest radial levels alone give
$N^<(K^2)>W_4(K)$ for every $0<K\le12$ and every $0\le\rho<1$.
\end{lemma}
\begin{proof}
In angular degree $\ell\ge1$, use the ball trial
\[
 u_\ell(s)=s^\ell\left(1-\frac{\ell+3}{\ell+5}s^2\right).
\]
Since $u_\ell'(1)=-6/(\ell+5)\ne0$, this trial is not a
Neumann eigenfunction. Its Rayleigh quotient therefore bounds the
ball ground eigenvalue strictly, and the ground-level comparison gives
\[
 \lambda_{q_\ell,0}(\rho)<Q_\ell,
 \qquad Q_\ell=\frac{6(\ell+1)(\ell+2)(\ell+4)}{5\ell+19}.
\]
Indeed, the value of the radial form (derivative and angular terms) and
the squared norm of this trial are $6(\ell+1)/(\ell+5)^2$ and
$(5\ell+19)/[(\ell+2)(\ell+4)(\ell+5)^2]$, respectively.
The sequence $Q_\ell$ increases. The cumulative multiplicity through
degree $\ell-1$ is $\ell(\ell+1)(2\ell+1)/6$. For $1\le\ell\le9$,
\begin{equation}\label{eq:d4-ground-threshold}
 Q_\ell^2<\frac{64\ell(\ell+1)(2\ell+1)}6.
\end{equation}
To verify the comparison, clear positive factors and put
$y=\ell-1\in[0,8]$. The numerator becomes
\[
 1674+9765y+y^2(7860+1438y-86y^2-27y^3)
 \ge1674+9765y+36y^2>0.
\]
Here $1438-86y-27y^2\ge-978$ and $7860-978y\ge36$.
Put $Q_0=0$. On each interval
$Q_{\ell-1}\le K^2\le Q_\ell$, the degrees $0,\ldots,\ell-1$
and \eqref{eq:d4-ground-threshold} prove the claim. Finally,
$Q_9=2145/16<144$ and the degrees $0,\ldots,9$ contribute
$385>12^4/64$, covering the remaining interval.
\end{proof}

\subsection{Uniform trial counts}
\begin{lemma}[Uniform endpoint estimates]\label{lem:var-uniform-endpoints}
Let $d\ge5$ be an integer, $p=d-1$, $\alpha=(d-2)/2$, and put
\[
 \mathfrak p_d=53d^2-74d+24,\qquad
 a_d=\frac{(5d-6)(53d^2-46d+8)}{5d\,\mathfrak p_d},\qquad
 b_d=\frac{32d^2(d-1)^2}{\mathfrak p_d}.
\]
Define
\[
 G_d(v)=\frac{2v}{p!}\prod_{i=1}^{p-1}(v+i-p/2),\qquad
 W_d(K)=\frac{K^d}{2^d\Gamma(d/2+1)^2},
\]
\[
 K_1=\frac32d^{3/2},\quad K_2=\frac74d^{3/2},\quad
 v_0(K)=\sqrt{a_dK^2-b_d+\alpha^2}-\frac12,
\]
\[
 v_1(K)=\sqrt{\left(1-\frac9{5p}\right)
       \left(K-\frac{47}{20}\sqrt p\right)^2-d+\frac54}-\frac12.
\]
Then
\begin{equation}\label{eq:var-uniform-endpoints}
 \frac{G_d(v_0(K_1))}{W_d(K_1)}>1,\qquad
 \frac{G_d(v_0(K_2))}{W_d(K_2)}>\frac{451}{500},\qquad
 \frac{G_d(v_1(K_1))}{W_d(K_1)}>\frac{49}{500}.
\end{equation}
\end{lemma}
\begin{proof}See Appendix~\ref{app:var-endpoints}.\end{proof}

\begin{lemma}[One maximum for a radial trial count]
\label{lem:var-trial-ratio-unimodal}
Let $d\ge5$, $\alpha=(d-2)/2$, $a>0$ and $b>\alpha^2$.
For $K>\sqrt{b/a}$ define
\[
 v(K)=\sqrt{aK^2-b+\alpha^2}-\frac12,
 \qquad F(K)=\frac{G_d(v(K))}{W_d(K)}.
\]
Then $F$ has precisely one critical point, which is a maximum.
In particular, on every compact subinterval of this range,
$F$ is bounded below by the smaller of its two endpoint values.
\end{lemma}
\begin{proof}
Set $c=b-\alpha^2>0$ and
$x=\sqrt{aK^2-c}-\alpha>0$.  Up to a positive constant,
\[
 F(K)=\frac{Q(x)}{H(x)^{d/2}},\qquad
 Q(x)=\left(x+\frac{d-3}{2}\right)\prod_{j=0}^{d-3}(x+j),
 \quad H(x)=x^2+2\alpha x+b.
\]
Let $Z$ be the multiset $\{0,1,\ldots,d-3,(d-3)/2\}$.
The logarithmic derivative of this quotient has the sign of
\[
 J(x)=H(x)\sum_{z\in Z}\frac1{x+z}-d(x+\alpha).
\]
Since $H(x)=(x+\alpha)^2+c$ and $|Z|=d-1$,
\[
 J'(x)=-1-\sum_{z\in Z}
          \frac{(\alpha-z)^2+c}{(x+z)^2}<0.
\]
Moreover $J(x)\to+\infty$ as $x\downarrow0$, and
$J(x)\sim-x\to-\infty$ as $x\to\infty$.
Hence $J$ has one zero, proving the assertion.
\end{proof}

\begin{lemma}[An elementary low-frequency range]
\label{lem:constant-low-frequency}
For every integer $d\ge5$ and every $0\le\rho<1$, the lowest
radial eigenvalues alone give the strict Neumann P\'olya inequality whenever
\[
 0<K\le K_*(d):=\frac{27}{40}d^{3/2}.
\]
More precisely, their counting function for the unit ball is greater than
$W_d(K)=K^d/[2^d\Gamma(d/2+1)^2]$ throughout this range.
\end{lemma}
\begin{proof}See Appendix~\ref{app:constant-bridge}.\end{proof}

\section{Proof of the strict inequality}\label{sec:dimensions}
We begin with the clipped angular weight for $d\ge5$, then treat
$d=3,4$ with quadratic and cubic weights. The planar argument in
Section~\ref{sec:planar} uses semicircle subtraction and concavity of
the inverse action. In every dimension, radial trials cover low
frequencies and small widths, while the comparison covers the remaining
range. The resulting strict inequalities include $\rho=0$ in every dimension.

Table~\ref{tab:roadmap} records this division of the parameter range.
In its four-dimensional row, put
$s(K)=\frac{16}{3}(1-1/(\sqrt2K))^3$.

\begin{table}[htbp]
\centering\small
\renewcommand{\arraystretch}{1.25}
\begin{tabular}{@{}c >{\raggedright\arraybackslash}p{0.33\textwidth}
                       >{\raggedright\arraybackslash}p{0.47\textwidth}@{}}
\toprule
$d$ & radial trial spaces & weighted shell--ball comparison\\
\midrule
$2$ & $K\le8$ or $T\le3/2$ (Lemma~\ref{d2:lem:variational})
    & semicircle comparison and width endpoints
      (Theorem~\ref{thm:planar-complete})\\
$3$ & $K\le44/5$ or $T\le4/3$ (Proposition~\ref{d3:thm:low})
    & quadratic $\Phi_2$, explicit degree-zero phase
      (Proposition~\ref{d3:prop:jointclosure})\\
$4$ & $K\le12$ (Lemma~\ref{lem:d4-ground-direct}) or $(1-\rho)K\le s(K)$
      (Proposition~\ref{prop:var-ranges})
    & cubic $\Phi_3$, clipping gain, localized phase defect
      (Theorem~\ref{thm:d4all})\\
$\ge5$ & $K\le\frac74d^{3/2}$ or $T\le\frac34\sqrt{d-1}$
      (Proposition~\ref{prop:var-ranges})
    & one clipped angular polynomial
      (Theorem~\ref{thm:clipped-surplus})\\
\bottomrule
\end{tabular}
\caption{Coverage in the proof of Theorem~\ref{thm:all-dimensions}.
Every comparison for $d\ge3$ uses Proposition~\ref{prop:physical-transfer}.
The phase input is Proposition~\ref{prop:phase}, with the refinements
stated in dimensions three and four. The ball inputs are
Lemma~\ref{lem:onepoint-ball} for $d\ge5$ and
Lemmas~\ref{d3:lem:ballpoly} and~\ref{lem:d4ball} for $d=3,4$.}
\label{tab:roadmap}
\end{table}

\subsection{Dimensions at least five: the clipped angular comparison}
\label{sec:clipped-comparison}
We apply the counting reduction to the single weight $2\Psi_p/p!$.
Expanding $R_p$ in the ball comparison retains the lower-power volume
terms without introducing further shell slope losses.

\begin{theorem}\label{thm:clipped-surplus}
Let $d\ge5$, $p=d-1$, $0\le\rho<1$, and $K>0$.
The strict Neumann P\'olya inequality holds on $\Omega_{\rho,1}$.
Moreover, if
\begin{equation}\label{eq:clipped-ranges}
 K\ge K_0:=\frac74(p+1)^{3/2},\qquad
 T=\frac{(1-\rho)K}{\pi}\ge\frac34\sqrt p,
\end{equation}
then
\begin{equation}\label{eq:clipped-surplus}
 \frac{p!}{2K^p}
 \bigl[N^<(K^2)-\W_{p+1}(K,\rho K)\bigr]>\frac1{2000}.
\end{equation}
\end{theorem}
\begin{proof}
The variational estimates in Proposition~\ref{prop:var-ranges}
cover the complement of \eqref{eq:clipped-ranges}. We prove the
surplus in the remaining range. Put
\[
 \mathcal S_p=\frac{p!}{2K^p}
    \bigl[N^<(K^2)-\W_{p+1}(K,\mu)\bigr],\qquad
 \Delta=\frac{p!D_{\mathrm{ph}}}{2K^p},
\]
where $D_{\mathrm{ph}}$ is the defect from Proposition~\ref{prop:phase}.
Proposition~\ref{prop:physical-transfer},
Lemma~\ref{lem:clipped-angular}, and
\eqref{eq:clipped-ball-reduction} give
\begin{equation}\label{eq:clipped-power-comparison}
 \begin{split}
 \mathcal S_p\ge{}&\frac{E_{B,p}}{K^p}
 -\frac{I_p}{2\pi}(1-\rho^p)
 -\frac{b_pI_{p-1}}{\pi(p-1)K}(1-\rho^{p-1})\\
 &-\frac p{8K}-\frac{b_p}{4K^2}
 -\frac{M_{\Psi_p}}{8K^p}-\frac{J_C(\mu)}{K^p}-\Delta.
 \end{split}
\end{equation}
Indeed, the moment identity evaluates the three polynomial integrals,
and Lemma~\ref{lem:exactquarter} bounds the ball sampling errors of
the subtracted powers $p-1$ and $p-2$ from above by $K^{p-1}/4$
and $K^{p-2}/4$. The transfer loss involves only the slope
of the clipped weight $\Psi_p$.

\paragraph{The coefficient bounds for $p\ge5$.}
Set $X=\rho^{p/2}$. Using $M_{\Psi_p}\le M_p$,
\eqref{eq:Mbound}, $\rho^{p-1}\ge X^2$, and
\eqref{eq:clipping-correction}, we obtain
\begin{equation}\label{eq:clipped-quadratic}
 \mathcal S_p\ge\frac{E_{B,p}}{K^p}
       -a_p(1-X^2)-c_pX-r_p-\Delta,
\end{equation}
where
\begin{align*}
 a_p&=\frac{I_p}{2\pi}
       +\frac{p(p-5)I_{p-1}}{24\pi K_0},
 &c_p&=\frac{h_p}{6\sqrt p},\\
 r_p&=R(p+1)+\frac1{200000},
 &R(v)&=\frac{v-1}{14v^{3/2}}
       +\frac{(v-1)(v-2)(v-6)}{294v^3}.
\end{align*}
The following entries are strict upper bounds:
\begin{equation}\label{eq:clipped-coefficients}
 \begin{array}{c|ccc}
  &a_p&c_p&r_p\\\hline
  p=5,6&17/200&49/400&1/40\\
  p\ge7&19/250&4/35&1/44.
 \end{array}
\end{equation}

For the first coefficient, the Wallis recurrence and
$\pi>157/50$ give
\[
 a_5<\frac{40}{471}<\frac{17}{200},\qquad
 a_6<\frac5{64}+\frac1{720}<\frac{17}{200},
\]
where $K_0>32$ was used at $p=6$.
For $p\ge7$, the Wallis identity
$I_{p+1}=pI_{p-1}/(p+1)<I_p$ implies
\[
 a_p<A_p:=\frac{I_p}{2\pi}g(p),\qquad
 g(p)=1+\frac{p-5}{21\sqrt{p+1}}.
\]
The function
\[
 \frac{g(p)}{\sqrt{p+1}}
 =\frac1{\sqrt{p+1}}+\frac1{21}-\frac2{7(p+1)}
\]
decreases for $p\ge7$. A second application of the Wallis recurrence gives
\[
 \frac{A_{p+2}}{A_p}
 =\frac{p+1}{p+2}\frac{g(p+2)}{g(p)}
 <\frac{\sqrt{(p+1)(p+3)}}{p+2}<1.
\]
It remains to check the first odd and even values. Using
$\pi>25/8$ and $\sqrt2>7/5$,
\[
 A_7<\frac{64}{875}\left(1+\frac5{147}\right)<\frac{19}{250},
 \qquad A_8=\frac{55}{768}<\frac{19}{250}.
\]

For the second coefficient, the decreasing ratio $h_p/\sqrt p$
from Corollary~\ref{lem:radiusM} gives
$c_5=\sqrt{15}/32<49/400$ and $c_6<c_5$.
Also $h_7^2/7=21875/46656<(24/35)^2$, which proves the second row.
For the last coefficient, $v\ge6$ implies $v-3\ge\sqrt v$, and hence
\[
 R'(v)=-\frac{v-3}{28v^{5/2}}
       +\frac{9v^2-40v+36}{294v^4}
       <-\frac1{196v^2}<0.
\]
The bounds $\sqrt6>12/5$ and $\sqrt2>24/17$ give
\[
 R(6)<\frac{25}{1008},\qquad
 R(8)<\frac{17}{768}+\frac1{1792}=\frac{61}{2688}.
\]
Adding $1/200000$ proves the claimed bounds $1/40$ and $1/44$.

\paragraph{The phase defect for $p\ge5$.}
The cutoff envelope gives $H_d^*\le(9/8)p^{3/2}$,
since $(9/8)5^{3/2}>19/2$.
As $(9/8)^{4/3}<6/5$ and $\alpha+3/8=p/2-1/8$, the
sufficient threshold for a vanishing defect is less than
\[
 \frac{81}{64}p^3+\frac{97}{60}p^2
 <\frac{49}{36}(p+1)^3.
\]
The difference is
$55p^3/576+37p^2/15+49p/12+49/36>0$.
Consequently $\Delta=0$ whenever $\rho\ge2/3$.
For every $p\ge5$ and every radius, we also have
\begin{equation}\label{eq:clipped-phase-bounds}
 \Delta<\frac{27}{1000}.
\end{equation}
At $p=5$, the cutoff $H_6^*<13$ permits only degrees zero through
ten, whose cumulative multiplicity is
$\binom{15}{5}+\binom{14}{5}=5005$. Hence
\[
 \Delta\le\frac{300300}{K_0^5}<\frac{27}{1000},
 \qquad K_0^2=\frac{1323}{2}.
\]
For $p\ge6$, the cutoff bound and $p^{3/2}>10$ imply
$H_d^*+1/2<(47/40)p^{3/2}$. The angular upper bound gives
\[
 \Delta<\left(\frac{47}{70}\right)^p
       \left(\frac p{p+1}\right)^{3p/2}
 <\frac4{15}\left(\frac{47}{70}\right)^6
 <\frac{27}{1000}.
\]
Here $(p/(p+1))^p$ decreases with $p$, since its logarithmic
derivative is $\log(p/(p+1))+1/(p+1)<0$, and $(6/7)^9<4/15$.

\paragraph{Completing the comparison for $p\ge5$.}
Let $\bar a,\bar c,\bar r$ be the corresponding row of
\eqref{eq:clipped-coefficients}. For $\rho\ge2/3$, the phase
defect vanishes, and completing the square bounds the loss in
\eqref{eq:clipped-quadratic} by
$\bar a+\bar c^2/(4\bar a)+\bar r$.

For $\rho\le2/3$, use $T\ge K/(3\pi)$ in the slope term of
\eqref{eq:clipped-power-comparison}.
Both $h_p/\sqrt p$ and $\sqrt p/(p+1)^{3/2}$ decrease for $p\ge5$;
the logarithmic derivative of the latter is
$(1-2p)/(2p(p+1))<0$. Thus the coefficient of $X=\rho^{p/2}$
in the slope loss is bounded uniformly by
\[
 \frac{3\pi h_p}{8K_0}
 \le\frac{15\pi}{448\sqrt2}
 <\frac{935}{12544}<\frac3{40},
\]
using $\pi<22/7$ and $\sqrt2>24/17$.
Also $X\le(2/3)^{5/2}<3/8$.
Since $\bar a\le17/200$, the quadratic
$\bar a(1-X^2)+(3/40)X$ increases on $[0,3/8]$.
Together with \eqref{eq:clipped-phase-bounds}, this bounds the loss by
$55\bar a/64+9/320+\bar r+27/1000$.

Substitution of the two coefficient rows gives
\[
 \max\left\{
 \bar a+\frac{\bar c^2}{4\bar a}+\bar r,
 \frac{55\bar a}{64}+\frac9{320}+\bar r+\frac{27}{1000}
 \right\}
 <\begin{cases}
 309/2000,&p=5,6,\\
 287/2000,&p\ge7.
 \end{cases}
\]
The respective ball lower bounds $31/200$ and $18/125$ from
Lemma~\ref{lem:onepoint-ball} therefore give
$\mathcal S_p>1/2000$ in every case $p\ge5$.

\paragraph{Dimension five: the two-power comparison.}
Now $p=4$, $b_4=0$, $K_0=35\sqrt5/4$, and $T\ge3/2$.
Retain the derivative correction in the clipped weight. If $\mu>2$,
then, for $2<z<\mu$,
\[
 \Psi_4'(z)=4z^3-6z^2
 \le\left(1-\frac3{2\mu}\right)4z^3,
 \qquad
 M_{\Psi_4}\le\left(1-\frac3{2\mu}\right)M_4.
\]
For $\mu\le2$, $M_{\Psi_4}=0$.
Since $I_4=3\pi/16$, the counting reduction gives
\begin{equation}\label{eq:clipped-five-loss}
 \mathcal S_4\ge\frac{E_{B,4}}{K^4}-\mathcal L,
 \qquad
 \mathcal L=\frac3{32}(1-\rho^4)+\frac1{2K}
       +\frac{M_{\Psi_4}}{8K^4}+\frac1{100000}+\Delta.
\end{equation}

The cutoff $H_5^*=19/2$ permits only degrees $0,\ldots,7$.
If $\mu\le8$, the lower condition in the phase window excludes
degrees six and seven because $\mathfrak b_5(6)>66>64$;
the remaining cumulative multiplicity is $196$.
If $\mu\ge8$, the upper condition excludes all orders at most
$13/2$, since
\[
 (13/2)^2+(13/2)^{4/3}+375/64<169/4+13+6<64.
\]
The remaining degrees six and seven have multiplicities $140$ and
$204$. If $\mu\ge10$, the largest order $17/2$ is also excluded:
\[
 (17/2)^2+(17/2)^{4/3}+375/64<289/4+18+6<100.
\]
Thus $\Delta\le4128/K^4$, and $\Delta=0$ for $\mu\ge10$.

Suppose first that $\mu<10$. The derivative bound gives
$M_{\Psi_4}\le(17/20)M_4$, including $\mu\le2$.
The sharper estimate in \eqref{eq:Mbound}, together with the width
relation, gives
\[
 \frac{M_4}{8K^4}
 \le\frac\pi{\sqrt3K}\frac{\rho^3}{1-\rho^3}
 <\frac1{10}\frac{\rho^3}{1-\rho^3}
 \quad(\rho>0).
\]
Since $K\ge K_0>39/2$, one has
$\rho<20/39<8/15$ and $\rho^3<1/6$. Hence
\[
 \frac{M_{\Psi_4}}{8K^4}
 \le\frac{17}{200}\frac{\rho^3}{1-\rho^3}
 \le\frac{34}{625}\rho^2\le\frac3{50}\rho^2.
\]
At $\rho=0$, both slopes vanish.
Also $1/(2K)+1/100000<13/500$ and $\Delta<3/100$.
Completing the square in $\rho^2$ therefore gives
\[
 \mathcal L<\frac3{32}
   +\frac{(3/50)^2}{4(3/32)}+\frac7{125}
 =\frac{3187}{20000}<\frac4{25}.
\]

For $\mu\ge10$, the phase defect vanishes. Put
$S=1+\rho+\rho^2$. The width condition and derivative correction yield
\[
 \frac1{2K}+\frac{M_{\Psi_4}}{8K^4}
 \le\frac{2\rho^3}{3\sqrt3S}
       +\frac1K\left(\frac12-\frac{\rho^2}{\sqrt3S}\right).
\]
The coefficient of $1/K$ is positive, because $S\ge3\rho^2$.
Replace $K$ by $K_0$, then use $3\rho\le S\le3$, to obtain
\[
 \mathcal L\le\frac3{32}(1-\rho^4)
 +\frac1{\sqrt3}\left(\frac29-\frac1{3K_0}\right)\rho^2
 +\frac1{2K_0}+\frac1{100000}.
\]
Since $K_0<20$ and $\sqrt3>12/7$, the coefficient of $\rho^2$
is less than $259/2160<3/25$. Also
$1/(2K_0)+1/100000<13/500$. A second completion of the square gives
\[
 \mathcal L<\frac3{32}
   +\frac{(3/25)^2}{4(3/32)}+\frac{13}{500}
 =\frac{3163}{20000}<\frac4{25}.
\]
In both cases Lemma~\ref{lem:onepoint-ball} gives
$\mathcal S_4>161/1000-4/25=1/1000>1/2000$.
This proves \eqref{eq:clipped-surplus} in the range
\eqref{eq:clipped-ranges}; Proposition~\ref{prop:var-ranges}
covers all remaining frequencies and widths.
\end{proof}

\subsection{Angular weights and action estimates for dimensions three and four}
\label{sec:lowdim-tools}
We collect the angular weights and the inverse-action and tail estimates
used in the remaining two dimensions.

\subsubsection{The quadratic and cubic angular weights}
For dimensions three and four, use the full angular polynomials
\begin{equation}\label{eq:jp-poly}
 \Phi_2(z)=(z-1/2)_+^2,\qquad
 \Phi_3(z)=
 \begin{cases}
 0,&0\le z\le1,\\
 z(z-1/2)(z-1)/3,&z>1.
 \end{cases}
\end{equation}
Here $p=2,3$ and $\alpha=(p-1)/2$.

\begin{lemma}\label{lem:jp-poly}
For $p=2,3$, one has $C_{p+1}\ge\Phi_p$.
The weight $\Phi_p$ satisfies the hypotheses of
Lemma~\ref{lem:general-transfer}, and
\begin{equation}\label{eq:jp-derivative}
 0\le\Phi_p'(z)\le\frac{2}{(p-1)!}z^{p-1}
 \quad\text{for almost every }z\ge0,
 \qquad 0\le\Phi_p(z)\le\frac2{p!}z^p.
\end{equation}
\end{lemma}
\begin{proof}
In \eqref{eq:angular-exact}, substitute $u\ge z-1/2$ into the
increasing product for $z\ge\alpha$; below $\alpha$, both lower
weights vanish. The displayed formulas show convexity above the
clipping point, and the derivative jumps there are nonnegative.
For $\Phi_2$, the nonzero derivative has its root at $1/2$.
For $\Phi_3$, the derivative is $z^2-z+1/6$, with roots
$(1\pm1/\sqrt3)/2$ in $(0,1)$.
In each case all derivative roots $r_i$ are positive and at or below the
clipping point, so
\[
 \frac{d}{dz}\left(\frac{z\Phi_p''(z)}{\Phi_p'(z)}-1\right)
 =-\sum_i\frac{r_i}{(z-r_i)^2}<0\qquad(z>\alpha).
\]
The one-sided inner slope is finite since $\alpha>0$.
The remaining bounds follow directly from \eqref{eq:jp-poly}.
\end{proof}

For $p=2,3$, retain the notation
\[
 Q_{B,p}(K)=Q_{\Phi_p}(K),\qquad J_p(L)=J_{\Phi_p}(L).
\]
Integration by parts also gives
$J_p(L)=\int_0^L\Phi_p'(z)G_L(z)\,dz$.
When $\mu>\alpha$, let
$M_{\Phi_p}=\sup_{\alpha<z<\mu}\Phi_p'(z)/(zH(z^2))$;
otherwise set $M_{\Phi_p}=0$.
For $\mu>\alpha$, Lemma~\ref{lem:general-transfer} gives
\begin{equation}\label{eq:jp-Jensen}
 M_{\Phi_p}\le\frac{\mu K}{T}
 \max_{\alpha\le z\le\mu}
 \frac{\Phi_p'(z)}z\sqrt{1-\beta_*z^2},
\end{equation}
where $\beta_*$ is defined in \eqref{eq:general-Jensen}.
One-sided derivatives are understood at the endpoints.
Proposition~\ref{prop:physical-transfer} now applies with the
nonnegative inner-volume term and the quantities $Q_{B,p}$, $J_p$,
and $M_{\Phi_p}$.

\subsubsection{Inverse-action and tail estimates}
The ball estimates require bounds on the inverse action over the
first sampling period and a positive part of the later tail.

\begin{lemma}\label{lem:inverse-action-seven-tenths}
For $0\le t\le K/\pi$, put $a_K=(3\pi t/K)^{2/3}/2$. Then
$\xi_K(t)/K\ge1-a_K$. If $0\le a_K\le7/10$, then
\begin{equation}\label{eq:inverse-action-wide}
 1-a_K+\frac{a_K^2}{30}\le\frac{\xi_K(t)}K
                  \le1-a_K+\frac{a_K^2}{25}.
\end{equation}
For $0\le t\le5/4$, the hypothesis holds whenever $K\ge73/10$.
\end{lemma}
\begin{proof}See Appendix~\ref{app:proof-lem-inverse-action-seven-tenths}.\end{proof}

\paragraph{The retained tail.}
\label{sec:jp-tail}
The following estimate requires only three derivatives in the interior
of the support and includes the terminal derivative jump.

\begin{lemma}\label{lem:jp-tail}
Let $f$ be continuous and vanish on $[L,\infty)$, with $f(L)=0$.
Suppose $f$ is three times continuously differentiable on $(0,L)$,
$f''\ge0$, and $f'''\le0$ there. Write $c=-f'(L-)\ge0$.
For $0<a<L$, define
\[
 \mathcal R_f(a)=\frac12f(a)+\sum_{j\ge1}f(a+j)
                         -\int_a^\infty f(t)\,dt.
\]
Then
\begin{equation}\label{eq:jp-tail}
 \mathcal R_f(a)\ge -\frac{f'(a)}{12}
       -\frac{f''(a)}{72\sqrt3}
       -\frac c{12}.
\end{equation}
\end{lemma}
\begin{proof}
Put $\widetilde f(t)=f(t)-c(L-t)_+$. The derivative of $\widetilde f$ is continuous at
$L$, and $\widetilde f''$, extended by zero after $L$, is nonnegative and
nonincreasing. The trapezoidal identity on each unit interval gives
\[
 \mathcal R_{\widetilde f}(a)=\int_0^\infty w(\{s\})\widetilde f''(a+s)\,ds,
 \qquad w(s)=\frac{s(1-s)}2.
\]
The mean of $w$ on $[0,1]$ is $1/12$. Its centered primitive is
\[
 B(s)=\int_0^s\left(w(v)-\frac1{12}\right)dv
     =\frac{s(1-s)(2s-1)}{12},
 \qquad |B(s)|\le\frac1{72\sqrt3}.
\]
Extend $B$ periodically. Stieltjes integration by parts against the
nonincreasing function $\widetilde f''(a+s)$ therefore yields
\[
 \mathcal R_{\widetilde f}(a)\ge-\frac{\widetilde f'(a)}{12}
                       -\frac{\widetilde f''(a)}{72\sqrt3}.
\]
The function $(L-t)_+$ is convex, so its trapezoidal error is
nonnegative. Since $\widetilde f'(a)=f'(a)+c$ and $\widetilde f''(a)=f''(a)$,
the claimed bound follows.
\end{proof}

For $p=2,3$, put $f(t)=\Phi_p(\xi_K(t))$ and $L=G_K(\alpha)$,
and suppose $K>\alpha$. The required derivative signs hold on $(0,L)$.
Indeed, the first three derivatives of the polynomial $\Phi_p$ are
nonnegative above $\alpha$, whereas, with $u=\arccos(\xi_K(t)/K)$,
\[
 \xi_K'=-\frac\pi u,\qquad
 \xi_K''=\frac{\pi^2}{Ku^3\sin u},\qquad
 \xi_K'''=-\frac{\pi^3(3\sin u+u\cos u)}{K^2u^5\sin^3u}.
\]
The chain rule consequently gives $f''\ge0$ and $f'''\le0$.
Its terminal linear coefficient is explicit:
\begin{equation}\label{eq:jp-terminal}
 c=\frac{\pi\Phi_p'(\alpha+)}{\arccos(\alpha/K)}
  =\begin{cases}0,&p=2,\\
  \pi/[6\arccos(1/K)],&p=3.\end{cases}
\end{equation}

For $f=\Phi_p\circ\xi_K$, the first-period decomposition is exact:
\[
 E_B(\Phi_p)=f(1/4)+\tfrac12f(5/4)-\int_0^{5/4}f(t)\,dt
                         +\mathcal R_f(5/4).
\]
\subsection{Dimension three: the quadratic comparison}
\label{app:dimension-three}
Use the common normalized parameters $0\le\rho<1$, $K$,
$\mu=\rho K$ and $T=(K-\mu)/\pi$. Write $N^<$ for the strict
count in dimension three, and put
\begin{equation}\label{d3:eq:notation}
 h=1-\rho,\qquad D=1+\rho+\rho^2,\qquad X^2=DK^2/3.
\end{equation}
\begin{equation}\label{d3:eq:weyl}
 \W_3(K,\mu)=\frac{2(K^3-\mu^3)}{9\pi}.
\end{equation}
Equivalently,
\begin{equation}\label{d3:eq:W-X}
 \W_3(K,\mu)=\frac{2T}{3}X^2.
\end{equation}
\begin{theorem}\label{d3:thm:main}
For every $0\le\rho<1$ and $K>0$,
\begin{equation}\label{d3:eq:main}
 N^<(K^2)>\W_3(K,\rho K).
\end{equation}
Both sides vanish when $K=0$.
\end{theorem}

We combine the common radial phase estimates with the comparison for
the full angular polynomial and elementary trial functions. The trial functions cover
$T\le4/3$ or $K\le44/5$; the joint estimate covers
$T\ge4/3$ and $K\ge44/5$.

\subsubsection{Trial functions and the elementary ranges}
In angular degree $\ell$, write $\lambda_{\ell,0}$ for the first
radial eigenvalue; its multiplicity is $2\ell+1$ and
$\lambda_{0,0}=0$. The constant trial and the ball power trial
$s^a$, together with Lemma~\ref{lem:var-radial}, give, for
$\ell\ge1$ and $a\ge0$,
\begin{equation}\label{d3:eq:constant}
 \lambda_{\ell,0}<
 \min\left\{\frac{3\ell(\ell+1)}D,
       (a^2+\ell(\ell+1))\frac{2a+3}{2a+1}\right\}.
\end{equation}
The constant radial function is not an eigenfunction in a positive angular
sector, so the first bound is strict, including at $\rho=0$.
The ball power trial is also not an eigenfunction: for $a>0$ its
derivative at one is nonzero, and $a=0$ is the constant trial.
Its strict Rayleigh bound and the comparison of lowest radial eigenvalues
with the ball give the second inequality for every $0\le\rho<1$.
The power trials $a=1,2$, together with the ball trial
$s^6(1-3s^2/4)$ in degree six, give
\begin{equation}\label{d3:eq:Qs}
 \begin{aligned}
 Q_j&=\frac53(j^2+j+1)&& (1\le j\le3),\\
 Q_j&=\frac75(j^2+j+4)&& (j=4,5,7),\qquad Q_6=\frac{127}{2}.
 \end{aligned}
\end{equation}
Indeed, the last trial has Rayleigh quotient $39330/623<127/2$.
These bounds increase with $j$, so, for $K>0$,
\begin{equation}\label{d3:eq:firstrows}
 K^2\ge Q_j\quad\Longrightarrow\quad N^<(K^2)\ge(j+1)^2.
\end{equation}

\begin{lemma}[Initial intervals]\label{d3:lem:initial}
For $K>0$, $N^<(K^2)>\W_3(K,\mu)$ holds when $K^2\le Q_6$;
it also holds when $\rho\ge1/2$ and $K^2\le Q_7$.
\end{lemma}
\begin{proof}See Appendix~\ref{app:proof-d3-lem-initial}.\end{proof}

\begin{proposition}\label{d3:thm:low}
The bound $N^<(K^2)\ge\W_3(K,\mu)$ holds if $T\le4/3$ or
$K\le44/5$. It is strict when $K>0$.
\end{proposition}
\begin{proof}
The case $K=0$ is immediate. Suppose first that $hK\le4\pi/3$,
and choose the integer $m\ge1$ with
\begin{equation}\label{d3:eq:m}
 m(m-1)\le X^2<m(m+1).
\end{equation}
The constant trial gives
\begin{equation}\label{d3:eq:mcount}
 N^<(K^2)\ge m^2,\qquad
 \W_3(K,\mu)\le\frac89X^2<\frac89m(m+1).
\end{equation}
This suffices for $m\ge8$. For $m\le7$, put $j=\max\{6,m\}$.
If $K^2>Q_j$, then \eqref{d3:eq:firstrows} gives
$N^<\ge(j+1)^2>8m(m+1)/9$. Otherwise apply
Lemma~\ref{d3:lem:initial}. Its radius hypothesis for $m=7$
follows because $D/h^2$ increases with $\rho$, and
\[
 \rho\le\frac12\ \Longrightarrow\ X^2\le\frac{112\pi^2}{27}<42,
\]
contrary to $X^2\ge m(m-1)$. This proves the thin range.

For $K\le44/5$, the initial lemma covers $K^2\le Q_6$.
If $Q_6\le K^2\le(44/5)^2$, the first seven sectors give
$N^<\ge49>\W_3(44/5,0)$, using $\pi>157/50$.
\end{proof}

\subsubsection{The degree-zero phase and positive degrees}
Let $\phi_\ell=\phi_{1/2,\ell+1/2}$ be the common Robin phase.
It is also the continuous phase of $j_\ell'+iy_\ell'$ because
$j_\ell'+iy_\ell'=\sqrt{\pi/2}\,t^{-3/2}Q_{1/2,\ell+1/2}$.
Write $\Delta_\ell(K)=\phi_\ell(K)-\phi_\ell(\mu)$.
The exact count is given by Lemma~\ref{lem:exact}.

For the zero angular sector the phase is explicit:
\begin{equation}\label{d3:eq:phasezero}
 \phi_0(t)=t+\arctan(1/t),\qquad
 \Delta_0(K)=\pi T-\arctan\frac{K-\mu}{1+\mu K}.
\end{equation}
Here the first formula is for $t>0$, with the continuous extension
$\phi_0(0)=\pi/2$. The formula for $\Delta_0$ therefore includes
$\mu=0$. In particular, $\Delta_0(K)>\pi(T-1/2)$.

\label{d3:sec:enclosure}
Write $\nu=\ell+1/2$.

\begin{theorem}[Positive angular degrees]
\label{d3:thm:phaseenclosure}
For every integer $\ell\ge1$, the common enclosure
\eqref{eq:cp-enclosure} holds with $\alpha=1/2$ and $\nu=\ell+1/2$.
\end{theorem}
\begin{proof}
For $\ell\ge2$, use Lemma~\ref{lem:common-low-robin-phase}
with $\alpha=1/2$ and $\nu=\ell+1/2$.

For $\ell=1$, the half-integer Hankel formula gives
\[
 \phi_1(t)=t+\arg\bigl(2t+i(2-t^2)\bigr),\qquad
 \phi_1'(t)=\frac{t^2(t^2-2)}{t^4+4},
\]
where the argument lies in $(-\pi/2,\pi/2)$.
The phase has a single minimum and $\phi_1(\pi/2)<\pi/2$;
hence $\phi_1(t)<\pi/2$ for $0<t\le\pi/2$.
For $t\ge\pi/2$, put $z=t^2>7/3$. Then
\[
 (4z-9)(z^2+4)^2-4z^3(z-2)^2
 =z^2(7z^2+16z-72)+16(4z-9)>0.
\]
Squaring the positive derivatives therefore gives
$\phi_1'(t)<S_{3/2}'(t)$.
Thus $\phi_1-S_{3/2}$ decreases on $[\pi/2,\infty)$,
which proves the upper bound. The lower bound is already contained
in \eqref{eq:flps} and \eqref{eq:robinphase}.
\end{proof}

Consequently the three-quarter bound \eqref{eq:threequarter}
holds in every positive degree. Put $P(K,\mu)=P_3(K,\mu)$.

\begin{proposition}[Two refinements for the zero sector]\label{d3:prop:zeroimproved}
If $\mu\ge1$, then
\begin{equation}\label{d3:eq:nocost}
 N^{<}(K^2)\ge P(K,\mu).
\end{equation}
If $0\le\mu<1$ and $K>1/2$, then
\begin{equation}\label{d3:eq:smallmu}
 N^{<}(K^2)\ge P(K,0)-1.
\end{equation}
Here $P(K,0)$ means the same action model with $G_0=0$.
\end{proposition}
\begin{proof}
For $\mu\ge1$, the argument of the arctangent in
\eqref{d3:eq:phasezero} is less than one. Therefore
$\Delta_0(K)/\pi>T-1/4\ge A(1/2)-1/4$. The ceiling argument
used for \eqref{eq:threequarter} now also applies to $\ell=0$.

If $\mu<1$, all positive angular sectors have
$\ell+1/2>\mu$, so their model terms equal those for the ball.
Moreover, $\phi_0$ is increasing on $[0,\infty)$ and
$\phi_0(\mu)<\phi_0(1)=1+\pi/4$. The common estimates \eqref{eq:flps} and \eqref{eq:robinphase} give
$\phi_0(K)>\pi G_K(1/2)+\pi/4$. Hence
\[
 \frac{\Delta_0(K)}\pi>G_K(1/2)-\frac1\pi
 >G_K(1/2)-\frac12.
\]
It follows that $n_0^{<}\ge\lfloor G_K(1/2)+1/2\rfloor$, which
is at least $\lfloor G_K(1/2)+3/4\rfloor-1$.
\end{proof}

\subsubsection{The quadratic angular comparison}
\label{d3:sec:joint}
The common row identity \eqref{eq:rows} becomes
\begin{equation}\label{d3:eq:C3}
 C_3(z)=\lfloor z+1/2\rfloor^2\ge(z-1/2)_+^2=\Phi_2(z).
\end{equation}
We therefore apply Proposition~\ref{prop:physical-transfer} with
$Q_{B,2}$, $J_2$ and $M_{\Phi_2}$.

\begin{lemma}\label{d3:lem:ballpoly}
For $K\ge44/5$, put $v=K^{-1/3}$ and
\[
 B(v)=\frac18-\frac{683}{2000}v^2
                  -\frac{29}{500}v^3+\frac{21}{250}v^4.
\]
Then
\begin{equation}\label{d3:eq:ballpoly}
 \frac{Q_{B,2}(K)-\W_3(K,0)}{K^2}\ge B(v).
\end{equation}
\end{lemma}
\begin{proof}See Appendix~\ref{app:proof-d3-lem-ballpoly}.\end{proof}

\begin{lemma}\label{d3:lem:scalarjoint}
For $K\ge44/5$, put $v=K^{-1/3}$ and
\[
 U(v)=\frac3{16}-\frac{2503}{6400}v^3
                         +\frac{13467}{51200}v^6.
\]
Then
\begin{equation}\label{d3:eq:scalar1}
 B(v)-2U(v)^2>\frac1{100000},\qquad 2U(v)^2>\frac1{32}.
\end{equation}
\end{lemma}
\begin{proof}
Write $b=2503/6400<2/5$, $c=13467/51200$, and
$H=B-2U^2$. For $0<v\le1/2$, one has
$b-2cv^3>0$ and $1/8<U(v)<3/16$. Hence
\[
 \frac{H'(v)}v
 =-\frac{683}{1000}-\frac{87}{500}v+\frac{42}{125}v^2
        +12vU(v)(b-2cv^3)
 <-\frac{683}{1000}+\frac{21}{250}+\frac9{20}
 =-\frac{149}{1000}<0.
\]
Cubing gives $(44/5)^{-1/3}<31/64$, and direct rational
substitution gives $H(31/64)>1/100000$. This proves the claim.
\end{proof}

\begin{proposition}\label{d3:prop:jointclosure}
For $0\le\rho<1$, $K\ge44/5$ and $T\ge4/3$,
\begin{equation}\label{d3:eq:quantitative-surplus}
 N^<(K^2)-\W_3(K,\mu)>\frac{K^2}{100000}.
\end{equation}
\end{proposition}
\begin{proof}
If $\mu<1$, Proposition~\ref{d3:prop:zeroimproved} and
$C_3\ge\Phi_2$ give
\[
 N^<(K^2)\ge P(K,0)-1\ge Q_{B,2}(K)-1
 >\W_3(K,0)+\frac1{32}K^2-1.
\]
Consequently
\[
 \frac{N^<(K^2)-\W_3(K,\mu)}{K^2}
 >\frac1{32}-\frac1{K^2}
 \ge\frac1{32}-\frac{25}{1936}
 =\frac{71}{3872}>\frac1{100000}.
\]
Suppose henceforth that $\mu\ge1$, so $N^<\ge P(K,\mu)$.
Set
\[
 \chi=\left(\frac{1+\rho+\rho^2}{12K^2\rho^2}\right)^{1/3}.
\]
The Jensen bound \eqref{eq:jp-Jensen} and maximization of
$(1-1/(2z))\sqrt{1-\beta_*z^2}$ give
\[
 \frac{M_{\Phi_2}}{8K^2}
 \le\frac\rho{4T}(1-\chi)^{3/2}.
\]
Indeed the unrestricted maximum occurs at
$z=(2\beta_*)^{-1/3}$, where both $1/(2z)$ and
$\beta_*z^2$ equal $\chi$; enlarging the maximization interval can
only increase the maximum. Here $0<\chi<1$.
The action moments, and $G_\mu(z)\le\mu/\pi$, also give
\[
 \W_3(\mu,0)-J_2(\mu)
 =\frac{\mu^2}{8}-\int_0^{1/2}(1-2z)G_\mu(z)\,dz
 \ge\frac{\mu^2}{8}-\frac\mu{4\pi}.
\]
Thus the inner loss divided by $K^2$ is at most
\begin{equation}\label{d3:eq:jointloss}
 L=\frac\rho{4T}(1-\chi)^{3/2}
                         -\frac{\rho^2}{8}+\frac\rho{4\pi K}.
\end{equation}

To bound the inner loss throughout the range $K\ge44/5$, observe that
\[
 9\bigl[2(1+\rho+\rho^2)-27\rho^2(1-\rho)\bigr]
 =(7\rho-4)^2(1+5\rho)+2(1-\rho)^3>0.
\]
Since $T\ge4/3$ and $\pi>157/50$, this gives
\[
 K^3\chi^3\ge\frac{\pi(1+\rho+\rho^2)}{9\rho^2(1-\rho)}
 >\frac{3\pi}{2}>\left(\frac{67}{40}\right)^3.
\]
The polynomial $1-3\chi/2+\chi^2/2$ decreases on $[0,1]$, and
$\sqrt{1-\chi}\le1-\chi/2$. Consequently
\[
 (1-\chi)^{3/2}\le a(K):=
 1-\frac{201}{80K}+\frac{4489}{3200K^2}.
\]
Using $T\ge4/3$, completing the square in \eqref{d3:eq:jointloss},
and then using $\pi>25/8$, we obtain
\[
 L\le2\left(\frac{3a(K)}{16}+\frac1{4\pi K}\right)^2
 <2U(K^{-1/3})^2.
\]
Here $U$ is the polynomial in Lemma~\ref{d3:lem:scalarjoint}.
Proposition~\ref{prop:physical-transfer} now gives
\[
 \frac{N^<(K^2)-\W_3(K,\mu)}{K^2}
 \ge B(K^{-1/3})-L
 >B(K^{-1/3})-2U(K^{-1/3})^2>\frac1{100000}.
\]

\end{proof}

\begin{proof}[Proof of Theorem~\ref{d3:thm:main}]
At $K=0$ both sides are zero. Proposition~\ref{d3:thm:low}
covers $T\le4/3$ or $K\le44/5$, and
Proposition~\ref{d3:prop:jointclosure} covers the complementary
range. This proves \eqref{d3:eq:main}, including every endpoint.
\end{proof}

\subsection{Dimension four: the cubic comparison}
\label{sec:lowphase-four}
\subsubsection{Localizing the phase defect}

\begin{proposition}[A localized defect in dimension four]
\label{prop:lowphase-four}
For $d=4$, $K>0$, and $0\le\rho<1$, one has
\[
 N_4^<(K^2)\ge P_4(K,\mu)-D_4(\mu),\qquad
 D_4(\mu)=
 \begin{cases}
 0,&\mu=0\ \text{or}\ \mu\ge\sqrt{26},\\
 1,&0<\mu<\sqrt5,\\
 16,&\sqrt5\le\mu<\sqrt{26}.
 \end{cases}
\]
\end{proposition}
\begin{proof}
The ball has no defect. For shells, write $m=\ell+1$.
Lemma~\ref{lem:common-low-robin-phase} proves the three-quarter count
for every $m\ge5$. For $m=2,3,4$, the first-crossing and local phase
criteria in Lemmas~\ref{lem:phase-upper-window} and
\ref{lem:tp-localphase} leave only the respective intervals
\[
 \frac{21}{4}<\mu^2<10,\qquad
 \frac{64}{5}<\mu^2<17,\qquad
 \frac{45}{2}<\mu^2<26.
\]
Indeed, the lower endpoints are $\mathfrak b_4(m-1)$, and
$m^2+m^{4/3}+605/192$ is below the stated upper endpoint in each case.
The intervals are disjoint and lie above five. The half-shift count
therefore loses at most one radial mode in one of these sectors,
whose multiplicities are $4,9,16$.

For $m=1$, the recurrence gives $Q_{1,1}=-tH_2$, so
$\phi_{1,1}=\theta_2+\pi<S_2+3\pi/4$ for $t>2$.
The difference $S_1-S_2$ increases on $[2,\infty)$, and
\[
 S_1(\sqrt5)-S_2(\sqrt5)
 =1-\frac\pi2+3\arctan\frac12
 >\frac{19}{8}-\frac\pi2>\frac\pi4.
\]
Thus the zero sector has no defect when $\mu\ge\sqrt5$.
Below this threshold its multiplicity is one. Combining the sectors
proves the stated bound, including every endpoint.
\end{proof}

\subsubsection{The cubic angular comparison}
\label{app:d4high}\label{app:d4compact}

The common angular polynomial is
\[
 \Phi_3(z)=0\quad(z\le1),\qquad
 \Phi_3(z)=z^3/3-z^2/2+z/6\quad(z\ge1),
\]
with ball quadrature error $E_B(\Phi_3)$.

\begin{lemma}\label{lem:d4ball}
For every $K\ge12$,
\[
 \frac{E_B(\Phi_3)}{K^3}>\frac3{64}.
\]
\end{lemma}
\begin{proof}See Appendix~\ref{app:d4-ball-proof}.\end{proof}

\begin{theorem}\label{thm:d4high}\label{thm:d4all}
For $0\le\rho<1$ and $K>0$, the strict Neumann P\'olya
inequality holds in dimension four. Moreover, if $K\ge12$ and
$T>s(K)/\pi$, where
\[
 s(K)=\frac{16}{3}\left(1-\frac1{\sqrt2K}\right)^3,
\]
then
\begin{equation}\label{eq:d4-quantitative-surplus}
 N_4^<(K^2)-\W_4(K,\rho K)>\frac{K^3}{20000}.
\end{equation}
\end{theorem}
\begin{proof}
Lemma~\ref{lem:d4-ground-direct} covers $0<K\le12$.
For $K\ge12$, Proposition~\ref{prop:var-ranges} proves strict
P\'olya if $(1-\rho)K\le s(K)$. We prove the stated quantitative
bound in the remaining range $T>s(K)/\pi$.

For $\mu\ge\sqrt5$, setting the polynomial to zero below one
improves the volume comparison:
\begin{equation}\label{eq:d4-clipping-help}
 \int_0^\infty \Phi_3(x(t))\,dt
 \ge\frac{K^4-\mu^4}{64}-\frac{K^3-\mu^3}{9\pi}
                         +\frac{K^2-\mu^2}{48}.
\end{equation}
Indeed, extend the polynomial expression of $\Phi_3$ to $[0,1]$.
It is positive below $1/2$, negative above $1/2$, and has integral
zero.  Since $-A'(z)$ increases on $[0,1]$ when $\mu>1$,
integration by parts gives
$\int_0^1\Phi_3'(z)A(z)\,dz=\int_0^1\Phi_3(z)(-A'(z))\,dz\le0$.
The full polynomial integral is exactly the right-hand side of
\eqref{eq:d4-clipping-help}.

First suppose $\mu\ge\sqrt{26}$. The phase defect is zero by
Proposition~\ref{prop:lowphase-four}. The Jensen slope bound can also
be improved. For $z\ge1$, $\Phi_3'(z)/z=z-1+1/(6z)\le z-5/6$.
With $\beta=(K^2+K\mu+\mu^2)/(3K^2\mu^2)$, the elementary
inequality $(u-a)\sqrt{1-u^2}\le(1-a)/2$ for $0\le a\le u\le1$
gives
\[
 M_{\Phi_3}\le\frac{K\mu}{2T}\left(\frac1{\sqrt\beta}-\frac56\right)
 \le\frac{K^3\rho^{3/2}}{2T}-\frac{5K\mu}{12T}.
\]
The first inequality follows by putting $u=z\sqrt\beta$ and
$a=5\sqrt\beta/6$; the second uses $3\rho\le1+\rho+\rho^2$.

Set
\[
 \varepsilon=1/K,\quad \varepsilon_0=1/12,\quad c=1/\sqrt2,\quad
 a=1/(9\pi),\quad B(\varepsilon)=\frac{3\pi}{256(1-c\varepsilon)^3}.
\]
The coefficient $\rho^{3/2}-5\varepsilon\rho/6$ is positive, since
$K\sqrt\rho=\sqrt{K\mu}>5/6$.  Proposition~\ref{prop:physical-transfer},
\eqref{eq:d4-clipping-help}, and $T>s(K)/\pi$ therefore show that
the normalized loss is at most
\[
 L(\varepsilon,\rho)=a(1-\rho^3)+B(\varepsilon)\rho^{3/2}
 -\varepsilon\left(\frac{1-\rho^2}{48}+\frac56B(\varepsilon)\rho\right).
\]
We claim $L<3/64-1/20000$ on the whole rectangle
$0\le \varepsilon\le \varepsilon_0$, $0\le\rho\le1$. Put
\[
 D(\varepsilon)=B(\varepsilon)\left(1-\frac56\varepsilon\right)+\frac \varepsilon{48}.
\]
Since $5B/6>1/48$, the inequalities
$1-\rho^2\ge1-\rho$ and $\rho\ge\rho^{3/2}$ give
\[
 L\le a-\frac \varepsilon{48}+D(\varepsilon)\rho^{3/2}-a\rho^3
 \le J(\varepsilon):=a-\frac \varepsilon{48}+\frac{D(\varepsilon)^2}{4a}.
\]
This upper bound increases with $\varepsilon$. Indeed, $c>7/10$ and
$\varepsilon\le \varepsilon_0<1/10$ give
\[
 D'(\varepsilon)=B(\varepsilon)\left[\frac{3c(1-5\varepsilon/6)}{1-c\varepsilon}-\frac56\right]
                +\frac1{48}>B(\varepsilon)+\frac1{48}>\frac1{20}.
\]
Hence $D(\varepsilon)\ge D(0)>1/30$, and $\pi>3$ yields
\[
 J'(\varepsilon)=-\frac1{48}+\frac{D(\varepsilon)D'(\varepsilon)}{2a}
 >-\frac1{48}+\frac1{30}\frac1{20}\frac{27}{2}
 =\frac1{600}>0.
\]
The bounds $333/106<\pi<22/7$ and $c<99/140$ give
\[
 a<\frac{17}{480},\qquad \frac1{4a}<\frac{99}{14},\qquad
 B(\varepsilon_0)<\frac2{45}.
\]
Consequently
\[
 J(\varepsilon)\le J(\varepsilon_0)<\frac{17}{480}-\frac1{576}
       +\frac{99}{14}\left(\frac{1117}{25920}\right)^2
 <\frac3{64}-\frac1{20000}.
\]
Together with Lemma~\ref{lem:d4ball}, this proves
\eqref{eq:d4-quantitative-surplus} when $\mu\ge\sqrt{26}$.

Now suppose $0\le\mu<\sqrt{26}$. The phase defect is at most sixteen.
We allow a clipping error throughout this range: on $0\le z\le1$,
$|z(z-1/2)(z-1)/3|\le1/(36\sqrt3)$ and
$\int_0^1[-A'(z)]\,dz\le T$. Thus the right-hand side of
\eqref{eq:d4-clipping-help}, minus $T/(36\sqrt3)$, is a lower bound
for $\int_0^\infty\Phi_3(x(t))\,dt$.

For $1\le z\le\mu$, the function
$\Phi_3'(z)/z^2=1-1/z+1/(6z^2)$ increases. When $\mu>1$,
\eqref{eq:jp-Jensen} and
$z\sqrt{1-\beta_*z^2}\le1/(2\sqrt{\beta_*})$ therefore give
\[
 \frac{M_{\Phi_3}}{8K^3}
 \le\frac{\pi\sqrt3\,\Phi_3'(\mu)}
 {16K^3(1-\rho)\sqrt{1+\rho+\rho^2}}
 <\frac{3\pi\Phi_3'(\mu)}{20K^3}.
\]
Indeed, $\mu<\sqrt{26}<36/7$ and $K\ge12$ imply $\rho<3/7$, so
\[
 (1-\rho)^2(1+\rho+\rho^2)
 =(1-\rho)(1-\rho^3)
 >\frac47\left(1-\frac{27}{343}\right)>\frac{25}{48}.
\]
For $\mu\le1$, one has $M_{\Phi_3}=0$. Hence the slope loss is
bounded in the whole range by
$3\pi(\mu^2-\mu+1/6)_+/(20K^3)$.

Put $\varepsilon=1/K$ and $a=1/(9\pi)$. Since $T\le K/\pi$,
the total normalized loss is at most
\[
 a-\frac{\varepsilon}{48}
   +\varepsilon^3F(\mu)
   +\frac{\varepsilon^2}{36\pi\sqrt3},\qquad
 F(\mu)=16+\frac{\mu^2}{48}-a\mu^3
             +\frac{3\pi}{20}\left(\mu^2-\mu+\frac16\right)_+.
\]
For $0\le\mu\le\sqrt5$, dropping the negative cubic term gives
\[
 F(\mu)\le16+\frac5{48}+\frac{31\pi}{40}<19,
 \qquad
 F(36/7)>\frac{290533}{13720}>21,
\]
where the second bound uses $\pi>3$.
On $\sqrt5\le\mu\le36/7$, the positive part is unnecessary and
\[
 F'(\mu)>
 \frac\mu{24}-\frac{\mu^2}{9}+\frac9{20}(2\mu-1)
 >\frac\mu3-\frac9{20}>0.
\]
Thus $F(\mu)\le F(36/7)$ throughout the range.
After this substitution the loss bound is convex in
$0\le\varepsilon\le1/12$. Since $F(36/7)>3$, its value
at $1/12$ exceeds its value at zero; its maximum is therefore at
$\varepsilon=1/12$. There the bound is
\[
 a\left(1-\frac{27}{343}\right)
 -\frac1{576}\frac{40}{49}
 +\frac{3\pi}{20\cdot1728}
   \left[\left(\frac{36}{7}\right)^2-\frac{36}{7}+\frac16\right]
 +\frac1{108}+\frac1{5184\pi\sqrt3}.
\]
Using $28/9<\pi<22/7$, this is less than
\[
 \frac{33}{1000}-\frac1{720}+\frac3{512}+\frac1{108}
      +\frac1{15552}
 =\frac3{64}-\frac{1259}{15552000}
 <\frac3{64}-\frac1{20000}.
\]
Lemma~\ref{lem:d4ball} proves \eqref{eq:d4-quantitative-surplus}
throughout $0\le\mu<\sqrt{26}$, including the ball endpoint,
and completes the proof.

\end{proof}

\subsection{Planar annuli and disks}\label{sec:planar}
We now prove the planar inequality using the phase and variational
estimates of this paper. The annulus result was proved by
Dai--Sun~\cite{DSannuli}. The disk inequality is due to
\cite[Corollary~1.6]{FLPSdisk}; \cite[Section~4]{Li} gives a proof using
explicit trial spaces and finite rational calculations. We first give a
short disk proof and then treat annuli.

\begin{theorem}[The planar inequality]\label{thm:planar-complete}
For every $0\le r<R$ and $\Lambda>0$,
\[
 N^<_{\Omega_{r,R}}(\Lambda)>\frac{R^2-r^2}{4}\Lambda.
\]
Here $\Omega_{0,R}$ is the disk of radius $R$.
\end{theorem}

Retain the notation $K,\mu,T,A,x,\xi_K$ from
Section~\ref{sec:common-quadrature}. In dimension two,
\[
 \W_2(K,\mu)=\frac{K^2-\mu^2}{4}
            =2\int_0^\infty x(t)\,dt.
\]
Radial trials cover $K\le8$, and $K\ge8$ with $T\le3/2$
(Lemma~\ref{d2:lem:variational}). The remaining annulus range is treated
by comparing inverse actions and reducing the angular count to width
endpoints.

\paragraph{The radial count.}
Lemma~\ref{lem:exact} applies with $\alpha=0$ and $\nu=m$.
For every integer $m\ge1$, Lemma~\ref{lem:common-low-robin-phase}
therefore gives
\[
 n_m^<(K^2)\ge\lfloor A(m)+3/4\rfloor.
\]
The same bound holds for $m=0$. Indeed, $H_0'=-H_1$ and
Nicholson's integral give
$|H_1(t)|^2>|H_{1/2}(t)|^2=2/(\pi t)$.
Writing the derivative phase as
$\varphi_0(t)=t+\pi/4+e_0(t)$, its Wronskian derivative gives
$e_0'(t)=2/(\pi t|H_1(t)|^2)-1<0$.
The endpoint limits are $e_0(0+)=\pi/4$ and $e_0(\infty)=0$.
Thus $\varphi_0(K)-\varphi_0(\mu)>K-\mu-\pi/4$, and the exact
ceiling formula gives the asserted floor, including $\mu=0$.
Consequently
\begin{equation}\label{d2:eq:phase-count}
 N^<(K^2)\ge P_2(K,\mu):=
 \lfloor T+3/4\rfloor+2\sum_{m\ge1}\lfloor A(m)+3/4\rfloor.
\end{equation}
All these counts are strict; order zero includes the constant mode.

For $t_j=j+1/4$, interchange angular orders and radial rows in
\eqref{d2:eq:phase-count}:
\begin{equation}\label{d2:eq:angular-rows}
 P_2(K,\mu)=\sum_{t_j\le T}\bigl(1+2\lfloor x(t_j)\rfloor\bigr).
\end{equation}
The terminal row $t_j=T$ contributes one point.

\subsubsection{The disk}

\begin{proposition}\label{d2:prop:disk}
For every $K>0$, $N^<_{B_1}(K^2)>K^2/4$. Moreover, for $K\ge8$,
\begin{equation}\label{d2:eq:disk-surplus}
 2E_{1/4}(\xi_K)-\lfloor K/\pi+3/4\rfloor>\frac K{500}.
\end{equation}
\end{proposition}
\begin{proof}
Lemma~\ref{d2:lem:variational} gives the strict inequality for $K\le8$.
For $K\ge8$, put $n=\lfloor K/\pi+3/4\rfloor$. By
\eqref{d2:eq:angular-rows},
\[
 P_2(K,0)-K^2/4\ge2E_{1/4}(\xi_K)-n.
\]
For $8\le K\le10$, $n=3$, and \eqref{eq:onepoint-ball-linear} gives
\[
 2E_{1/4}(\xi_K)-n>
 \frac K2-\frac8{25}K^{1/3}-\frac{10}{3}
 \ge\frac2{75}>\frac K{500}.
\]
The middle expression increases with $K$ and equals $2/75$ at $K=8$.
For $K\ge10$, use $n\le K/\pi+3/4$ and $\pi>25/8$ to obtain
\[
 2E_{1/4}(\xi_K)-n>
 \frac9{50}K-\frac8{25}K^{1/3}-\frac{13}{12}>\frac K{500}.
\]
After division by $K$, the expression increases; at $K=10$,
$10^{1/3}<13/6$ bounds it below by $7/3000>1/500$.
Now \eqref{d2:eq:phase-count} proves the disk inequality.
\end{proof}

\subsubsection{Planar annuli}

\paragraph{Cancellation of the terminal singularity.}
The linear weight needed in dimension two has an unbounded inner
slope. A semicircle has the same leading singularity and restores
the bounded-variation comparison.

\begin{lemma}\label{d2:lem:semicircle}
For $K\ge8$ and $0<\mu<K$, put
\[
 y(t)=\sqrt{\mu K}\left(1-\frac{t^2}{T^2}\right)_+^{1/2},
 \qquad
 e(T)=\sum_{\substack{j\ge0\\j+1/4<T}}
       \sqrt{1-\frac{(j+1/4)^2}{T^2}}-\frac{\pi T}{4}.
\]
Then
\begin{equation}\label{d2:eq:semicircle-transfer}
 E_{S,1}>\frac K4-\frac4{25}K^{1/3}+\frac1{12}
       +\sqrt{\mu K}\bigl(e(T)-1/4\bigr)
       -\frac18\sqrt{\frac{\pi K}{T}}.
\end{equation}
\end{lemma}
\begin{proof}
Write $c=G_K(\mu)$, and let $s$ be uniformly distributed on
$[\mu,K]$. For $0<z<\mu$, set
\[
 a=\mathbb E(s^{-2})=\frac1{\mu K},\quad
 B=\mathbb E\frac{s^{-2}}{\sqrt{1-z^2/s^2}},\quad
 C=\mathbb E\frac{s^{-2}}{(1-z^2/s^2)^{3/2}}.
\]
Then $A=T\mathbb E\sqrt{1-z^2/s^2}$, $A'=-TzB$, and
$A''=-TC$. At $t=A(z)$, Jensen's and H\"older's inequalities give
$z^2\le a^{-1}(1-t^2/T^2)$ and $B^3\le Ca^2$. Hence
\[
 -x''(t)=\frac{C}{T^2z^3B^3}
 \ge\frac1{T^2a^2z^3}
 \ge\frac1{T^2\sqrt a(1-t^2/T^2)^{3/2}}=-y''(t).
\]

Put $L=K/\pi$, $r(t)=2(L-t)_+$, and $h=\xi_K-r$.
Extend the slopes $g_S=-x'$, $g_B=-\xi_K'$, and $g_y=-y'$
by zero past their endpoints. On $(0,L)$, the function
$w=g_y+g_B-g_S-2$ equals $g_y-2$ on $(0,c)$ and decreases
on $(c,L)$, since $w'\le g_B'<0$ on $(c,T)$.
The expansion $A(z)=T-Tz^2/(2\mu K)+O(z^4)$ gives
$g_y-g_S=O(\sqrt{T-t})$, so $w$ continues at $T$ to $g_B-2$
and tends to zero at $L$. It is continuous at $c$ because
$A'(\mu)=G_K'(\mu)$.
Thus $M_y\ge2$, $w(0+)=-2$, $w(c)=M_y-2$, and the positive and negative
variation masses of $-w$ are $M_y-2$ and $M_y$, where
\[
 M_y=g_y(c)=\frac{\sqrt{\mu K}}T
            \frac{c/T}{\sqrt{1-(c/T)^2}}.
\]
In the proof of Lemma~\ref{lem:variation-comparison},
$Q_{1/4}(0)=1/32$. Applying its lower bound to $x$ and $h+y$ gives
\[
 E_{S,1}\ge E_{1/4}(h)+E_{1/4}(y)
       +\frac L2-\frac14\sqrt{\mu K}-\frac{M_y}8+\frac1{16}.
\]
Applying \eqref{eq:onepoint-quarter} to $h$, as in the linear case
of Lemma~\ref{lem:onepoint-ball}, and using
\eqref{eq:linear-onepoint-functional}, gives
\[
 E_{1/4}(h)\ge\mathcal L(\xi_K)-\mathcal L(r)
 >\frac K4-\frac4{25}K^{1/3}-\frac L2+\frac1{48}.
\]
Finally $E_{1/4}(y)=\sqrt{\mu K}\,e(T)$, while Jensen's inequality
at $z=\mu$ gives $(c/T)^2\le1-\rho$ and hence
$M_y^2\le\pi K/T$.
\end{proof}

\paragraph{Concavity in the squared inner frequency.}
For fixed $K$, write $x_q(t)=(G_K-G_{\sqrt q})^{-1}(t)$.
Concavity holds on the interval where the sample exists; extending
it by zero past that interval would not preserve this property.

\begin{lemma}\label{d2:lem:concavity}
For $0<t<K/\pi$, the function $q\mapsto x_q(t)$ is nonincreasing
and concave on $0\le q\le(K-\pi t)^2$.
Consequently every fixed-row surplus
\[
 2\sum_{j=0}^{n-1}x_q(t_j)-n-\frac{K^2-q}{4}
\]
is concave on any common active interval, and attains its minimum
at an endpoint.
\end{lemma}
\begin{proof}
On the inner branch, put $\mu=\sqrt q$, $x=x_q(t)$,
$a=\arccos(x/K)$, $b=\arccos(x/\mu)$, and $\Delta=a-b$.
Implicit differentiation gives
\[
 \partial_\mu x=-\frac{\sin b}{\Delta},\qquad
 \partial_q x=-\frac{\sin b}{2\mu\Delta},\qquad
 \mu b_\mu=\cot b+\frac1\Delta>0,
\]
\[
 (\mu\Delta)_\mu
 =\Delta-\cot b-\frac{1-\tan b\cot a}{\Delta}<0.
\]
Here $\tan b\cot a<1$ and $\Delta<\pi/2-b<\cot b$.
Thus $\sin b/(\mu\Delta)$ increases with $\mu$, so $\partial_qx$
is negative and decreasing on this branch.
At $x=\mu$, $\partial_qx$ tends to zero, joining the constant outer branch
with zero derivative. Continuity includes both active endpoints.
\end{proof}

Define
\begin{equation}\label{d2:eq:fixed-row-surplus}
 T_n=n+1/4,\qquad
 \Sigma_n(T)=2\sum_{j=0}^{n-1}x_{(K-\pi T)^2}(t_j)-n
                         -\W_2(K,K-\pi T).
\end{equation}
By Lemma~\ref{d2:lem:concavity}, $\Sigma_n$ attains its minimum at an
endpoint of any closed interval where its $n$ samples are active:
it is concave as a function of $(K-\pi T)^2$.

\begin{proposition}[The planar lattice comparison]\label{d2:prop:reduction}
For $K\ge8$ and $3/2\le T\le K/\pi$, one has
$P_2(K,\mu)>\W_2(K,\mu)$.
\end{proposition}
\begin{proof}
The endpoint estimates proved in Appendix~\ref{d2:app:endpoints} are
\begin{equation}\label{d2:eq:endpoints}
 \begin{gathered}
 \Sigma_2(3/2)>0,\qquad \Sigma_n(T_n)>0\quad(n\ge2,\ T_n\le K/\pi),\\
 \Sigma_n(K/\pi)>0,\qquad n=\lfloor K/\pi+3/4\rfloor.
 \end{gathered}
\end{equation}
The two-row surplus has positive endpoint values on
$[3/2,\min(T_2,K/\pi)]$, proving the initial interval.
Now increase $T$ from $T_{n-1}$ across the next strip, with $n\ge3$.
Keeping the new row's zero-angular point and applying
$1+2\lfloor x\rfloor\ge2x-1$ to the other rows gives
\[
 P_2-\W_2\ge\max\{\Sigma_n(T),\Sigma_{n-1}(T)+1\}.
\]
The two bounds first agree when the last inverse reaches one.
If this happens at inner frequency $\mu_1$, then $\mu_1\ge1$:
after joining the outer branch the sample is constant. Put $\mu_0=K-\pi T_{n-1}$. The action equation gives
\[
 \mu_0-\mu_1
 =\int_{\mu_1}^K\left(1-\sqrt{1-s^{-2}}\right)ds
 <\frac1{\mu_1},\qquad
 \frac{\mu_0^2-\mu_1^2}{4}<\frac34.
\]
The older inverse values increase as $T$ increases. Hence at this
crossing $\Sigma_{n-1}+1>\Sigma_{n-1}(T_{n-1})+1/4>0$.
Writing $T_c=(K-\mu_1)/\pi$, apply concavity to $\Sigma_{n-1}+1$
on $T_{n-1}\le T\le T_c$ and to $\Sigma_n$ after $T_c$, using the
prescribed right endpoint of the strip.
If no crossing occurs, $\Sigma_{n-1}+1$ is positive at both endpoints:
at the right endpoint it is at least $\Sigma_n$, because their difference
is $2(1-x(t_{n-1}))\ge0$.
Concavity again proves the strip. A newly entering zero inverse
contributes one point at each exact quarter-height endpoint.
Induction proves the proposition, including the ball endpoint.
\end{proof}

\begin{proof}[Proof of Theorem~\ref{thm:planar-complete}]
Scale to outer radius one. Lemma~\ref{d2:lem:variational} covers
$K\le8$, and $K\ge8$ with $T\le3/2$. In the remaining annulus
range, Proposition~\ref{d2:prop:reduction} gives $P_2>\W_2$, and
\eqref{d2:eq:phase-count} gives $N^<\ge P_2$.
The disk was proved in Proposition~\ref{d2:prop:disk}.
\end{proof}

\subsection{Completion of the proof}

\begin{proof}[Proof of Theorem~\ref{thm:all-dimensions} from the dimensional estimates]
Scale the outer radius to one. Theorem~\ref{thm:clipped-surplus}
proves the assertion for $d\ge5$, Theorem~\ref{thm:d4all} for $d=4$,
Theorem~\ref{d3:thm:main} for $d=3$, and
Theorem~\ref{thm:planar-complete} for $d=2$. Each estimate is strict
at positive frequency. At $\Lambda=\lambda_k^N>0$, the strict count
is at most $k$, which gives the eigenvalue inequality. Conversely,
if the Weyl term is $w>0$, put $m=\lfloor w\rfloor$. The eigenvalue
inequality gives $\lambda_m^N<\Lambda$ when $m\ge1$; for $m=0$,
this follows from $\lambda_0^N=0$. Hence at least $m+1>w$
eigenvalues lie strictly below $\Lambda$. At zero, both sides of
the counting inequality vanish.
\end{proof}

\begin{proof}[Proof of Corollary~\ref{cor:balls}]
Set $\rho=0$ in Theorems~\ref{thm:planar-complete},
\ref{d3:thm:main}, \ref{thm:d4all}, and~\ref{thm:clipped-surplus}
for $d=2$, $d=3$, $d=4$, and $d\ge5$, respectively, and restore
the outer radius by scaling. This proves
the strict ball inequality for every
positive frequency. Moreover, $T=K/\pi$ and $k_d>w_d$ for $d\ge3$.
Thus Proposition~\ref{d3:prop:jointclosure},
\eqref{eq:d4-quantitative-surplus}, and
Theorem~\ref{thm:clipped-surplus} give the ball surplus for $d\ge3$.
For $d=2$, Proposition~\ref{d2:prop:disk} gives the larger coefficient
$1/500>c_2$. Thus, in every dimension,
\begin{equation}\label{eq:ball-effective-surplus}
 N^<_{B_R}(\Lambda)-\frac{R^d\Lambda^{d/2}}{2^d\Gamma(d/2+1)^2}
 >c_dR^{d-1}\Lambda^{(d-1)/2}
 \qquad(R\sqrt\Lambda\ge k_d, d\ge2).
\end{equation}
\end{proof}

\section{Uniform surplus and removal of small holes}\label{sec:enhanced-removal}
The positive margins in the dimensional proofs give an explicit surplus
in every dimension.
An interval Riesz mean inequality extends the estimate to thin
shells above a dimension-dependent outer frequency. The resulting
bound is uniform in the radius ratio. The strict shell and ball
inequalities, including Corollary~\ref{cor:balls}, and compactness
extend this uniformity to all positive frequencies. The removal
argument of \cite{HJLW}, made uniform over compact families of host
shells, then gives Corollary~\ref{cor:uniform-perforated-shells}.

Throughout this section, scale the outer radius to one and write
\[
 p=d-1,\qquad \rho=r/R,\qquad K=R\sqrt\Lambda,\qquad
 h=1-\rho,\qquad T=hK/\pi,\qquad \mu=\rho K.
\]
Thus the normalized surplus is
\[
 S(\rho,K)=N^<_{\Omega_{\rho,1}}(K^2)-\W_d(K,\rho K).
\]

We first extract a uniform surplus from the planar proof.

\begin{proposition}[A uniform planar surplus]\label{prop:planar-surplus}
For every $K\ge8$ and $0\le\rho<1$,
\begin{equation}\label{eq:planar-surplus}
 N^<_{\Omega_{\rho,1}}(K^2)-\frac{1-\rho^2}{4}K^2>\frac K{4000}.
\end{equation}
\end{proposition}
\begin{proof}
If $T\le3/2$, Lemma~\ref{d2:lem:variational} gives a surplus greater
than $K/120$. Suppose $T\ge3/2$. The endpoint estimates in
Appendix~\ref{d2:app:endpoints} and \eqref{d2:eq:disk-surplus} give
\[
 \begin{aligned}
 \Sigma_2(3/2)&>\frac16+\frac6{385}(K-6),\\
 \Sigma_n(T_n)&>\frac K{250}\quad(n\ge3,\ T_n\le K/\pi),\\
 \Sigma_n(K/\pi)&>\frac K{500},\qquad n=\lfloor K/\pi+3/4\rfloor.
 \end{aligned}
\]
Each exceeds $K/4000$. For the remaining endpoint, the same appendix gives
\[
 \Sigma_2(9/4)>\frac{13}{28}K-7K^{1/3}+\frac{1047}{100}
 \quad(8\le K\le14),\qquad
 \partial_K\Sigma_2(9/4)>\frac3{700}\quad(K\ge14).
\]
Put $a=13/28-1/4000>0$ and $y=K^{1/3}$. The minimum of
$ay^3-7y+1047/100$ on $y\ge0$ is positive because
\[
 \left(\frac{1047}{100}\right)^2>
 \frac{196}{9}\,\frac7{3a}.
\]
Thus $\Sigma_2(9/4)>K/4000$ on $[8,14]$; the derivative bound extends
it to every $K\ge14$.

Fix $K\ge8$ and repeat the strip argument of
Proposition~\ref{d2:prop:reduction}. Each lower bound is concave in
$\mu^2$ on its interval, and hence is at least its smaller endpoint
value. At a crossing, the bound exceeds
$\Sigma_{n-1}(T_{n-1})+1/4$. The endpoint estimates above therefore give
$P_2-\W_2>K/4000$ on every strip, including terminal zero rows.
Finally, $N^<\ge P_2$ by \eqref{d2:eq:phase-count}.
\end{proof}

\begin{remark}[The larger-frequency planar coefficient]
\label{rem:planar-dai-sun}
Dai--Sun~\cite[Theorem~6.1]{DSannuli} prove, with inclusive counting,
\[
 N^\le_{\Omega_{\rho,1}}(s^2)>
 \frac{1-\rho^2}{4}s^2+\frac s{20}
 \qquad(0<\rho<1,\ s\ge80).
\]
For fixed $K>80$, letting $s\uparrow K$ gives
\begin{equation}\label{eq:planar-cited-surplus}
 N^<_{\Omega_{\rho,1}}(K^2)-\frac{1-\rho^2}{4}K^2\ge\frac K{20}.
\end{equation}
This improves the coefficient in \eqref{eq:planar-surplus} at larger
frequencies; it is not used in our proofs.
\end{remark}

\subsection{The margins in the dimensional proof}
\begin{proof}[Proof of Theorem~\ref{thm:effective-surplus} under
\eqref{eq:effective-surplus-onset}]
The hypotheses are $K\ge k_d$ and $hK\ge w_d$.
For $d=2$, Proposition~\ref{prop:planar-surplus} gives
$S(\rho,K)>K/4000$. For $d=3$,
Proposition~\ref{d3:prop:jointclosure} gives
$S(\rho,K)>K^2/100000$. For $d=4$, one has
$hK\ge16/3>s(K)$, so \eqref{eq:d4-quantitative-surplus} gives
$S(\rho,K)>K^3/20000$.

For $d\ge5$, the hypotheses are those of
Theorem~\ref{thm:clipped-surplus}. With $p=d-1$, its conclusion is
\[
 \frac{p!}{2K^p}S(\rho,K)>\frac1{2000},
 \qquad\text{hence}\qquad
 S(\rho,K)>\frac{K^p}{1000p!}.
\]
These are the coefficients in \eqref{eq:effective-surplus-constants}.
Scaling restores the dependence on $R$.
\end{proof}

\subsection{High-frequency surplus uniform in the radius ratio}
The Neumann cosine modes give a surplus estimate in the thin-shell
range. The positive remainder in the following interval inequality
makes the frequency threshold independent of the shell width.

\begin{lemma}[Riesz means of the Neumann interval]\label{lem:interval-riesz}
For every $\gamma\ge1$ and $t>0$,
\begin{equation}\label{eq:interval-riesz}
 \sum_{n\ge0}\Bigl(1-\frac{n^2}{t^2}\Bigr)_+^{\gamma}
 \ge t\int_0^1(1-s^2)^\gamma\,ds+\frac13.
\end{equation}
\end{lemma}
\begin{proof}
Let $\gamma=1$. For an integer $N\ge0$ and $N<t\le N+1$, the left side
of \eqref{eq:interval-riesz} minus the right side, multiplied by $t^2$,
equals
\[
 F_N(t)=\Bigl(N+\frac23\Bigr)t^2-\frac23t^3-\frac{N(N+1)(2N+1)}6.
\]
Since $F_N'(t)=2t(N+2/3-t)$, the function $F_N$ increases and then
decreases on $[N,N+1]$; its endpoint values are $F_N(N)=N(N-1)/6\ge0$
and $F_N(N+1)=N(N+1)/6\ge0$. Hence
\begin{equation}\label{eq:interval-riesz-one}
 \sum_{n\ge0}(E-n^2)_+\ge\frac23E^{3/2}+\frac E3\qquad(E>0).
\end{equation}
For $\gamma>1$ use $a_+^\gamma=\gamma(\gamma-1)\int_0^\infty
\tau^{\gamma-2}(a-\tau)_+\,d\tau$ for real $a$. With $E=t^2$, summation
over $n$ and \eqref{eq:interval-riesz-one} for $E-\tau$ give
\[
 \sum_{n\ge0}(E-n^2)_+^\gamma
 \ge\gamma(\gamma-1)\int_0^E\tau^{\gamma-2}
 \Bigl[\frac23(E-\tau)^{3/2}+\frac{E-\tau}3\Bigr]d\tau
 =E^{\gamma+1/2}\int_0^1(1-s^2)^\gamma\,ds+\frac{E^\gamma}{3}.
\]
The two Beta integrals are
$\gamma(\gamma-1)\int_0^E\tau^{\gamma-2}(E-\tau)\,d\tau=E^\gamma$ and
$\frac23\gamma(\gamma-1)\int_0^E\tau^{\gamma-2}(E-\tau)^{3/2}d\tau
=\frac{\sqrt\pi}{2}\frac{\Gamma(\gamma+1)}{\Gamma(\gamma+3/2)}E^{\gamma+1/2}$,
and the constant factor in the second expression equals
$\int_0^1(1-s^2)^\gamma\,ds$. Dividing by
$t^{2\gamma}$ proves \eqref{eq:interval-riesz}. Equality at
$\gamma=t=1$ shows that $1/3$ cannot be increased uniformly in the
stated range.
\end{proof}

\begin{lemma}[A surplus from cosine trials]\label{lem:thin-uniform-surplus}
Let $p=d-1\ge2$. If $K\ge24p^2$ and $0<T\le\sqrt p$, then
\begin{equation}\label{eq:all-radii-thin-surplus}
 \frac{p!}{2K^p}S(\rho,K)>\frac1{10}.
\end{equation}
\end{lemma}
\begin{proof}
For $j>0$ and $t>0$, put
\[
 f_j(s)=(1-s^2)_+^{j/2},\qquad
 \mathcal L_j(t)=\sum_{n\ge0}f_j(n/t).
\]
The substitution $s=\sin u$ gives
$\int_0^1f_j(s)\,ds=I_{j+1}$, with $I_j$ as in
\eqref{eq:moment}. Lemma~\ref{lem:interval-riesz} with exponent
$\gamma=p/2\ge1$ therefore gives
\begin{equation}\label{eq:all-radii-left-sampling}
 \mathcal L_p(t)-tI_{p+1}
 \ge\frac13.
\end{equation}
Monotonicity of $f_{p-1}$ also gives
$\mathcal L_{p-1}(t)\le1+tI_p$.

For angular eigenvalue $q=\ell(\ell+p-1)$, write
$\lambda_{q,n}(\rho)$, $n\ge0$, for the ordered radial
eigenvalues in that sector. Use the first $n+1$ Neumann cosine
functions on $(\rho,1)$ as a radial trial space.
The weighted derivative energy is at most its unweighted
counterpart, the weighted mass is at least $\rho^p$ times the
unweighted mass, and $q/s^2\le q/\rho^2$ on the shell. Hence
\begin{equation}\label{eq:all-radii-cosine}
 \lambda_{q,n}(\rho)
 \le\rho^{-p}\left(\frac{\pi n}{h}\right)^2+\rho^{-2}q.
\end{equation}

The strict angular count satisfies, with $\alpha=(p-1)/2$,
\begin{equation}\label{eq:all-radii-angular}
 \sum_{\ell(\ell+p-1)<s^2}\kappa_{p+1,\ell}
 \ge\frac2{p!}(s-\alpha)_+^p
 \ge\frac2{p!}(s-p/2)_+^p.
\end{equation}
Indeed, for $s>0$ put
$w=\sqrt{s^2+\alpha^2}-\alpha$ and $m=\lceil w\rceil-1$.
The cumulative multiplicity in \eqref{eq:angular-exact} gives
\[
 \sum_{\ell=0}^m\kappa_{p+1,\ell}
 =\frac{(2m+p)(m+1)\cdots(m+p-1)}{p!}
 \ge\frac{2(m+1)^p}{p!}\ge\frac{2w^p}{p!}.
\]
Since $w\ge s-\alpha$, this proves
\eqref{eq:all-radii-angular}; at $s=0$ both lower bounds vanish.

Set $\widehat T=\rho^{p/2}T$. Apply
\eqref{eq:all-radii-angular} to the angular cutoffs for which the
right side of \eqref{eq:all-radii-cosine} is strictly below $K^2$.
Summing over radial rows and using
$(v-a)_+^p\ge v^p-pa v^{p-1}$ yields
\begin{align*}
 \frac{p!}{2K^p}N^<(K^2)
 &\ge\sum_{n\ge0}
 \left(\rho\sqrt{\left(1-\frac{n^2}{\widehat T^2}\right)_+}
                    -\frac p{2K}\right)_+^p\\
 &\ge\rho^p\mathcal L_p(\widehat T)
       -\frac{p^2}{2K}\mathcal L_{p-1}(\widehat T).
\end{align*}
The omitted factor $\rho^{p-1}$ in the last term is at most one.
Rows with $n\ge\widehat T$ contribute zero to these lower bounds,
so the estimate respects strict counting also at radial endpoints.

The exact moment identity gives
\[
 \frac{p!}{2K^p}\W_{p+1}(K,\rho K)
 =TI_{p+1}\frac{1+\rho+\cdots+\rho^p}{p+1}
 \le TI_{p+1}.
\]
Using \eqref{eq:all-radii-left-sampling},
$\rho^p\ge1-ph$, and $1-\rho^{3p/2}\le3ph/2$, we obtain
\begin{align*}
 \frac{p!}{2K^p}S(\rho,K)
 &\ge\frac{\rho^p}{3}-(1-\rho^{3p/2})TI_{p+1}
       -\frac{p^2}{2K}(1+TI_p)\\
 &\ge\frac13-\frac{B_p}{K},
\end{align*}
where we used $T\le\sqrt p$ and set
\[
 B_p=\frac{\pi p^{3/2}}3
       +\frac{3\pi p^2}{2}I_{p+1}
       +\frac{p^2}{2}(1+\sqrt p\,I_p).
\]
For every $p\ge2$, monotonicity of the cosine moments gives
$I_{p+1}\le I_3=2/3$. Also $\cos u\le e^{-u^2/2}$ gives
$\sqrt p\,I_p<\sqrt{\pi/2}$. Consequently,
\[
 \frac{B_p}{p^2}
 <\frac{\pi}{3\sqrt2}+\pi+\frac{1+\sqrt{\pi/2}}2
 <\pi+2<\frac{16}{3}.
\]
Indeed, $\pi<22/7$ bounds the first and last terms of the middle
expression by $3/4$ and $5/4$, respectively.
Since $K\ge24p^2$, it follows that
\[
 \frac{p!}{2K^p}S(\rho,K)
 >\frac13-\frac29=\frac19>\frac1{10}.
\]
This proves the lemma.
\end{proof}

\begin{proof}[Proof of Theorem~\ref{thm:effective-surplus} under
\eqref{eq:all-radii-explicit-onset}]
For $d\ge3$, $K_d^*>k_d$: in dimensions three and four this is
$96>44/5$ and $216>12$, respectively, while for $p=d-1\ge4$,
$k_d=\tfrac74(p+1)^{3/2}<\tfrac{21}{4}p^2<24p^2$.
The width thresholds satisfy
\[
 \frac{w_d}{\pi}=
 \begin{cases}
  4/3,&d=3,\\
  16/(3\pi),&d=4,\\
  3\sqrt{d-1}/4,&d\ge5,
 \end{cases}
 \qquad \frac{w_d}{\pi}<\sqrt{d-1}.
\]
For $T\ge w_d/\pi$, the first regime of the theorem applies.
For $T\le\sqrt{d-1}$, Lemma~\ref{lem:thin-uniform-surplus}
gives the larger coefficient $1/[5(d-1)!]>c_d$.
These ranges overlap and cover every $T>0$.
For $d=2$, Proposition~\ref{prop:planar-surplus} already applies
for every radius ratio when $K\ge K_2^*=8$.
We have therefore proved
\begin{equation}\label{eq:all-radii-explicit-tail}
 S(\rho,K)>c_dK^{d-1}
 \qquad(0<\rho<1,\ K\ge K_d^*).
\end{equation}
Scaling restores the asserted dependence on $R$.
\end{proof}

\subsection{Enhancement at every frequency}
The uniform high-frequency estimate covers large $K$. For bounded $K$, the
two limiting geometries require different estimates: strictness at
the ball endpoint and upper semicontinuity of the eigenvalues handle
$\rho\downarrow0$, and the zero Neumann mode handles $\rho\uparrow1$.
Put $S(0,K)=N^<_{B_1}(K^2)-W_d(K)$.

\begin{lemma}[A positive margin up to the ball]
\label{lem:compact-surplus}
For every $0<\rho_1<1$ and every $A>0$ there exists $m>0$ such that
$S(\rho,K)\ge m$ for $0\le\rho\le\rho_1$ and $0<K\le A$.
\end{lemma}
\begin{proof}
Write $\lambda_j(\rho)$ for the ordered Neumann eigenvalues of
$\Omega_{\rho,1}$, and $\lambda_j(0)$ for those of $B_1$.

\emph{Continuity for $\rho>0$.} Fix $0<\rho_-<\rho_+<1$. The radial maps
\[
 \Psi_\rho(s\theta)=
 \left[\rho+\frac{1-\rho}{1-\rho_-}(s-\rho_-)\right]\theta,
 \qquad \rho_-\le s\le1,
\]
pull the Neumann energy and mass forms back to the fixed domain
$\Omega_{\rho_-,1}$. Their coefficients vary continuously with
$\rho\in[\rho_-,\rho_+]$ and are uniformly elliptic. The corresponding
forms at nearby parameters differ by relative factors tending to one.
The min--max principle therefore shows that each $\lambda_j$ is
continuous on $(0,1)$, including at multiplicities.

\emph{Upper semicontinuity at $\rho=0$.} Let $u_0,\ldots,u_j$ be
$L^2$-orthonormal Neumann eigenfunctions of $B_1$ for
$\lambda_0(0),\ldots,\lambda_j(0)$. As $\rho\downarrow0$, the Gram and
energy matrices of their restrictions to $\Omega_{\rho,1}$ converge to
the identity and to $\operatorname{diag}(\lambda_i(0))$. The min--max
principle gives $\limsup_{\rho\downarrow0}\lambda_j(\rho)\le\lambda_j(0)$.
(The eigenvalues in fact converge as the hole shrinks; compare
\cite{RauchTaylor}. Only the upper bound is used.)

\emph{Semicontinuity of the count.} By the two preceding steps, each set
$\{(\rho,K)\in[0,\rho_1]\times(0,A]:\lambda_j(\rho)<K^2\}$ is relatively
open. Hence
\[
 N^<_{\Omega_{\rho,1}}(K^2)
 =\sup_{m\ge0}\sum_{j=0}^m
       \mathbf1_{\{\lambda_j(\rho)<K^2\}}
\]
is lower semicontinuous on $[0,\rho_1]\times(0,A]$, and so is $S$.
Extend $S$ to $K=0$ by $S(\rho,0)=1$; this is the limit appropriate to
positive frequencies, rather than the literal strict count at zero. The
extension is lower semicontinuous on $[0,\rho_1]\times[0,A]$ because
$N^<(K^2)\ge1$ for $K>0$ and $\W_d(K,\rho K)\le W_d(K)\to0$ as
$K\downarrow0$.

\emph{Positivity.} For $\rho>0$, Theorem~\ref{thm:all-dimensions} gives
$S(\rho,K)>0$. For $\rho=0$, Corollary~\ref{cor:balls} gives
$S(0,K)>0$ in every dimension. A positive lower semicontinuous function
attains a positive
minimum on the compact set $[0,\rho_1]\times[0,A]$.
\end{proof}

\begin{proof}[Proof of Theorem~\ref{thm:uniform-enp}]
Choose $\rho_1\in(0,1)$ so close to one that
$\W_d(K_d^*,\rho_1K_d^*)\le1/2$, and let $m>0$ be given by
Lemma~\ref{lem:compact-surplus} with $A=K_d^*$. Recall $p=d-1$ and
$K_d^*\ge1$.
\begin{itemize}
\item For $K\ge K_d^*$ and every $0<\rho<1$,
 \eqref{eq:all-radii-explicit-tail} gives $S(\rho,K)>c_dK^p\ge
 c_d(1+K^p)/2$.
\item For $\rho_1\le\rho<1$ and $0<K\le K_d^*$, the zero Neumann mode
 gives $S(\rho,K)\ge1-\W_d(K,\rho K)\ge1/2$.
\item For $0<\rho\le\rho_1$ and $0<K\le K_d^*$,
 Lemma~\ref{lem:compact-surplus} gives $S(\rho,K)\ge m$.
\end{itemize}
Consequently one may take
\[
 \varepsilon_d=\min\left\{\frac{c_d}{2},\
 \frac{\min(m,1/2)}{1+(K_d^*)^p}\right\}>0,
\]
which depends only on $d$. Scaling proves
\eqref{eq:uniform-enhancement}.
\end{proof}

\begin{remark}[Dependencies]\label{rem:enp-dependencies}
In every dimension this proof uses Theorems~\ref{thm:all-dimensions}
and~\ref{thm:effective-surplus} and Corollary~\ref{cor:balls}; it does
not use the removal theorem of \cite{HJLW} or a two-term Weyl law.
The planar high-frequency estimate is deduced from our own endpoint
bounds in Proposition~\ref{prop:planar-surplus}.
For every $d\ge2$, restricting Lemma~\ref{lem:compact-surplus} to
$\rho=0$ and combining it with the ball tail
\eqref{eq:ball-effective-surplus} shows that $B_1$ satisfies the enhanced
Neumann P\'olya inequality; compare \cite[Section~1.2]{HJLW}.
\end{remark}

\begin{remark}[Relation to the two-term Weyl law]
The coefficient and both thresholds in
Theorem~\ref{thm:effective-surplus} are explicit.
The constant $\varepsilon_d$ for all positive frequencies also depends on
the minimum in Lemma~\ref{lem:compact-surplus}, so it is not explicit. The two-term Weyl law was proved by Ivrii \cite{Ivrii1980}
under a condition on periodic billiards, and for the disk by Kuznetsov
and Fedosov \cite{KuznetsovFedosov}; see \cite{GMWW} for planar annuli.
For a fixed shell, Guo--Jiang--Wang--Yang~\cite[Theorem~1.1]{GJWY}
bound the remainder in the two-term formula:
\begin{equation}\label{eq:shell-two-term}
 N^<_\Omega(\Lambda)
 =L_{0,d}|\Omega|\Lambda^{d/2}
 +\frac14L_{0,d-1}|\partial\Omega|\Lambda^{(d-1)/2}
 +O_\Omega(\Lambda^{d/2-2/3}).
\end{equation}
Their inclusive convention yields the same estimate for the strict
count by approaching each threshold from below. This fixed-domain
asymptotic does not by itself provide a threshold uniform in both
degenerating limits of the radius ratio.
\end{remark}

\subsection{Uniform removal of a small hole}
\begin{proof}[Proof of Corollary~\ref{cor:uniform-perforated-shells}]
We use the removal argument of
He--Jiang--Lin--Wang~\cite[Section~3]{HJLW}, with its low-energy
estimate made uniform over the host shells.
Write $\gamma=(d-1)/2$. By Theorem~\ref{thm:uniform-enp},
we may take $C_1=C_2=\varepsilon_d$ in
\begin{equation}\label{eq:uniform-removal-host}
 N_{\Omega_\rho}^<(\Lambda)
 \ge L_{0,d}|\Omega_\rho|\Lambda^{d/2}
          +\max\{C_1\Lambda^\gamma,C_2\}.
\end{equation}
The $O(\Lambda^{(d-1)/2})$ remainder estimate in Weyl's law for smooth
domains \cite{SafarovVassiliev} (see also
\cite[Lemma~2.2 and Section~3]{HJLW}) gives $\alpha_U>0$ such that
\begin{equation}\label{eq:uniform-removal-obstacle}
 N_U^<(\Lambda)
 \le L_{0,d}|U|\Lambda^{d/2}+\alpha_U\Lambda^\gamma+1
 \qquad(\Lambda>0).
\end{equation}
Indeed, the remainder is $O_U(\Lambda^\gamma)$ at infinity.
Connectedness gives $N_U^<(\Lambda)=1$ below its first positive
eigenvalue, so increasing $\alpha_U$ extends this estimate to every
positive frequency.

Set
\[
 \Lambda_1=\left(\frac4{C_1}\right)^{1/\gamma},\qquad
 V_{\max}=\omega_d(1-\rho_-^d).
\]
Let $E_\rho$ be the span of the Neumann eigenfunctions of
$\Omega_\rho$ whose eigenvalues are at most $\Lambda_1$.  There is a
finite constant $M=M(d,I,\Lambda_1)$ such that
\begin{equation}\label{eq:uniform-removal-spectral-bound}
 \|f\|_{L^\infty(\Omega_\rho)}^2
 \le M\|f\|_{L^2(\Omega_\rho)}^2
 \qquad(\rho\in I,\ f\in E_\rho).
\end{equation}
For completeness, fix $\rho_0\in I$ and use the radial maps
\[
 F_\rho(s\theta)=
 \left(\rho+\frac{1-\rho}{1-\rho_0}(s-\rho_0)\right)\theta,
 \qquad \rho_0\le s\le1,\quad\theta\in\mathbb S^{d-1}.
\]
These maps and their inverses have uniformly bounded derivatives of
every fixed order, and their Jacobians are uniformly bounded away
from zero.  After pullback to $\Omega_{\rho_0}$, the Neumann problems
therefore have uniformly elliptic coefficients and uniform boundary
regularity.  Iterating the Neumann elliptic estimate gives, for each
integer $m\ge1$,
\[
 \|f\|_{H^{2m}(\Omega_\rho)}
 \le C_{m,I}\sum_{j=0}^m\|(-\Delta)^j f\|_{L^2(\Omega_\rho)}
 \le C_{m,I}\left(\sum_{j=0}^m\Lambda_1^j\right)
                          \|f\|_{L^2(\Omega_\rho)}.
\]
Here every $(-\Delta)^j f$ satisfies the Neumann boundary condition,
because $f$ is a finite sum of Neumann eigenfunctions.  Choosing
$2m>d/2$ and applying the Sobolev embedding, whose constant is also
uniform under these maps, proves
\eqref{eq:uniform-removal-spectral-bound}.

Choose
\begin{equation}\label{eq:uniform-removal-size}
 \eta=\min\left\{\frac12,
       \frac{C_2}{dL_{0,d}V_{\max}\Lambda_1^{d/2}}\right\},
 \qquad
 a_0=\min\left\{
       \left(\frac{C_1}{2\alpha_U}\right)^{1/(d-1)},
       \left(\frac{\eta}{|U|M}\right)^{1/d}\right\}.
\end{equation}
Fix an admissible $\rho,a,y$ and abbreviate
$\Omega=\Omega_\rho$, $\Omega'=\Omega'_{\rho,a,y}$.
For $f\in E_\rho$, \eqref{eq:uniform-removal-spectral-bound} gives
\[
 \int_{\Omega'}|f|^2
 \ge\bigl(1-a^d|U|M\bigr)\int_\Omega|f|^2
 \ge(1-\eta)\int_\Omega|f|^2,
 \qquad
 \int_{\Omega'}|\nabla f|^2\le\int_\Omega|\nabla f|^2.
\]
Restriction is thus injective on $E_\rho$.  Applying the min--max
principle to the first $j+1$ host eigenfunctions yields
\[
 \lambda_j^N(\Omega')
 \le\frac{\lambda_j^N(\Omega)}{1-\eta}
 \qquad\text{whenever }\lambda_j^N(\Omega)\le\Lambda_1.
\]
Consequently, for $0<\Lambda\le\Lambda_1$,
\[
 N_{\Omega'}^<(\Lambda)
 \ge N_\Omega^<((1-\eta)\Lambda).
\]
Using \eqref{eq:uniform-removal-host},
$(1-\eta)^{d/2}\ge1-d\eta/2$, and $|\Omega'|\le|\Omega|$, we obtain
\begin{equation}\label{eq:uniform-removal-low}
 \begin{split}
 N_{\Omega'}^<(\Lambda)-L_{0,d}|\Omega'|\Lambda^{d/2}
 &\ge C_2-\frac d2L_{0,d}|\Omega|\eta\Lambda^{d/2}\\
 &\ge\frac{C_2}{2}
 \qquad(0<\Lambda\le\Lambda_1).
 \end{split}
\end{equation}

For higher frequencies, Neumann bracketing, scaling, and
\eqref{eq:uniform-removal-obstacle} give
\[
 \begin{split}
 N_{\Omega'}^<(\Lambda)
 &\ge N_\Omega^<(\Lambda)-N_U^<(a^2\Lambda)\\
 &\ge L_{0,d}|\Omega'|\Lambda^{d/2}
       +(C_1-\alpha_Ua^{d-1})\Lambda^\gamma-1.
 \end{split}
\]
Since $a<a_0$ and $\Lambda_1^\gamma=4/C_1$, this implies
\begin{equation}\label{eq:uniform-removal-high}
 N_{\Omega'}^<(\Lambda)-L_{0,d}|\Omega'|\Lambda^{d/2}
 \ge\frac{C_1}{4}\Lambda^\gamma
 \qquad(\Lambda\ge\Lambda_1).
\end{equation}
Finally put
\[
 C_2'=\min\{C_2/2,1\},\qquad C_1'=C_1C_2'/4.
\]
Then $C_1'\Lambda_1^\gamma=C_2'$.  Below $\Lambda_1$, the maximum
of the two remainder terms is $C_2'\le C_2/2$; above
$\Lambda_1$, it is $C_1'\Lambda^\gamma\le(C_1/4)\Lambda^\gamma$.
Thus \eqref{eq:uniform-removal-low} and
\eqref{eq:uniform-removal-high} prove the assertion with constants
independent of the radius ratio within the fixed family and of the hole location.
\end{proof}

The constants in \eqref{eq:uniform-removal-size} are not explicit:
$C_1=C_2=\varepsilon_d$ is the all-frequency enhancement constant,
$\alpha_U$ comes from the counting estimate for $U$, and $M$ is a
uniform spectral-subspace bound on the compact host family.
Although $\varepsilon_d$ is independent of the radius ratio,
the bound on $M$ uses $I\Subset(0,1)$; the admissible size $a_0$
therefore depends on $I$. For a shell with outer radius $R$,
the same result allows holes $y+aR\,U$ with $a<a_0$; the constants in
the enhanced inequality are $R^{d-1}C_1'$ and $C_2'$.  In particular, fixing a radius ratio
recovers the small-hole statement for each individual concentric shell.

\section{Conclusion and further questions}\label{sec:discussion}

The strict shell and ball inequalities hold in every dimension and at
every positive energy. Their proof compares inverse radial actions while
retaining the removed volume exactly, and combines this comparison with
localized phase estimates and radial trial spaces. In dimension two,
cancellation of the singular terminal slope and concavity in the squared
inner frequency give the required linear-weight estimate. The positive
margins in these proofs yield enhancement uniform in the radius ratio. For each fixed smooth hole
shape, this enhancement permits uniform removal of a sufficiently small
hole from shells whose radius ratios lie in a fixed compact subinterval
of $(0,1)$.

For a fixed normalized shell, the boundary coefficient in the
two-term Weyl law is $(1+\rho^{d-1})/[2(d-1)!]$~\cite{GJWY}.
The explicit coefficients proved here are smaller, but their frequency
threshold is uniform as the hole shrinks or the shell becomes thin.
In dimension two, our coefficient $1/4000$ applies from $K=8$,
whereas Dai--Sun's coefficient $1/20$ applies for $K>80$
(Remark~\ref{rem:planar-dai-sun}). Improving the coefficient while
retaining a uniform threshold is one quantitative question. Others are
to improve the dimension dependence of the threshold and to give an
explicit lower bound for the all-frequency constant $\varepsilon_d$.

The comparison separates the change in the inverse action from the
radial phase estimate. For spherical cones or higher-dimensional annular
sectors, its use would require angular counting estimates in place of
the explicit spherical multiplicities. Beyond concentric geometry, a
direct action argument would also require a replacement for the radial
phase input. Uniform hole removal from degenerating families of host
shells would require spectral-subspace estimates that remain uniform in
those limits.

\appendix
\section{Details of the phase enclosures}
\label{app:phase-details}
\subsection{An elementary half-line bound}
\label{app:proof-lem-phase-new-airy}
\begin{proof}[Proof of Lemma~\ref{lem:phase-new-airy}]
Put
\[
 y=\sqrt{x+\frac{121}{100}},\qquad t=y-\frac{11}{10},\qquad
 Q(x)=y-\frac{9}{10(y+2/15)}.
\]
Then $t\ge0$, $Q(0)=37/100+1/3700$, and
\begin{align*}
 2y(y+2/15)^2\left(x-\frac{29}{50}-Q^2+Q'\right)
 & =\frac12(t-1/5)^2\\
 &\quad+\frac{619+5110t+41100t^2+9000t^3}{450000}>0.
\end{align*}
Completion of squares now gives, for $0\le c\le37/100$,
\begin{align*}
 &\int_0^\infty\bigl(|v'|^2+x|v|^2\bigr)\,dx
       -c|v(0)|^2-\frac{29}{50}\int_0^\infty|v|^2\,dx\\
 &\quad=\int_0^\infty|v'+Qv|^2\,dx
       +\int_0^\infty\left(x-\frac{29}{50}-Q^2+Q'\right)|v|^2\,dx\\
 &\qquad +(Q(0)-c)|v(0)|^2>0
\end{align*}
for every nonzero $v$ in the form domain. The identity follows first for
smooth functions on $[0,\infty)$ with compact support, and then by density, since
$Q^2=O(1+x)$ and $Q'$ is bounded. Applying it to a lowest eigenfunction
proves the assertion.
\end{proof}

\subsection{The phase cutoff}
\label{app:proof-prop-phase-new-cutoff}
\begin{proof}[Proof of Proposition~\ref{prop:phase-new-cutoff}]
The substitution $w(r)=r^{(d-1)/2}u(r)$ transforms the ball form into
\[
 \int_0^1\left(|w'|^2+\frac{\nu^2-1/4}{r^2}|w|^2\right)\,dr
 -(\alpha+1/2)|w(1)|^2.
\]
For a regular first radial eigenfunction,
$w(r)=O(r^{\nu+1/2})$ and $w'(r)=O(r^{\nu-1/2})$ at zero.
Thus $\nu\ge19/2$ gives $w\in H^1(0,1)$ and $w(0)=0$.
Use $r^{-2}\ge1+2(1-r)$ and set $x=1-r$; the zero trace makes
extension by zero beyond $x=1$ admissible in the half-line form domain.
After scaling, the Robin parameter is
$c=(\alpha+1/2)/[2(\nu^2-1/4)]^{1/3}\le37/100$.
Lemma~\ref{lem:phase-new-airy} therefore places the first ball frequency
$s_1$ above
\[
 s_1^2>\nu^2-\frac14+
 \frac{29}{50}\bigl(2(\nu^2-1/4)\bigr)^{2/3}.
\]
Set $v=\nu^{-2/3}<9/40$, $b=23/25$, $k=17/250$,
$q=b-kv$, and $t_*^2=\nu^2+q\nu^{4/3}$.
Since $2^{2/3}>46/29$ and $(1-y)^{2/3}\ge1-y$,
\[
 s_1^2>\nu^2+b\nu^{4/3}-\frac14-\frac b4v>t_*^2;
\]
the last inequality follows from
$v/4+bv^2/4\le10863/160000<k$.
At this joining point, \eqref{eq:phase-new-condition} and
$2^{1/3}<63/50$ bound the sum of the FLPS and Robin errors by
\begin{equation}\label{eq:phase-new-scalar}
 F(v)=\frac7{24q^{3/2}}+\frac{3v}{8\sqrt q}
       +\arctan\frac{a-v/2}{\sqrt q},\qquad a=\frac{2331}{5000}.
\end{equation}
Use the tangent bound for the concave function $\arctan$ at
$z=\tan(\pi/8)=\sqrt2-1$. Put
\[
 A=\frac{a(2+\sqrt2)}4,\qquad
 B=\frac{1-\sqrt2}{8},\qquad
 C=\frac\pi8-\frac{\sqrt2}{4}.
\]
Then $F(v)\le H(v):=7/(24q^{3/2})+(A+Bv)/\sqrt q+C$.
Here $q>9/10$, $B<0$ and $kA+Bb<0$, so
\[
 2q^{5/2}H'(v)
 =\frac{7k}{8}+(kA+Bb)q+Bq^2
 \le\frac{7k}{8}+\frac9{10}(kA+Bb)+\frac{81}{100}B<0.
\]
The last expression decreases with $\sqrt2$ and is already negative
when $\sqrt2$ is replaced by $41/29$.
Thus $F(v)\le H(0)$. At this endpoint,
\[
 H(0)-\frac\pi8
 =\frac{875/552+(2331/4000)(2+\sqrt2)}{\sqrt{23}}
     -\frac{\sqrt2}{4}
 <\frac{157}{400}<\frac\pi8.
\]
For the first inequality, use $\sqrt{23}>139/29$ and then
$\sqrt2>41/29$; after the first substitution the coefficient of
$\sqrt2$ is negative. The last inequality uses $\pi>157/50$.
Hence $F(v)<\pi/4$ throughout the required range.

The joining argument proves the proposition.
\end{proof}

\section{Scalar inverse-action estimates}
\label{app:scalar-details}

\subsection{Two-sided inverse-action bounds}
\label{app:proof-lem-inverse-action-seven-tenths}
\begin{proof}[Proof of Lemma~\ref{lem:inverse-action-seven-tenths}]
For $\delta=1-\xi_K(t)/K$, integration of
$\arccos(1-s)\ge\sqrt{2s}$ gives
$\pi t/K\ge(2\sqrt2/3)\delta^{3/2}$, hence $\delta\le a_K$.
For the sharper bounds, write
\[
 J(\delta)=\frac3{2\sqrt2}\int_0^\delta\arccos(1-s)\,ds=a_K^{3/2}.
\]
The positive arcsine series gives
$J(\delta)\ge\delta^{3/2}(1+\delta/20+9\delta^2/1120)$.
For the lower inverse bound, put $e=a_K/30$ and $q=a_K(1-e)$.
Using $(1-e)^{3/2}\ge1-3e/2+3e^2/8$ and
$(1-e)^r\ge1-re$ for $r=5/2,7/2$ yields
\[
 \frac{J(q)}{a_K^{3/2}}
 \ge1+a_K^2\left(\frac3{700}-\frac{3a_K}{3200}\right)>1
 \qquad(0<a_K\le7/10).
\]
Thus $\delta<q$, proving the lower bound; $a_K=0$ follows by continuity.

For the upper bound, the series for $\arccos(1-s)/\sqrt{2s}$ starts with
$1+s/12+3s^2/160+5s^3/896$ and has successive coefficient ratios
less than $1/2$. For $s\le7/10$, after factoring $s^2$ out of the
quadratic term and tail, the remaining factor is bounded by
\[
 \frac3{160}+\frac5{896}\frac{s}{1-s/2}
 \le\frac{103}{4160}<\frac1{40}.
\]
Integration yields
\begin{equation}\label{eq:d4J}
 J(\delta)\le\delta^{3/2}
       \left(1+\frac\delta{20}+\frac{3\delta^2}{280}\right)
       \qquad(0\le\delta\le7/10).
\end{equation}
We prove the upper bound
$\xi_K(t)/K\le1-a_K+a_K^2/30+a_K^3/200$.
Put $e=a_K/30+a_K^2/200$ and $d=a_K(1-e)$. For $0<a_K\le7/10$,
\[
 (1-e)^{3/2}\le1-\frac{3e}{2}+\frac{e^2}{2}
 =1-\frac{a_K}{20}-\frac{a_K^2}{144}+\frac{a_K^3}{6000}
    +\frac{a_K^4}{80000}
 \le1-\frac{a_K}{20}-\frac{a_K^2}{150},
\]
since $a_K/6000+a_K^2/80000<1/3600$. Also
$d/20\le a_K/20-a_K^2/600$ and $d^2\le a_K^2$, so
\[
 \frac{J(d)}{a_K^{3/2}}
 \le\left(1-\frac{a_K}{20}-\frac{a_K^2}{150}\right)
       \left(1+\frac{a_K}{20}+\frac{19a_K^2}{2100}\right)
 \le1-\frac{a_K^2}{8400}<1.
\]
Monotonicity of $J$ proves this bound. The coefficient
$1/30+a_K/200$ of its quadratic correction is at most
$221/6000<1/25$ when $a_K\le7/10$.
The case $a_K=0$ follows by continuity. Finally, using $\pi^2<10$,
\[
 a_K^3\le\frac{225\pi^2}{128K^2}
 <\frac{28125}{85264}<\frac{343}{1000}\quad(K\ge73/10).
\]
\end{proof}

\section{Variational range estimates}
\label{app:variational-details}
\subsection{Assembly of the variational ranges}
\label{app:proof-prop-var-ranges}
\begin{proof}[Proof of Proposition~\ref{prop:var-ranges}]
Lemma~\ref{lem:d4-ground-direct} proves the four-dimensional
low-frequency range. For $d\ge5$, Lemma~\ref{lem:constant-low-frequency}
already covers $K\le K_*=(27/40)d^{3/2}$. We first establish the
thin-width criterion, then extend the low-frequency range by the
polynomial ground trial and one excited radial row.

The constant radial trial gives the exact thin-width criterion for
$0\le\rho<1$. Its quotient is $\beta_\rho q$ in angular parameter
$q$. For $q>0$ the bound is strict, because the potential $q/s^2$
is nonconstant, so a constant radial function cannot be an eigenfunction.
The degree-zero ground level is zero. The cumulative multiplicity
through degree $L$ is
\[
 \sum_{\ell=0}^{L}\kappa_{d,\ell}=\frac{2L+p}{p!}\prod_{i=1}^{p-1}(L+i).
\]
Pairing factors shows that the sum over
$\ell(\ell+p-1)\le\kappa^2$ exceeds
$2(\kappa-1/2)^p/p!$ when $\kappa>1/2$.
For $A_\rho=h^{-1}\int_\rho^1s^p\,ds$, H\"older gives
$\beta_\rho^{-p/2}\ge A_\rho$. The constant trial is strict in every
positive angular degree, so
\[
 N^<(K^2)>\frac{2K^p A_\rho}{p!}
 \left(1-\frac{\sqrt{\beta_\rho}}{2K}\right)^p.
\]
Dividing by $(1-\rho^d)W_d(K)$ proves
\eqref{eq:var-thin-exact}; its four-dimensional form uses
$B_4=16/3$ and $\beta_\rho\le2$.

For $d\ge5$, use the notation of
Lemma~\ref{lem:var-uniform-endpoints}.
The ball trial $u=r^{2d-2}(1-(d-1)r^2/d)$ gives
$\lambda_{q,0}(0)<(q+b_d)/a_d$ for $q\ge0$.
Its Rayleigh quotient is positive, and the radial operator lowers
each monomial by two degrees; hence this polynomial cannot satisfy
an eigenvalue equation with its positive quotient. The Rayleigh bound
is therefore strict, although the trial satisfies the outer Neumann
condition.
To verify the coefficients, put $x=r^2$. Twice the difference
between $K^2$ times the squared radial norm and the radial energy is
\[
 \int_0^1x^{(d-4)/2}
 \{(K^2x-q)f^2-4x^2(f')^2\}\,dx.
\]
For $f=x^{d-1}(1-(d-1)x/d)$, the $x$ moment and derivative
energy divided by the angular norm are $a_d$ and $b_d$.
By the ground-level comparison, the same strict bound holds for
every $0\le\rho<1$. Consequently the ground row contributes at
least $G_d(v_0(K))$, including degrees at equality in the trial cutoff.
The half-unit shift accounts for taking the largest admissible
integer angular degree.
The continuous ratio $G_d(v_0(K))/W_d(K)$ has no interior minimum, by
Lemma~\ref{lem:var-trial-ratio-unimodal}; indeed
$b_d>(32/53)(d-1)^2>\alpha^2$.
The lower endpoint now follows from the overlap of the two trials.
Put $\beta=(d-2)/d$. The coefficients satisfy
\[
 a_d-\beta=\frac{32(d-1)(11d-6)}{5d\,\mathfrak p_d},\qquad
 \frac{b_d}{a_d-\beta}
 =\frac{5d^3(d-1)}{11d-6}
 <\frac5{11}d^3<K_*^2.
\]
Thus $a_dK_*^2-b_d>\beta K_*^2>0$, which verifies the domain
condition $K_*>\sqrt{b_d/a_d}$ of the ratio lemma.
For $d\ge6$, the continuous paired-product bound
\eqref{eq:var-continuous-constant} and the constant-trial endpoint
estimate in Appendix~\ref{app:constant-bridge} give
\[
 \frac{G_d(v_0(K_*))}{W_d(K_*)}
 >\frac{2(K_*\sqrt\beta-1/2)^p}{p!W_d(K_*)}>1.
\]
For $d=5$, direct substitution gives $v_0(K_*)>11/2$.
Since $K_*=27\sqrt5/8$, the bounds $\pi>3$ and $\sqrt5<9/4$ give
\[
 W_5(K_*)<\frac{3^{14}}{2^{16}}
 <\frac{4719}{64}=G_5(11/2)<G_5(v_0(K_*)).
\]
At $K_1$ the ratio exceeds one by
Lemma~\ref{lem:var-uniform-endpoints}. Its absence of an interior
minimum therefore covers $K_*\le K\le K_1$.
Together with the constant-trial range this proves the result for
every $K\le K_1$.

We can now prove the thin-width range. For $d\ge5$, the preceding
argument covers $K\le K_1$. For $K\ge K_1$, retain the Gamma
correction established in Lemma~\ref{lem:var-uniform-endpoints}:
\[
 \frac{B_d^2}{2\pi d}
 >1+\frac{d+2}{2(d+1)^2}>1+\frac2{5d}.
\]
Thus $B_d>\sqrt{2\pi(d+2/5)}$. Put $u=p^{-1/2}\le1/2$.
Bernoulli's inequality and $\beta_\rho\le d/(d-2)$ show that
\eqref{eq:var-thin-exact} follows from
\[
 \sqrt{\frac2\pi}\left\{
 \sqrt{1+\frac75u^2}
 -\frac u3\sqrt{\frac{1+7u^2/5}{(1+u^2)^2(1-u^2)}}
 \right\}>\frac34.
\]
For $0\le u\le1/2$, we have
$(1+u^2)^2(1-u^2)\ge1$ and
$\sqrt{1+7u^2/5}\ge1+3u^2/5$.
Thus the expression in braces is at least
\[
 \left(1+\frac35u^2\right)\left(1-\frac u3\right)
 \ge1-\frac u3+\frac12u^2
 =\frac{17}{18}+\frac12\left(u-\frac13\right)^2.
\]
Since $\pi<22/7$ and $(17/18)^2(7/11)>9/16$, this proves the
thin-width assertion for $d\ge5$.

For $d=4$, frequencies $K\le12$ are already covered. For $K\ge12$,
the explicit Gamma factor, $\pi<22/7$, and $\sqrt3<7/4$ give
\[
 \frac{B_4}{\pi}\left(1-\frac{1}{\sqrt2K}\right)^3
 >\frac{56}{33}\left(\frac{15}{16}\right)^3
 >\frac{21}{16}>\frac34\sqrt3.
\]
This proves the thin-width assertion and justifies its use below.

Return to $d\ge5$ and, for $K_1\le K\le K_2$, set
\[
 c=\frac{47}{20},\quad s=c\sqrt p,\quad
 r_*=1-s/K,\quad
 q_*=(1-9/(5p))(K-s)^2-\frac{d^2-1}{4}.
\]
Here $K-s>d^{3/2}$ and $1-9/(5p)\ge11/20$, so
$0<r_*<1$ and $q_*>0$. Since $c<3\pi/4$, the thin-width result
covers $\rho\ge r_*$. At this radius,
\eqref{eq:var-excited} and $\pi^2/c^2<9/5$ give
$\lambda_{q_*,1}(r_*)<K^2$.
At the ball endpoint use the orthogonal trial functions
\[
 u_0(r)=r^{(p-1)/2},\qquad
 u_1(r)=\bigl((p+1)r^2-p\bigr)u_0(r).
\]
Their squared radial norms are $1/(2p)$ and $1/(2(p+2))$.
Put $z=K/\sqrt p$, $q=q_*$, and let $M$ and $Q$ denote the
mass and energy matrices in this basis. Direct integration gives
\[
 2p(K^2M-Q)=\begin{pmatrix}A&B\\ B&C\end{pmatrix},
\]
where
\begin{align*}
 A&=pz^2-\frac{pq}{p-1}-\frac{p(p-1)}4,\\
 B&=\frac{q}{p-1}-\frac{(p-1)(4p+3)}4,\\
 C&=\frac{p^2z^2}{p+2}-\frac{pq}{p-1}-\frac{p(17p+15)}4.
\end{align*}
The required frequencies lie in
\[
 L:=\frac{6p+9}{4}\le z\le U:=\frac{7p+12}{4}.
\]
Indeed, squaring the endpoint comparisons reduces them to
$3p+4>0$ and $21p^2-3p-49>0$.
On this interval,
\[
 B=\left(1-\frac4{5(p-1)}\right)(z-c)^2-p^2-\frac3{4(p-1)},
 \qquad 0<B<3(p+1)^2.
\]
For the lower bound, use $z-c>5p/4$ and
$1-4/(5(p-1))\ge11/15$; for the upper bound, use
$z-c<7(p+1)/4$.

It remains to prove that $H:=C-B/4$ is positive.
Its quadratic coefficient in $z$ is
\[
 -\frac{29p^2-71p-18}{20(p+2)(p-1)}<0,
\]
so it suffices to check the endpoints. They are
\begin{align*}
 H(L)&=\frac{75p^4-650p^3+2286p^2+14906p+768}
                  {2000(p+2)(p-1)},\\
 H(U)&=\frac{275p^4+9405p^3+55139p^2+89879p+6042}
                  {8000(p+2)(p-1)}.
\end{align*}
Both are positive: for the first, use
$75p^2-650p+2286=75(p-5)^2+100p+411$.
Finally,
\[
 A=\frac{p+2}{p}C+2B+6(p+1)^2>4B,
\]
and hence
\[
 Ax^2+2Bxy+Cy^2
 =B(2x+y/2)^2+(A-4B)x^2+Hy^2>0
 \qquad((x,y)\ne(0,0)).
\]
Thus $K^2M-Q$ is positive definite throughout the required interval,
and min--max gives the strict ball bound $\lambda_{q_*,1}(0)<K^2$.

Together with \eqref{eq:var-endpoint}, this supplies a first excited row for every $0\le\rho\le r_*$, with count at least
\[
 G_d(v_1(K)),\qquad
 v_1(K)=\sqrt{(1-9/(5p))(K-s)^2-d+5/4}-\frac12.
\]
Pairing roots gives $G_d'(v)/G_d(v)\ge p/v$, so
\[
 \frac{d}{dK}\log\frac{G_d(v_1(K))}{K^d}
 >\frac p{K-s}-\frac dK>0
 \qquad(K\le K_2<ds).
\]
The absence of an interior minimum gives a ground-row ratio greater
than $451/500$ throughout $[K_1,K_2]$: its two endpoint values exceed
$1$ and $451/500$, respectively. The excited-row ratio increases,
so its minimum is at $K_1$, where it exceeds $49/500$.
Both assertions use Lemma~\ref{lem:var-uniform-endpoints}.
Their sum exceeds $451/500+49/500=1$. Therefore
$N^<(K^2)>W_d(K)$ throughout $K_1\le K\le K_2$,
which completes the proof.
\end{proof}

\subsection{Uniform endpoint estimates}
\label{app:var-endpoints}
\begin{proof}[Proof of Lemma~\ref{lem:var-uniform-endpoints}]
We use the continuous parameter $u=d^{-1/2}$ and the fixed bound
$0<u\le1/\sqrt5< U:=9/20$.
Pairing the roots of $G_d$ gives
\begin{equation}\label{eq:var-product-minorant}
 G_d(v)\ge\frac{2v^p}{p!}
 \left(1-\frac{p(p-1)(p-2)}{24v^2}\right)
\end{equation}
when the expression in parentheses is positive. This follows from
$\prod_j(1-x_j)\ge1-\sum_jx_j$ for $0\le x_j\le1$.
Let
\[
 B_d=\frac{2^{d+1}\Gamma(d/2+1)^2}{d!},\qquad
 R_d=\frac{B_d}{\sqrt{2\pi d}}.
\]
The Gamma recurrence and the Wallis limit give
\[
 \log R_d=-\frac12\sum_{j=0}^\infty
 \log\left(1-\frac1{(d+2j+1)^2}\right)
 >\frac{d+2}{4(d+1)^2}
 =:g(u)=\frac{u^2(1+2u^2)}{4(1+u^2)^2}.
\]
Indeed, after $-\log(1-x)\ge x$, the trapezoid bound for the convex
function $(d+1+2x)^{-2}$ proves the last inequality.
Consequently, if $K=\vartheta d^{3/2}$ and $v/K\ge1-D(u)>0$, set
\[
 n(u)=(1-u^2)(1-2u^2)(1-3u^2),\qquad
 S(u)=\frac{n(u)}{24\vartheta^2(1-D(u))^2}.
\]
Whenever $S<1$, \eqref{eq:var-product-minorant} and the elementary
logarithm bounds imply
\begin{equation}\label{eq:var-log-majorant}
 \frac{G_d(v)}{W_d(K)}>
 \frac{\sqrt{2\pi}}{\vartheta}\exp\{-H(u)\},\qquad
 H=(u^{-2}-1)\left(D+\frac{D^2}{2(1-D)}\right)-g+\frac S{1-S}.
\end{equation}

We first treat the two ground endpoints. For
$\mathfrak r_{\vartheta}=1-a_d+(b_d-\alpha^2)/(\vartheta^2d^3)$,
\[
 \mathfrak r_{3/2}\le\frac{83}{100d}+\frac1{2d^2},\qquad
 \mathfrak r_{7/4}\le\frac{63}{80d}+\frac2{5d^2}.
\]
For completeness, with $y=d-5$, the differences between the right
and left sides are respectively
\[
 \frac{51y^4+1872y^3+36818y^2+224880y+415675}
 {900d^3\mathfrak p_d},
\]
\[
 \frac{y^2(59y^2-90y+39504)+343482y+712405}
 {3920d^3\mathfrak p_d},
\]
both positive for $y\ge0$.
For $0\le r\le93/500$, we have $\sqrt{1-r}>9/10$ and
\[
 1-\sqrt{1-r}=\frac r2+
 \frac{r^2}{2(1+\sqrt{1-r})^2}
 \le\frac r2+\frac{50r^2}{361}.
\]
It follows that the required deficits $1-v_0(K)/K$ are bounded by
\[
 D_A=\frac{83}{200}u^2+\frac13u^3+\frac{37}{100}u^4
 \quad(\vartheta=3/2),\qquad
 D_0=\frac{63}{160}u^2+\frac27u^3+\frac{61}{200}u^4
 \quad(\vartheta=7/4).
\]
We give one elementary bound for both logarithms. Write
$D=u^2e(u)$, $e=A+Bu+Cu^2$; the two triples $(A,B,C)$ are the
coefficients just displayed, and $e(u)<2/3$ for $0\le u\le U$.
Put $t=u^2\le1/5$. The identity
\[
 (1-7t/2)(1-2t/3)^2-n(u)
 =t\left\{\frac16+(1-5t)(1-8t/9)\right\}\ge\frac t6
\]
gives
\[
 \frac S{1-S}\le\frac{1-7u^2/2}{24\vartheta^2-1}.
\]

Also $g(u)\ge u^2/4-u^6/4$. Since $e<1$,
\[
 (u^{-2}-1)\left(D+\frac{D^2}{2(1-D)}\right)
 \le e-u^2e+\frac{u^2e^2}{2}.
\]
Thus $H(u)\le\mathcal C(u)$, where the concave cubic is
\[
 \mathcal C(u)=A+\frac1{24\vartheta^2-1}+Bu-a_2u^2-B(1-A)u^3,
 \qquad
 a_2=A-\frac{A^2}{2}+\frac14+\frac7{2(24\vartheta^2-1)}-C>\frac14.
\]
Indeed, the omitted terms are nonpositive because, for both triples,
\[
 C(1-A)-\frac{B^2}{2}>\frac7{50}>\frac18>
 BCU+\left(\frac{C^2}{2}+\frac14\right)U^2.
\]
Since $\mathcal C''<-1/2$, the tangent quadratic gives
$\mathcal C(u)\le\mathcal C(t)+\mathcal C'(t)^2$ for $t\ge0$.
Substitution of $t=21/50$ for $D_A$ and $t=9/25$ for $D_0$ gives
\begin{equation}\label{eq:var-ground-log-bounds}
 H_A<\frac{511}{1000},\qquad H_0<\frac{1849}{4000}.
\end{equation}

For the excited endpoint use the quartic deficit
\[
 D_1(u)=\frac{37}{15}u^2+\frac13u^3-\frac3{10}u^4.
\]
To verify it, put $t=u^2\le1/5$ and
\[
 J(t)=1-\frac{37}{15}t+\frac3{10}t^2,\qquad
 T(t)=\frac{1-14t/5}{1-t}
       \left(1-\frac{47}{30}t(1-t/2)\right)^2
       -\frac49t^2+\frac59t^3.
\]
The bound $\sqrt{1-t}\le1-t/2$ gives
$((v_1(K_1)+1/2)/K_1)^2\ge T(t)$. Direct expansion yields
\[
 \frac{(1-t)(T-J^2)}{t^2}
 =\frac{659}{900}-\frac{20471}{4500}t
   +\frac{10721}{2000}t^2-\frac{14653}{9000}t^3
 \ge\frac{73}{100}-\frac{23}{5}t+5t^2\ge\frac1{100}.
\]
Since $J>0$, this proves $v_1(K_1)/K_1\ge1-D_1(u)$.

Put $E(u)=-(u^{-2}-1)\log(1-D_1(u))$ and write
$D_1=u^2e_1$, where $e_1=37/15+u/3-3u^2/10$.
For $0<u\le U<1/2$, we have $37/15\le e_1<8/3$ and
\[
 (1-u^2)e_1'+ue_1(e_1-2)
 \ge\frac13+u\left(\frac{259}{225}-\frac35-\frac16\right)
 =\frac13+\frac{173}{450}u.
\]
Differentiating and using
$-\log(1-x)\le x+x^2/(2(1-x))$ therefore gives
\[
 E'(u)\ge
 \frac{(1-u^2)e_1'+ue_1(e_1-2)}{1-u^2e_1}>\frac13.
\]
Moreover, with $t=u^2$,
\[
 (1-D_1)^2-n(u)\ge(1-8t/3)^2-n(u)
 =\frac t9\{6(1-3t)^2+t\}>0.
\]

Thus $S/(1-S)\le1/53$. Retaining the exact logarithm in the
derivation of \eqref{eq:var-log-majorant} gives the bound with
$H_1=E-g+1/53$, which is increasing because
\[
 g'(u)=\frac{u(1+3u^2)}{2(1+u^2)^3}\le\frac u2,
 \qquad H_1'(u)>\frac13-\frac u2>0.
\]
At the actual endpoint $u=1/\sqrt5$, we have $g(u)=7/144$ and
\[
 D_1(u)=\frac{361}{750}+\frac1{15\sqrt5}<\frac{409}{800},
 \qquad -\log\frac{391}{800}<\frac{179}{250}.
\]
The first inequality follows from $\sqrt5>223/100$; for the second,
$\sum_{j=0}^6(179/250)^j/j!>800/391$. Therefore
\[
 H_1(u)\le H_1(1/\sqrt5)
 <\frac{358}{125}-\frac7{144}+\frac1{53}<\frac{567}{200}.
\]

Finally, the classical bound $\pi>223/71$ and
\[
 e^{511/500}<\frac{139}{50},\qquad
 e^{1849/2000}<\frac{2521}{1000},\qquad
 e^{567/100}<\frac{2901}{10}
\]
turn \eqref{eq:var-log-majorant} into the three claimed estimates.
Their squared lower bounds are respectively
\[
 \frac{89200}{88821}>1,\qquad
 \frac{7136000}{8770559}>\left(\frac{451}{500}\right)^2,\qquad
 \frac{17840}{1853739}>\left(\frac{49}{500}\right)^2.
\]

The displayed exponential bounds follow by the geometric Taylor
remainder
\[
 e^x\le\sum_{j=0}^m\frac{x^j}{j!}
 +\frac{x^{m+1}}{(m+1)!}\frac1{1-x/(m+2)}
 \quad(0\le x<m+2),
\]
with $m=3,3,12$, respectively. This completes an analytic proof
uniform in $d$.
\end{proof}

\subsection{The constant-trial range}
\label{app:constant-bridge}
\begin{proof}[Proof of Lemma~\ref{lem:constant-low-frequency}]
Put $p=d-1$, $\alpha=(d-2)/2$, $\beta=(d-2)/d$, and
\[
 B_d=\frac{2^{d+1}\Gamma(d/2+1)^2}{d!}.
\]
The constant radial trial gives
\[
 \lambda_{q_\ell,0}(0)<\frac{q_\ell}{\beta},\qquad
 q_\ell=\ell(\ell+d-2),\qquad \ell>0.
\]
The inequality is strict because the constant radial function cannot be
an eigenfunction when $q_\ell>0$. The degree-zero ground level is zero.

We use a continuous lower bound for the angular count. Set
$\kappa=K\sqrt\beta>1/2$, $w=\sqrt{\kappa^2+\alpha^2}-\alpha$,
$L=\lfloor w\rfloor$, and $v=w+(p-2)/2$. The cumulative multiplicity is
\[
 \widehat C_d(L):=\sum_{\ell=0}^{L}\kappa_{d,\ell}
 =G_d(L+p/2)>G_d(v),
\]
since $L+1>w$ and $G_d$ increases on $v>(p-2)/2$.
Pair the factors in $G_d(v)$. Every paired product is at least
$w(w+p-2)$, and every unpaired factor exceeds its square root. Moreover,
\[
 \kappa-\sqrt{w(w+p-2)}
 =\frac{w}{\sqrt{w(w+p-1)}+\sqrt{w(w+p-2)}}<\frac12.
\]
Consequently
\begin{equation}\label{eq:var-continuous-constant}
 G_d\left(\sqrt{\kappa^2+\alpha^2}-\frac12\right)
 >\frac{2(\kappa-1/2)^p}{p!}.
\end{equation}
The ground-row count is therefore larger than the same quantity.
Its ratio to $W_d(K)$ exceeds
\[
 R_d(K):=\frac{dB_d}{K}\,\beta^{p/2}
       \left(1-\frac{1}{2K\sqrt\beta}\right)^p.
\]
With $b=1/(2\sqrt\beta)$, direct differentiation gives
\[
 \frac{d}{dK}\log R_d(K)=\frac{db-K}{K(K-b)}.
\]
Thus $R_d$ has a single maximum on $(b,\infty)$ and no interior
minimum. It suffices to bound it at the two ends of any closed interval
contained in this domain.

We use two simple monotonicity observations. First,
\[
 \beta^{p/2}
 =\left(\frac{p-1}{p+1}\right)^{p/2}
 =\exp\left[-\frac{\operatorname{arctanh}(1/p)}{1/p}\right]
\]
increases with $p>1$, because $\operatorname{arctanh}x/x$ increases
with $x\in(0,1)$. In particular $\beta^{p/2}\ge9/25$ for $d\ge5$.
Second, the functions
\[
 \left(1-\frac{3}{2\sqrt{p^2-1}}\right)^p,
 \qquad
 \left(1-\frac{20}{27(p+1)\sqrt{p-1}}\right)^p
\]
increase with $p\ge5$. To check both assertions uniformly, put $x=1/p$
and write the respective deficits as
\[
 \delta_1(x)=\frac{3x}{2\sqrt{1-x^2}},\qquad
 \delta_2(x)=\frac{20x^{3/2}}{27(1+x)\sqrt{1-x}}.
\]
Each $\delta_i$ and $\delta_i/x$ increases on $0<x\le1/5$.
For the second quotient, its logarithmic derivative is
\[
 \frac{1-x+2x^2}{2x(1-x)(1+x)}>0;
\]
the first quotient increases because its denominator decreases.
Since $-\log(1-v)/v$ also increases for
$0<v<1$, the identity
\[
 -\frac{\log(1-\delta_i(x))}{x}
 =\frac{\delta_i(x)}x\,
       \frac{-\log(1-\delta_i(x))}{\delta_i(x)}
\]
proves the claimed monotonicities.

Suppose now that $d\ge6$. The Gamma recurrence and Wallis estimate
proved in Lemma~\ref{lem:var-uniform-endpoints} give
$B_d>\sqrt{2\pi d}$. At $K=d/3$, the preceding monotonicities imply
\[
 R_d(d/3)>18\cdot\frac9{25}\left(\frac{69}{100}\right)^5>1.
\]
Here $B_d>6$, and the last factor in $R_d(d/3)$ is minimized at
$d=6$, where its base is $1-3/(2\sqrt{24})>69/100$.
The last rational comparison is
$162\cdot69^5-25\cdot100^5=3373078538>0$.

At the upper endpoint $K_*=(27/40)d^{3/2}$, direct substitution gives
\[
 R_6(K_*(6))=\frac{1024}{729}\left(\frac{76}{81}\right)^5
 >\frac75\left(\frac{15}{16}\right)^5>1.
\]
For $d\ge7$, the first monotonicity gives
$\beta^{p/2}\ge(5/7)^3=125/343$.
The second monotonicity, $B_d/\sqrt d>\sqrt{2\pi}>5/2$, and
$\sqrt5>20/9$ therefore give
\[
 R_d(K_*(d))>
 \frac{12500}{9261}\left(\frac{20}{21}\right)^6
 =\frac{800000000000}{794280046581}>1.
\]
Unimodality now proves the required estimate on
$d/3\le K\le K_*(d)$ for every $d\ge6$.

For $K\le d/3$, the zero mode suffices. Indeed,
\[
 W_d(d/3)=\frac{2(d/3)^d}{B_d\,d!}
 <\frac2{B_d}\left(\frac e3\right)^d<1.
\]
The elementary factorial bound $d!>(d/e)^d$ follows by integrating
$\log t$ from $1$ to $d$; also, $e<3$.

It remains to treat $d=5$. Here
\[
 W_5(K)=\frac{2K^5}{225\pi},\qquad
 \widehat C_5(L)=\frac{(L+2)^2(L+1)(L+3)}{12}.
\]
Use this polynomial to define $\widehat C_5$ also at real arguments.
Put $k_\ell^2=(5/3)\ell(\ell+3)$. For a real variable $s>0$, the
ratio $\widehat C_5(s-1)/W_5(k_s)$ is a positive constant times
\[
 \frac{(s+1)^2(s+2)}{s^{3/2}(s+3)^{5/2}}.
\]
It decreases strictly, since twice its logarithmic derivative is
\[
 \frac4{s+1}+\frac2{s+2}-\frac3s-\frac5{s+3}
 =-\frac{2s^3+5s^2+13s+18}{s(s+1)(s+2)(s+3)}<0.
\]
At $s=4$, using $\pi>3$ and $\sqrt{140/3}<7$ gives
\[
 W_5(k_4)<\frac{274400}{6075}<50=\widehat C_5(3).
\]
Thus $\widehat C_5(\ell-1)>W_5(k_\ell)$ simultaneously for all integers
$1\le\ell\le4$. Set $k_0=0$. On
$k_{\ell-1}\le K\le k_\ell$, the degrees $0,\ldots,\ell-1$
give a strict count at least $\widehat C_5(\ell-1)>W_5(k_\ell)\ge W_5(K)$.
These intervals cover $0<K\le k_4$. Finally
$k_4<K_*(5)<8$, and
\[
 W_5(K_*(5))<W_5(8)<\frac{65536}{675}<105=\widehat C_5(4).
\]
The degree-four row therefore covers the remaining interval.

We have proved the stronger unit-ball first-row count in every
dimension $d\ge5$. Lemma~\ref{lem:var-radial} transfers it to every
shell; its Weyl term is smaller by the factor $1-\rho^d$.
\end{proof}

\subsection{The initial three-dimensional intervals}
\label{app:proof-d3-lem-initial}
\begin{proof}[Proof of Lemma~\ref{d3:lem:initial}]
Put $Q_0=0$. On $Q_{j-1}\le K^2\le Q_j$, the strict trials give
at least $j^2$ eigenvalues. The derivative of
$(t^2+t+c)^3/t^4$ has the sign of $2t^2-t-4c$.
For $j=1,2,3$ and $c=1$, the function has its maximum at $j=1$ or $3$,
and $Q_1^3<Q_3^3/3^4<130$. For $j=4,5$ and $c=4$,
$Q_j^3/j^4$ increases, starting above $140$ at $j=4$.
Thus the first five intervals are covered by $\pi>3$ and
\[
 \W_3(\sqrt{Q_5},0)<\frac{2\cdot48\cdot7}{27}
 =\frac{224}{9}<25.
\]
The sixth interval follows from $\pi>157/50$ and
$\sqrt{Q_6}<8$, since
\[
 \W_3(\sqrt{Q_6},0)<\frac{100}{1413}\frac{127}{2}\,8
 =\frac{50800}{1413}<36.
\]
Finally, if $\rho\ge1/2$, then $\sqrt{Q_7}<46/5$ gives
\[
 \frac78\W_3(\sqrt{Q_7},0)
 <\frac78\frac{100}{1413}\,84\frac{46}{5}
 =\frac{22540}{471}<49.
\]
\end{proof}

\section{Ball estimates in dimensions three and four}
\label{app:low-ball-details}
Both calculations use the same first-period polynomial bound. For
$p=2$, assume $K\ge44/5$; for $p=3$, assume $K\ge12$.
Set $y_-(s)=1-s+s^2/30$, $y_+(s)=1-s+s^2/25$, and
\[
 \varepsilon=K^{-1},\quad a_K=\frac12(15\pi/(4K))^{2/3},\quad
 q=5^{-2/3},\quad \frac{41}{120}<q<\frac{171}{500}.
\]
Put $f=\Phi_p\circ\xi_K$ and, above the clipping point, write
\[
 \phi_p(y)=K^{-p}\Phi_p(Ky)=
 \begin{cases}
 (y-\varepsilon/2)^2,&p=2,\\
 y^3/3-\varepsilon y^2/2+\varepsilon^2y/6,&p=3.
 \end{cases}
\]
Lemma~\ref{lem:inverse-action-seven-tenths} gives, for $0\le t\le5/4$,
\[
 y_-(s)\le\xi_K(t)/K\le y_+(s),\qquad
 s=a_K(4t/5)^{2/3}.
\]
All these inverse values exceed the clipping point: for $p=2$,
$a_K<7/10$ and $K(1-a_K)>1/2$; for $p=3$, $a_K<1/2$ and
$K(1-a_K)>1$.

The dimension-specific checks below establish, with
$u=\arccos(\xi_K(5/4)/K)$ and $S=\sin u-u\cos u=5\pi/(4K)$,
that $u<6/5$ and $f''(5/4)/(-f'(5/4))<2$.
Taylor's bounds then give $5u^3/18<S\le u^3/3$, so
\[
 u<\frac{16}{15}\sqrt{2a_K},\qquad
 \frac{\pi}{K\sqrt{2a_K}}=\frac{8a_K}{15}.
\]
The terminal coefficient in \eqref{eq:jp-terminal} is zero for $p=2$
and less than $1/2$ for $p=3$, since
$\pi/[6\arccos(1/K)]<1/2$ when $K\ge12$.
Lemma~\ref{lem:jp-tail}, together with
$1/12-1/(36\sqrt3)>1/15$ and monotonicity of $\phi_p'$, therefore gives
\begin{equation}\label{eq:d4-newtail}
 \frac{\mathcal R_f(5/4)}{K^p}
 >\frac{a_K}{30}\phi_p'(y_-(a_K))-\delta_p,
 \qquad \delta_2=0,\quad\delta_3=\frac1{24K^3}.
\end{equation}
Integrating the upper polynomial over the first sampling period gives
\begin{equation}\label{d3:eq:ballpoly-first}
 \begin{split}
 \frac{E_B(\Phi_p)}{K^p}\ge R_p(K):={}&
 \phi_p(y_-(q a_K))+\tfrac12\phi_p(y_-(a_K))
 -\frac54\sum_{n=0}^{2p}\frac{c_n a_K^n}{1+2n/3}
 +\frac{a_K}{30}\phi_p'(y_-(a_K))-\delta_p,\\
 &\sum_{n=0}^{2p}c_ns^n=\phi_p(y_+(s)).
 \end{split}
\end{equation}
The inequality is strict for $p=3$. In the coefficient estimates below
we use $157/50<\pi<22/7$.

\subsection{The quadratic ball surplus}
\label{app:proof-d3-lem-ballpoly}
\begin{proof}[Proof of Lemma~\ref{d3:lem:ballpoly}]
Since $\sin(6/5)-(6/5)\cos(6/5)>9/20>5\pi/(4K)$,
one has $u<6/5$.
Taylor's bounds give $\cos u>1/3$, $S/u<1/2$, and
$\sin u/u>3/4$. As $K>8$, the derivative formulas preceding
\eqref{eq:jp-terminal} give
\[
 \frac{f''(5/4)}{-f'(5/4)}
 =\frac{4S}{u(5\cos u-2S/\pi)}
       +\frac{4S}{5u^2\sin u}
 <\frac{96}{65}+\frac{16}{45}<2.
\]
This verifies the hypotheses used in \eqref{d3:eq:ballpoly-first}.
For the full integral, the action moments give
\[
 J_2(K)=\W_3(K,0)-K^2/8+
        \int_0^{1/2}(1-2z)G_K(z)\,dz.
\]
Convexity gives $G_K(z)\ge K/\pi-z/2$, and hence
\[
 J_2(K)-\W_3(K,0)\ge-K^2/8+K/(4\pi)-1/48.
\]
Add this to \eqref{d3:eq:ballpoly-first}. With
$v=K^{-1/3}$ and $a_K=\eta v^2$, cubing gives
\[
 \eta=(225\pi^2/128)^{1/3},\qquad
 \frac{31}{12}<\eta<\frac{259}{100}.
\]
Consequently
\[
 (2q-17/30)\eta<\frac{38}{125},\qquad
 (16q^2/15-47/420)\eta^2>\frac{21}{250},\qquad
 (q-17/60)\eta>\frac3{20}.
\]
Using $\varepsilon=v^3$, $a_K=\eta v^2$, $q^3=1/25$ and
$\eta^3=225\pi^2/128$, expansion therefore gives the lower bound
\[
 \frac18-\frac{38}{125}v^2
 +\left(-\frac14+\frac1{4\pi}\right)v^3
 +\frac{21}{250}v^4+\frac3{20}v^5.
\]
The coefficients of the discarded terms of degrees six, seven and eight are positive:
\[
 \frac1{24}-\frac{\pi^2}{1280}>0,\qquad
 (1/210-q^2/30)\eta^2>0,\qquad
 (q^4/900+1/99000)\eta^4>0.
\]
Finally $v^5\ge3v^3/4-v^2/4$, by convexity of $v^3$ at $v=1/2$,
and $-11/80+1/(4\pi)>-29/500$. These give precisely $B(v)$.
\end{proof}

\subsection{The cubic ball surplus}
\label{app:d4-ball-proof}
\begin{proof}[Proof of Lemma~\ref{lem:d4ball}]
The action relation and $S\ge u^3/3-u^5/30$ give
$u<26/25<\pi/3$, hence $\cos u>1/2$ and $\sin u/u>4/5$.
Since $Ku^3\ge15\pi/4$, $K\ge12$, and $\pi^2<10$,
\[
 \frac\pi{Ku}<\frac4{15},\qquad
 \frac{\Phi_3''(z)}{\Phi_3'(z)}\le\frac2{z-1}\quad(z>1).
\]
Consequently
\[
 \frac{f''(5/4)}{-f'(5/4)}
 \le\frac{2\pi}{Ku(\cos u-1/K)}
                  +\frac\pi{Ku^2\sin u}
 <\frac{32}{25}+\frac13<2.
\]
The common bound \eqref{d3:eq:ballpoly-first} now applies strictly.
Set $v=(12/K)^{1/3}$ and $R(v)=R_3(K)$. Here
$\varepsilon=v^3/12$ and $a_K=\eta v^2$, where
\[
 \eta=(25\pi^2/2048)^{1/3},\qquad
 \frac{49}{100}<\eta<\frac12.
\]
After expansion and reduction by $q^3=1/25$, the coefficients of
$v^2$ and $v^5$ are, respectively,
\[
 \eta(17/60-q)>-3/100,\qquad
 \eta(q-17/60)/12>1/500.
\]
The coefficients of degrees eight and twelve are positive; those
of degrees six, seven, nine, ten and eleven are negative.
The coefficient of $v^4$ plus these five negative coefficients equals
\[
 \begin{split}
 &\frac{11}{41472}
 +\eta^2\left(\frac{25633q^2}{25920}-\frac{3719}{36288}\right)
 -\frac{443\eta^3}{54000}\\
 &\qquad-\eta^4\left(\frac1{2376000}+\frac{q}{540000}\right)
 -\eta^5\left(\frac1{17550}+\frac{q^2}{22500}\right)
 >\frac1{500}.
 \end{split}
\]
Indeed, the $\eta^2$ term exceeds $3/1000$, and the remaining terms
together exceed $-1/1000$, by the displayed bounds on $\eta$ and $q$.
For $0\le v\le1$, each negative term of degree at least six is
bounded below by its coefficient times $v^4$. Dropping the positive
terms of degrees eight and twelve therefore gives
\[
 R(v)\ge\frac1{12}-\frac3{100}v^2-\frac1{96}v^3
                         +\frac1{500}v^4+\frac1{500}v^5.
\]
The identity
$2v^5-3v^4+v^2=v^2(1-v)^2(2v+1)\ge0$
reduces this to
\[
 R(v)\ge P(v):=\frac1{12}-\frac{31}{1000}v^2
                         -\frac1{96}v^3+\frac1{200}v^4.
\]
Since $P'(v)/v\le-21/500<0$ on $(0,1]$,
\[
 R(v)\ge P(1)=\frac3{64}+\frac1{24000}.
\]
This proves the assertion.
\end{proof}

\section{Planar estimates}\label{app:planar-estimates}
\begin{lemma}[Planar variational ranges]\label{d2:lem:variational}
Let $0\le\rho<1$, $K>0$, and $T=(1-\rho)K/\pi$.
If $K\le8$, or if $K\ge8$ and $T\le3/2$, then
\[
 N^<_{\Omega_{\rho,1}}(K^2)>\frac{1-\rho^2}{4}K^2.
\]
In the second range, the difference is greater than $K/120$.
The low-frequency assertion includes the disk $\Omega_{0,1}=B_1$.
\end{lemma}
\begin{proof}
Write $E=K^2$, $W=(1-\rho^2)E/4$, and
$N^<=N^<_{\Omega_{\rho,1}}(K^2)$. Denote the radial levels in angular order
$m\ge0$ by $\lambda_{m,j}(\rho)$. We use Lemma~\ref{lem:var-radial}
and the ball three-quarter count established after \eqref{eq:robinphase}:
$n_m^<(K^2)\ge\lfloor G_K(m)+3/4\rfloor$ when $\rho=0$.
Positive angular orders have multiplicity two; order zero has multiplicity one.

For $0\le r<1$, put $h=1-r$. The trial
$\cos(\pi(s-r)/h)+4h/[\pi^2(1+r)]$ is orthogonal to constants
for the measure $s\,ds$ on $(r,1)$. Its squared norm is
$h(1+r)/4-8h^3/[\pi^4(1+r)]$, and its derivative energy is
$\pi^2(1+r)/(4h)$. Hence
\begin{equation}\label{d2:eq:cosine-trial}
 \lambda_{0,1}(r)\le Q(r):=
 \frac{\pi^2}{(1-r)^2}
 \left(1-\frac{32(1-r)^2}{\pi^4(1+r)^2}\right)^{-1}.
\end{equation}
We use $157/50<\pi<22/7$ throughout.

\emph{Low frequencies.}
For $1\le m\le6$, the ball trial $s^m(1-ms^2/(m+2))$ gives
\[
 \lambda_{m,0}(\rho)\le\lambda_{m,0}(0)
 <q_m:=\frac{2m(m+1)(m+3)(3m+4)}{(m+2)(5m+6)}
 <b_m:=\frac{m(12m+25)}{10}.
\]
The first strict inequality holds because the polynomial is not an
eigenfunction; the second follows from
\[
 b_m-q_m=\frac{m(-3m^2+44m+60)}{10(m+2)(5m+6)}>0,
 \qquad -3m^2+44m+60\ge26m+60.
\]
For the excited levels, \eqref{d2:eq:cosine-trial} gives
\[
 Q(0)<15,\qquad Q(3/8)<109/4,\qquad Q(1/2)<83/2.
\]
Strict convexity gives $G_K(m)>K/\pi-m/2$ for $0<m\le K$.
The phase bound at $K=7\pi/4$ and $K=9\pi/4$ therefore gives
\[
 \lambda_{1,1}(0)<(7\pi/4)^2<31,\qquad
 \lambda_{2,1}(0)<(9\pi/4)^2<51.
\]
The elementary comparison
$\lambda_{m,1}(r)\le\lambda_{0,1}(r)+m^2/r^2$ and the radial
endpoint principle therefore give
$\lambda_{0,1}<109/4$, $\lambda_{1,1}<35$, and
$\lambda_{2,1}<223/4$ throughout $\rho\le3/8$; they give
$\lambda_{0,1}<83/2$ and $\lambda_{1,1}<91/2$ throughout
$\rho\le1/2$. For the order-two bound, for example,
$109/4+256/9<223/4$.
Put $\chi_a(E)=\mathbf1_{[a,\infty)}(E)$ and
$G(E)=1+2\sum_{m=1}^6\chi_{b_m}(E)$.
For $0<E\le64$, the preceding estimates give $N^<\ge G+R$ and
$W\le CE$ in the following three radius intervals. The last column
is the minimum left limit of $G+R-CE$ at its breakpoints and at $E=64$.
\[
\begin{array}{c|c|c|c}
 \rho&R&C&\text{margin}\\ \hline
 {}[0,3/8]&\chi_{109/4}+2\chi_{35}+2\chi_{223/4}&1/4&1/16\\
 {}[3/8,1/2]&\chi_{83/2}+2\chi_{91/2}&55/256&43/512\\
 {}[1/2,1)&0&3/16&7/80
\end{array}
\]
Between breakpoints the difference decreases linearly; at a breakpoint
the entering level is already counted. The positive margins prove
the low-frequency assertion.

\emph{Small widths.}
Let $K\ge8$ and $T\le3/2$. Then $\rho>2/5$ and
$z=(1-\rho)/(1+\rho)<3/7$. The constant radial trial gives
$\lambda_{m,0}(\rho)<\beta_\rho m^2$ for $m\ge1$, where the
logarithmic series bounds
\[
 \beta_\rho=\frac{2\log(1/\rho)}{1-\rho^2}
 \le\frac4{(1+\rho)^2}\left(1+\frac{z^2}{3(1-z^2)}\right)
 \le\frac{43}{10(1+\rho)^2}
 <\left(\frac{21}{10(1+\rho)}\right)^2.
\]
The ground row thus contributes more than $(20/21)K(1+\rho)-1$.
If $T\le12/11$, it follows that
\[
 \frac{N^<-W}{K}
 >\frac75\left(\frac{20}{21}-\frac67\right)-\frac18
 =\frac1{120}.
\]
For $12/11\le T\le3/2$, use $\pi^4>97$ to obtain
$32z^2/\pi^4<288/4753<1/16$ in \eqref{d2:eq:cosine-trial}.
Consequently
\[
 \lambda_{m,1}(\rho)<\frac{16\pi^2}{15(1-\rho)^2}
                         +\frac{m^2}{\rho^2}.
\]
Adding this excited row and using $K\ge8$ gives
\[
 \frac{N^<-W}{K}>
 H(T,\rho):=(1+\rho)\left(\frac{20}{21}-\frac{\pi T}4\right)
       +2\rho\sqrt{1-\frac{16}{15T^2}}-\frac14.
\]
This function increases with $\rho$, since
\[
 \partial_\rho H>
 \frac{20}{21}-\frac{33}{28}+\frac35=\frac{157}{420}>0.
\]
It is concave in $T$, so it suffices to evaluate $H(T,2/5)$ at
$T=12/11$ and $T=3/2$. Using
$\sqrt{14/135}>3/10$ and $\sqrt{71/135}>18/25$ gives
\[
 H(12/11,2/5)>\frac{37}{300},\qquad
 H(3/2,2/5)>-\frac{19}{60}+\frac{72}{125}-\frac14
 =\frac7{750}>\frac1{120}.
\]
Both width ranges therefore give the stated surplus $K/120$.
\end{proof}

\subsection{Semicircle sampling}

Let $e(T)$ be the semicircle sampling error defined in
Lemma~\ref{d2:lem:semicircle}. Including a terminal zero sample
leaves its value unchanged.

\begin{lemma}[Semicircle sampling at width endpoints]
\label{d2:lem:quarter-semicircle}
For every integer $n\ge3$, with $T=n+1/4$,
\begin{equation}\label{d2:eq:quarter-semicircle}
 e(T)>\frac14-\frac5{12\sqrt{2T}}-\frac1{32T^2}.
\end{equation}
\end{lemma}
\begin{proof}
Put $f(t)=\sqrt{1-t^2/T^2}$ on $[0,T]$.
Let $R$ be the integral of $f$ minus its composite trapezoidal sum
on $[1/4,T]$, with unit intervals. Since $f(T)=0$,
\[
 e(T)=\frac12f(1/4)-\int_0^{1/4}f(t)\,dt-R
 >\frac14-\frac1{32T^2}-R.
\]
The function $-f''(t)=T(T^2-t^2)^{-3/2}$ is convex.
On each unit interval $[a,a+1]\subset[1/4,T-1]$,
\[
 \int_a^{a+1}f(t)\,dt-\frac{f(a)+f(a+1)}2
 =\int_0^1\frac{s(1-s)}2(-f''(a+s))\,ds
 \le\frac{f'(a)-f'(a+1)}{12}.
\]
Indeed, after pairing $s$ with $1-s$, the curvature average
decreases on $[0,1/2]$ while $s(1-s)$ increases; the reversed
Chebyshev inequality gives the factor $1/12$.
Summing bounds the error before the last interval by $-f'(T-1)/12$.
Put $v=\sqrt{1-1/(2T)}$. On the last interval, substitute $u=T-t$
and use $\sqrt{1-u/(2T)}\le1-u/(4T)$ to obtain
\[
 R\le\sqrt{\frac2T}\left(
 \frac23-\frac1{10T}-\frac v2+\frac{1-1/T}{24v}\right)
 =\sqrt{\frac2T}\left(
 \frac5{24}+\frac{(v-1)(4v-1)(6v-5)}{120v}\right)
 \le\frac5{24}\sqrt{\frac2T}.
\]
Here $T\ge13/4$, so $5/6<v<1$. This proves the claim.
\end{proof}

\subsection{Planar endpoint inequalities}\label{d2:app:endpoints}

We use the elementary bound from Lemma~\ref{lem:inverse-action-seven-tenths}:
\begin{equation}\label{d2:eq:airy}
 \xi_K(t)\ge K-c(t)K^{1/3},\qquad
 c(t)=\frac{(3\pi t)^{2/3}}2.
\end{equation}
If its right side is at least $\mu$, the shell sample is on the
outer branch and obeys the same bound.

\begin{lemma}[Derivative at fixed width]\label{d2:lem:width-slope}
Fix $T>0$ and $0<t<T$, and vary $K\ge\pi T$ with $\mu=K-\pi T$.
On the inner branch,
\[
 \partial_Kx(t)\ge\sqrt{1-(t/T)^2}.
\]
On the outer branch, writing $x(t)=K\cos u$ gives
\[
 \partial_Kx(t)=\frac{\sin u}{u},\qquad
 \sin u-u\cos u=\frac{\pi t}{K},\qquad 0<u<\frac\pi2.
\]
The derivatives agree at the branch junction.
\end{lemma}
\begin{proof}
On the inner branch put $a=\arccos(x/K)$, $b=\arccos(x/\mu)$,
and let $\mathbb E$ denote the uniform average on $[b,a]$.
Differentiation and the action equation give
\[
 p:=\partial_Kx=\mathbb E\cos\theta,\qquad
 q:=\frac tT=\frac{\mathbb E\tan^2\theta}
                      {\mathbb E(\tan\theta\sec\theta)}.
\]
Cauchy--Schwarz and Jensen imply
\[
 q\ge\sqrt{1-\frac1{\mathbb E\sec^2\theta}}
       \ge\sqrt{1-p^2},\qquad
 \mathbb E\sec^2\theta\ge(\mathbb E\cos\theta)^{-2}.
\]
Thus $p\ge\sqrt{1-q^2}$. Differentiating $G_K(x)=t$ gives the
outer formula, which agrees with the inner one at $b=0$.
\end{proof}

\begin{proof}[Proof of the endpoint estimates \eqref{d2:eq:endpoints}]
Fix $K\ge8$ and write $W(T)=\W_2(K,K-\pi T)$.
Throughout, $157/50<\pi<22/7$. Put $t_j=j+1/4$ and $c_j=c(t_j)$.
Cubing gives
\[
 c_0<\frac{443}{500},\qquad c_1<\frac{259}{100},\qquad
 2(c_0+c_1)<7.
\]

\emph{The starting width $T=3/2$.}
Lemma~\ref{d2:lem:width-slope} gives
$\partial_Kx(1/4)\ge2/\pi$ and
$\partial_Kx(5/4)\ge\sqrt{11}/6$ on either branch: outer slopes
exceed $2/\pi$, while the inner bounds are $\sqrt{35}/6$ and
$\sqrt{11}/6$. Hence
\[
 \partial_K \Sigma_2(3/2)
 \ge\frac4\pi+\frac{\sqrt{11}}3-\frac{3\pi}4
 >\frac{14}{11}+\frac{11}{10}-\frac{33}{14}=\frac6{385}>0.
\]
At $K=6$, both samples are outer by \eqref{d2:eq:airy}, since
$6<48\pi/25$. As $W<9$ and $6^{1/3}<11/6$,
\[
 \Sigma_2(3/2)>24-7\,6^{1/3}-2-W
 >24-\frac{77}{6}-2-9=\frac16.
\]
Monotonicity proves the required range $K\ge8$.

\emph{Two rows.}
At $T=9/4$, both samples are outer for $8\le K\le14$, since
\eqref{d2:eq:airy} guarantees this when $K\le162\pi/25$.
With $y=K^{1/3}$, the coefficient bound above gives
\[
 \Sigma_2(9/4)>\frac{13}{28}y^3-7y+\frac{1047}{100}>0.
\]
The cubic is positive on $y\ge0$ because
$(1047/100)^2>38416/351$.
For $K\ge14$, Lemma~\ref{d2:lem:width-slope} gives
\[
 \partial_Kx(1/4)>\frac{47}{50},\qquad
 \partial_Kx(5/4)>\frac{83}{100}
\]
on either branch. For the inner branch, square and compare with
$80/81$ and $56/81$. For the outer branch, the corresponding
angles are less than $3/5$ and $1$, since
\[
 \sin u-u\cos u\ge\frac{u^3}3-\frac{u^5}{30},\qquad
 \frac\pi{56}<\frac{(3/5)^3}3-\frac{(3/5)^5}{30},\quad
 \frac{5\pi}{56}<\frac3{10}.
\]
The bound $\sin u/u>1-u^2/6$ gives the two slopes. Therefore
\[
 \partial_K\Sigma_2(9/4)>
 2\left(\frac{47}{50}+\frac{83}{100}\right)-\frac{9\pi}8
 >\frac3{700}>0,
\]
which proves the remaining range.

\emph{Three or more rows.}
Let $T=n+1/4$ with $n\ge3$, and put $\rho=1-\pi T/K$.
We first note the uniform bound $e(T)>7/80$. For $n=3$,
\[
 e(13/4)=\frac{2\sqrt{42}+12+2\sqrt{22}}{13}-\frac{13\pi}{16}
 >\frac{1601}{18200}>\frac7{80},
\]
using $\sqrt{42}>162/25$, $\sqrt{22}>469/100$, and $\pi<22/7$.
For $n\ge4$, Lemma~\ref{d2:lem:quarter-semicircle} gives
\[
 2\bigl(1/4-e(T)\bigr)
 <\frac5{6\sqrt{2T}}+\frac1{16T^2}
 \le\frac5{3\sqrt{34}}+\frac1{289}
 <\frac{10}{33}+\frac1{289}<\frac{13}{40},
\]
where $\sqrt{34}>11/2$. Thus Lemma~\ref{d2:lem:semicircle} implies
\[
 \frac{\Sigma_n(T)}K>
 \frac12-\frac{13}{40}\sqrt\rho-\beta(T)(1-\rho)
       -\gamma(T)(1-\rho)^{2/3}-\delta(T)\sqrt{1-\rho},
\]
where
\[
 \beta(T)=\frac1\pi\left(1-\frac5{12T}\right),\qquad
 \gamma(T)=\frac8{25(\pi T)^{2/3}},\qquad \delta(T)=\frac1{4T}.
\]
At $\mu=0$, the same estimate follows from
\eqref{eq:onepoint-ball-linear}, since $\pi>2$.
For each fixed $\rho$, this lower bound increases with $T$:
$\gamma$ decreases, and, for $z=1-\rho$,
\[
 -\beta(T)z-\delta(T)\sqrt z
 =-\frac z\pi+\frac1T\left(\frac{5z}{12\pi}-\frac{\sqrt z}{4}\right),
\]
whose coefficient of $1/T$ is nonpositive.
It suffices to take $T=13/4$, where
\[
 \beta(T)<\frac5{18},\qquad
 \gamma(T)<\frac3{44},\qquad \delta(T)=\frac1{13}.
\]
These bounds use $\pi>157/50$ and $(352/75)^3<(2041/200)^2$.
Concavity gives $(1-\rho)^{2/3}\le1-2\rho/3$ and
$\sqrt{1-\rho}\le1-\rho/2$. Completing the square in $\sqrt\rho$
therefore gives
\[
 \frac{\Sigma_n(T)}K>
 \frac12-\frac5{18}-\frac3{44}-\frac1{13}
 -\frac{(13/40)^2}{4(5/18+1/22+1/26)}
 =\frac{15759239}{3834230400}>\frac1{250}.
\]

\emph{The ball endpoint.}
Here $\Sigma_n(K/\pi)=2E_{1/4}(\xi_K)-n>K/500$ by
\eqref{d2:eq:disk-surplus}. Terminal zero samples contribute nothing to
the inverse sum, so this includes quarter-height ball endpoints.
\end{proof}

\section*{Funding}
This work was supported by the Department of Education of Guangdong
Province (Grant No.~2025ZDJS099) and Nanfang College, Guangzhou
(Grant No.~2026BQ010).

\section*{Use of AI tools}
During the preparation of this work, the author used ChatGPT (OpenAI)
and Claude (Anthropic) to assist with mathematical exploration and proof development,
checking intermediate arguments and calculations, literature searches,
and drafting and revising the text and \LaTeX{} source.
The author takes full responsibility for the results, proofs,
references, and final manuscript.

\begingroup\small\raggedright

\endgroup
\end{document}